%% file: genlin-v5.tex
\documentclass[reqno,12pt]{amsart}
\usepackage{amsmath,amssymb,epsfig,color,slashbox,url}

\numberwithin{equation}{section}

\newtheorem{theorem}{Theorem}[section]
\newtheorem{conjecture}[theorem]{Conjecture}
\newtheorem{lemma}[theorem]{Lemma}
\newtheorem{proposition}[theorem]{Proposition}
\newtheorem{definition}[theorem]{Definition}
\newtheorem{corollary}[theorem]{Corollary}

\newtheorem{example}[theorem]{Example}
\newtheorem{remark}[theorem]{Remark}

\newcommand{\rank}{\operatorname{rk}}

\newcommand{\vanish}[1]{}

\newcommand{\walkw}{\vec{w}}

\begin{document}

\title{Labeling regions in deformations of graphical arrangements}

\author[G\'abor Hetyei]{G\'abor Hetyei}

\address{Department of Mathematics and Statistics,
  UNC Charlotte, Charlotte NC 28223-0001.
WWW: \tt http://webpages.uncc.edu/ghetyei/.}



\date{\today}

\begin{abstract}
Combining Carver's variant of the Farkas' lemma with the Flow Decomposition
Theorem we show that the regions of any
deformation of a graphical arrangement may be bijectively labeled with a
set of weighted digraphs containing directed cycles of negative weight
only. Bounded regions correspond to strongly connected digraphs. The
study of the resulting labelings allows us to add the omitted details in
Stanley's 
proof on the injectivity of the Pak-Stanley labeling of the regions of
the extended Shi arrangement, to generalize the ceiling diagrams in the
deleted Shi and Ish arrangements studied by Armstrong and Rhoades
and to introduce a new labeling of the
regions in the Fuss-Catalan arrangement. We also point out that
Athanasiadis-Linusson labelings may be used to directly count regions
in a class of arrangements properly containing the extended Shi
arrangement and the Fuss-Catalan arrangement.
\end{abstract}
\maketitle

\section*{Introduction}

Counting regions of a hyperplane arrangement is most often performed
by computing its characteristic polynomial and by using Zaslavsky's
formula~\cite{Zaslavsky}. This approach inspires 
combinatorial labelings of the regions in several important
special cases: when all coefficients in all equations are integers, the
finite field method~\cite[Theorem 2.2]{Athanasiadis-charpol} (explained
in detail by Stanley in~\cite[Lecture 5]{Stanley-PC} and in~\cite[Section
  3.11.4]{Stanley-EC1}) yields interesting models. For deformations of graphical
(or affinographic) arrangements, in which all equations are of the form
$x_i-x_j=c$, Whitney's formula~\cite{Whitney} and the gain graph
method~\cite{Berthome-etc,Forge-Zaslavsky,Zaslavsky-perpendicular} open
a gateway to combinatorics. For certain deformations of the braid
arrangement, Gessel's formula on the generating function of
labeled binary trees counted according to ascents and descents along
left or right edges (shown in~\cite{Drake,Kalikow}) has specializations
that may be used to count the regions~\cite{Bernardi,Gessel-Tewari}.   

A fundamentally different approach is to consider regions as sets defined
by inequalities. One of the earliest examples of this approach is the work of
Shi~\cite{Shi}, further important examples include 
Pak-Stanley labeling~\cite{Stanley-tinv}, the Athanasiadis-Linusson
labeling~\cite{Athanasiadis-Linusson} of the regions of the extended Shi
arrangement, Bernardi's annotated $m$-sketches labeling the regions of
the $a$-Catalan arrangement, the ceiling diagrams in the work of Armstrong and
Rhoades~\cite{Armstrong-Rhoades} and the work of Hopkins and
Perkinson~\cite{Hopkins-bigraphical} counting the regions in bigraphical
arrangements. 

A third approach is emerging in the work of Lam and
Postnikov~\cite{Lam-Postnikov-1, Lam-Postnikov-2} who use
chamber counts in a hyperplane arrangement to compute the volumes of
alcoved polytopes.

The present work belongs to the second, inequality based approach. The
main purpose of this paper is to show that combining the Carver-Farkas lemma with the Flow Decomposition Theorem yields an automatic
way to encode the regions with weighted digraphs in any deformation of a
graphical arrangement in such a way that 
regions correspond exactly to the ones in which the weight of the
directed cycles is negative. This approach generalizes the
representation of the regions of the Linial arrangement with semiacyclic
tournaments and associates  bounded regions to strongly connected
digraphs. Using this approach one may easily fill in the omitted details
in Stanley's proof of the injectivity of the Pak-Stanley labeling in the 
extended Shi arrangement, provide an alternative proof of Stanley's formula for
exponential arrangements, describe the ceiling hyperplanes in a
generalized Ish arrangement, and find a new labeling of the regions of the
$a$-Catalan arrangement. This approach also generalizes the partial
orientations introduced by Hopkins and
Perkinson~\cite{Hopkins-bigraphical} to label the regions of a
bigraphical arrangement (they used the same Carver-Farkas lemma)
and naturally explains why sleek posets 
represent the regions of the Linial arrangement and why bounded regions
of an interval order arrangement are in bijection with posets whose
incomparability graph is connected. In an effort to keep the paper's
size manageable, focus is on deformations of the braid arrangement, but
the key ideas apply to the deformations of all graphical arrangements. 

The paper is structured as follows. Section~\ref{sec:labeling} contains
the key bijection between regions of a 
graphical arrangement and weighted digraphs. The weighted digraphs
encoding the regions are described in Theorem~\ref{thm:mincost}, and it
is stated in Theorem~\ref{thm:br} that in this bijection bounded regions
correspond to strongly connected weighted digraphs.  
Section~\ref{sec:og} treats
deformations of the braid 
arrangement which are {\em sparse} in the sense that at most two
hyperplanes are associated to each edge. The Linial arrangement and
interval order arrangement are important examples, but the bigraphical
arrangements~\cite{Hopkins-bigraphical} are also related: they are also
sparse in the above sense, however their underlying graph can be any
simple graph, not only the complete graph. A key new idea is the
notion of {\em gains}, 
using the fact that banning nonnegative directed 
cycles is equivalent to restricting our attention to weighted digraphs
in which negative weights represent costs, and the achievable maximum
gain along any walk is finite. This 
idea is also used in Section~\ref{sec:separated} which treats
deformations of the braid arrangement which contain the hyperplane
$x_i-x_j=0$ for each pair $\{i,j\}$: their regions refine the
regions of the braid arrangement. Due to the 
presence of nonnegative weights, Dijkstra's algorithm can not be used, but
for an important special case, when the 
weight function satisfies the {\em weak triangle inequality}, the gain
function may be computed by building a tree recursively. This gain
function is used in the final Section~\ref{sec:eShi} to introduce a
bijective labeling of the regions of an extended Shi arrangement with labeled
$a$-Catalan paths. It is worth exploring in the future whether some
variant of the gain function could play a role similar to parking
functions in a broader setting. Section~\ref{sec:contiguous} focuses on
integral deformations of the braid arrangement, in which the constants
$c$ appearing in the equations
$x_i-x_j=c$ form a contiguous interval of integers. Among other results
this section contains limits on the sizes of the directed cycles which
we need to check to verify that a weighted digraph represents a nonempty
region.

Sections~\ref{sec:eShi}, \ref{sec:ceiling} and \ref{sec:FC} contain
applications of our approach to the extended Shi, generalized Ish and
$a$-Catalan arrangements, and also extend the use of
Athanasiadis-Linusson diagrams to a broader class of 
arrangements. It turns out that this approach provides the fastest way
to count the regions in the $a$-Catalan arrangement (earlier found by
Bernardi~\cite{Bernardi} in this instance), and also in a class
of hyperplane arrangements properly containing the extended Shi arrangements.
Section~\ref{sec:ceiling} also contains the Simplification Lemma
(Lemma~\ref{lem:simplification}), describing a way to identify those
directed edges in our weighted digraphs which do not represent facet
inequalities of the represented region, but inequalities that follow
from inequalities represented by other edges. The idea behind this
simplification generalizes the notion of a ceiling diagram in the work
of Armstrong and Rhoades~\cite{Armstrong-Rhoades}.   

The Carver-Farkas lemma is likely suitable to count
regions of a hyperplane arrangement in many other settings as well. 

\section{Preliminaries}

\subsection{Two classical results}
The following variant of the Farkas' lemma has also been used
in~\cite{Hopkins-bigraphical}. It is originally due to
Carver~\cite[Theorem 3]{Carver}, who stated it in a slightly different
form. The formulation below may be found
in~\cite{Roos}.

\begin{lemma}[Carver-Farkas]
\label{lem:Farkas}  
Let $A$ be a real $m\times n$ matrix and let $b$ be a real $n\times 1$
column vector. Then the system of inequalities $Ax<b$ has no solution if
and only if there is a nonzero real $m\times 1$ row vector $y$ satisfying
$y\geq 0$, $yA=0$ and $yb\leq 0$. 
\end{lemma}

The other classical result we need is the  Flow Decomposition
Theorem~\cite[Theorem 8.8]{Korte-Vygens}, originally due to
Gallai~\cite{Gallai}. We apply it in the 
following setting. Consider a directed graph with edge set $E$. A {\em
  circulation} is a function $f:E\rightarrow {\mathbb R_{\geq 0}}$ satisfying
$\sum_{v=t(e)}
f(e)=\sum_{v=h(e)} f(e)$
at every vertex $v$.  Here $t(e)$ and 
$h(e)$ denote the tail, respectively the head
of $e$. We set no capacity constraint (upper bound) on the values
$f(e)$. The simplest example of a nonzero circulation is supported by a {\em
  directed cycle} $(e_1,e_2,\ldots,e_k)$ on $k$ edges and $k$ vertices, where
$h(e_i)=t(e_{i+1})$ holds for
$i=1,2,\ldots, k-1$ and
$h(e_k)=t(e_1)$. We identify the
directed cycle $(e_1,e_2,\ldots,e_k)$ 
with the circulation $f$ that assigns $1$ to the edges
$e_1,e_2,\ldots,e_k$ and zero to all other edges. The restriction of the
Flow Decomposition Theorem to circulations is the following
statement. 
\begin{theorem}
\label{thm:circdec}
Every not identically zero circulation $f$ can be written as a positive
linear combination of directed cycles. Moreover, a directed edge $e$
appears in at least one of these cycles if and only if $f(e)>0$.
\end{theorem}  
We will apply Theorem~\ref{thm:circdec} to digraphs with {\em a cost
  function} assigning a (possibly negative) cost $c(e)$ to each directed
edge $e$, where the cost of a circulation $f$ is the sum $c(f)=\sum_{e\in 
  E} c(e)\cdot f(e)$.

\subsection{Deformations of a graphical arrangement}

A hyperplane arrangement ${\mathcal A}$ is a finite collection of
hyperplanes in a $d$-dimensional real vector space, which partition
the space into regions. We may use the poset
$L_{{\mathcal A}}$ of nonempty intersections (ordered by reverse
inclusion) of the hyperplanes to count the regions. The {\em
  characteristic polynomial} $\chi({\mathcal A},q)$ of the arrangement
is defined as 
\begin{equation}
\chi({\mathcal A},q)=\sum_{x\in L_{{\mathcal A}}} \mu(\widehat{0},x)
q^{\dim(x)}, 
\end{equation}
where $\mu(x,y)$ is the M\"obius function of $L_{{\mathcal A}}$ and
  $\widehat{0}$ is the entire vector space.   
The numbers $r({\mathcal A})$ and $b({\mathcal A})$ of all, respectively
bounded regions may be found using
{\em Zaslavsky's formulas}~\cite{Zaslavsky}, stating  
\begin{equation}
  \label{eq:Zaslavsky}
r({\mathcal A})=(-1)^d \chi({\mathcal A},-1)\quad\mbox{and}\quad 
b({\mathcal A})=(-1)^{\rank(L_{{\mathcal A}})} \chi({\mathcal A},1).
\end{equation}
Various methods are known to compute the characteristic polynomial.
In the case when all equations have only integer coefficients, the
characteristic polynomial may be computed using the finite field
method~\cite[Theorem 2.2]{Athanasiadis-charpol}. When $\mathcal{A}$ is
constructed from a graph, Whitney's formula~\cite{Whitney} or the gain graph
method~\cite{Berthome-etc,Forge-Zaslavsky,Zaslavsky-perpendicular} may be used.
We focus on {\em deformations of the braid
  arrangement}, and our main reference is the work of 
Postnikov and Stanley~\cite{Postnikov-Stanley}. The {\em braid
  arrangement} or {\em Coxeter 
  arrangement of type $A_{n-1}$} is the collection of hyperplanes 
\begin{equation}
x_i-x_j=0,\quad 1\leq i<j\leq n
\end{equation}
in the subspace $V_{n-1}$ of ${\mathbb R}^n$, given by
$x_1+x_2+\cdots+x_n=0$.  
The braid arrangement is also a special case of
a {\em graphical arrangement ${\mathcal A}_G$} induced by a simple
connected undirected graph $G$ with edge set $E(G)$ on the vertex set
$\{1,2,\ldots,n\}$. It consists of the hyperplanes 
\begin{equation}
x_i-x_j=0,\quad \{i,j\}\in E(G)
\end{equation}
in $V_{n-1}$. Hence the braid arrangement is ${\mathcal A}_{K_n}$ where
$K_n$ is the complete graph on $n$ vertices. A {\em deformation} of a
graphical arrangement consists of replacing each hyperplane $x_i-x_j=0$
with a set of hyperplanes
\begin{equation}
\label{eq:dbraid}  
x_i-x_j=a_{ij}^{(1)}, a_{ij}^{(2)}, \ldots, a_{ij}^{(n_{ij})}, \quad
\{i,j\}\in E(G), 
\end{equation}  
where the $n_{ij}$ are nonnegative integers and the $a_{ij}^{(k)}$ are real
numbers.
\begin{remark}
\label{rem:boundedeq}
{\em Our choice to restrict our definition of a graphical arrangement to
  connected graphs only and restrict our hyperplanes to $V_{n-1}$ is a
  consistent generalization of the notation and terminology
  in~\cite{Postnikov-Stanley}. Another option is to consider
  graphical arrangements in the entire space $\mathbb{R}^n$, associate
  graphical arrangements to disconnected undirected graphs as well, but
  consider {\em relative bounded} regions instead. We refer the
  interested reader to~\cite[Definition 1.7]{Hopkins-bigraphical} for
  the exact definitions of that approach. See
  also Remark~\ref{rem:Vn-1}.
}
\end{remark}

Well-studied deformations of the braid
arrangement are the {\em truncated affine arrangements
  $\mathcal{A}^{a,b}_{n-1}$}. The integer parameters $a$ and $b$ satisfy
$a+b\geq 2$ and the hyperplanes are 
$$
x_i-x_j=1-a,2-a,\ldots,b-1 \quad
1\leq i<j\leq n.
$$
In particular, $\mathcal{A}^{0,2}_{n-1}$ is the {\em Linial arrangement}, 
$\mathcal{A}^{1,2}_{n-1}$ is the {\em Shi arrangement} 
$\mathcal{A}^{a,a+1}_{n-1}$ with $a\geq 1$ is the {\em extended Shi
  arrangement}, $\mathcal{A}^{2,2}_{n-1}$ is the {\em Catalan arrangement}, 
and $\mathcal{A}^{a,a}_{n-1}$ with $a\geq 2$ is the {\em $a$-Catalan
  arrangement}.      
\begin{remark}
\label{rem:symmetry}
{\em In the study of truncated affine arrangements, without loss of
  generality we may assume that $a\leq b$ holds: 
  replacing each vector $(x_1,x_2,\ldots,x_n)$ with $(x_n,x_{n-1},\ldots,x_1)$
sends the arrangement $\mathcal{A}^{ab}_{n-1}$ into the arrangement
$\mathcal{A}^{ba}_{n-1}$, since the equation corresponding to $x_i-x_j=c$
is the equation $x_{n+1-j}-x_{n+1-i}=-c$ after the linear
transformation.}    
\end{remark}

\section{Labeling regions in deformations of graphical arrangements}
\label{sec:labeling}

Describing a region in a deformation of a graphical arrangement amounts to
determining whether a system of linear  inequalities of the form
\begin{equation}
\label{eq:dp}  
m_{ij}<x_i-x_j<M_{ij}, \quad 1\leq i< j\leq n 
\end{equation}
has a solution in $V_{n-1}$. Here we assume that $m_{ij}<M_{ij}$ holds for all
$(i,j)$ and we allow $m_{ij}=-\infty$ and $M_{ij}=\infty$
respectively.   
\begin{definition}
We call the solution set of a system of linear inequalities of the form
\eqref{eq:dp} in the set $V_{n-1}$ a {\em weighted digraphical polytope}. 
\end{definition}

\begin{remark}
\label{rem:Vn-1}
{\em When we count regions without regard to their
boundedness, we may equivalently consider the solution set of
\eqref{eq:dp} in ${\mathbb R}^n$. Here the solution set is either empty
or unbounded: for any real number $r$,
replacing each $x_i$ with $x_i+r$ leaves all differences $x_i-x_j$
unchanged. On the other hand, if we subtract $\sum_{i=1}^n x_i$ from each
$x_i$, we obtain a point in $V_{n-1}$.}
\end{remark}  

\begin{definition}
To each system of inequalities~\eqref{eq:dp} we create its {\em
  associated weighted digraph} as follows. For each $i<j$, if
$m_{ij}>-\infty$, we create 
  directed edge $i\rightarrow j$ with weight $m_{ij}$ and if $M_{ij}<\infty$
 we also create a directed edge $i\leftarrow j$ with weight
 $-M_{ij}$. An {\em $m$-ascending cycle} in the associated weighted
 digraph is a directed cycle, along which the sum of the labels is
 nonnegative.  We call the associated weighted digraph {\em $m$-acyclic},
 if it contains no {\em $m$-ascending cycle}.   
\end{definition}

The associated weighted digraph uniquely encodes the
system~\eqref{eq:dp} (but not its solution set). Our definition of an
$m$-ascending cycle is designed  
to match the conventions of labeling the regions of the
Linial arrangement using semiacyclic tournaments, see
Example~\ref{ex:sat} below. To maintain this compatibility we make the
following definition. 
\begin{definition}
\label{def:negcost}  
The cost of an edge in an associated weighted digraph
of a system of inequalities~\eqref{eq:dp} is the negative of its weight.
Equivalently, the weight is the amount we gain by using an edge. 
\end{definition}
Hence an $m$-ascending cycle is a cycle with non-positive cost. Using
Lemma~\ref{lem:Farkas} we obtain the following characterization of
nonempty weighted digraphical polytopes.  

\begin{theorem}
\label{thm:mincost}
A weighted digraphical polytope given by a system of inequalities of the form
\eqref{eq:dp} is not empty if and only if the weighted digraph
associated to~\eqref{eq:dp} is $m$-acyclic. 
\end{theorem}
\begin{proof}
When we rewrite the inequalities \eqref{eq:dp} in the form  $Ax<b$, we
must rewrite all inequalities $-\infty\neq m_{ij}<x_i-x_j$ as
$x_j-x_i<-m_{ij}$. Let us associate to each such inequality a distinct
variable $u_{ij}$.  We keep all remaining inequalities of the form
$x_i-x_j<M_{ij}\neq \infty$ unchanged, and we associate to each such
inequality a distinct variable $v_{ij}$. By Lemma~\ref{lem:Farkas} the system of
inequalities \eqref{eq:dp} has no solution if an only if there is a pair
of vectors $(u,v)$ where $u=(u_{ij}\::\: 1\leq i<j\leq n, m_{ij}\neq
-\infty)$ and $v=(v_{ij}\::\: 1\leq i<j\leq n, M_{ij}\neq \infty)$,
such that the following are satisfied: 
\begin{enumerate}
\item All coordinates $u_{ij}$ and $v_{ij}$ are nonnegative and at
  least one of the vectors $u$ or $v$ is not zero. 
\item For each $i\in\{1,2,\ldots,n\}$ we have
  $$
-\sum_{j>i} u_{ij}+\sum_{k<i} u_{k,i}-\sum_{k<i} v_{k,i}+\sum_{j>i}
v_{ij}=0.  $$
\item The inequality
\begin{equation}
\label{eq:mflow}    
-\sum_{i<j} u_{ij}\cdot m_{ij}+\sum_{i<j} v_{ij}\cdot M_{ij}\leq 0
\quad\mbox{holds}. 
\end{equation}
\end{enumerate}
Condition (2) amounts to stating the following: if for
each $i<j$ satisfying $m_{ij}\neq -\infty$ we let $u_{ij}\geq 0$
units flow from $i$ to $j$ and for each $i<j$ satisfying $M_{ij}\neq\infty$
we let $v_{ij}\geq 0$ units flow from $j$ to $i$, we obtain a 
  circulation. Let us call a circulation {\em
  $m$-ascending}  if it satisfies \eqref{eq:mflow}. If we think of the
numbers $-m_{ij}$ and $M_{ij}$ as costs, then an $m$-ascending
circulation is simply a circulation whose cost is not positive. 
By Theorem~\ref{thm:circdec} there is a nonzero $m$-ascending circulation
if and only if there is also an $m$-ascending cycle. 
\end{proof}  

\begin{remark}
  {\em The net flow between $i$ and $j$ does not
    change if we  decrease both $u_{ij}$ and $v_{ij}$ by the same
    positive real number $r$, whereas the sum on the left hand side of
    \eqref{eq:mflow} decreases by $M_{ij}-m_{ij}\geq 0$. Hence if there
    is an $m$-ascending circulation, then there is also such a
    circulation in which for any $i<j$ at most one of $u_{ij}$ and $v_{ij}$ is
    positive. A related observation is that there is an $m$-ascending
    cycle of length $2$ if and only if $M_{ij}-m_{ij}< 0$ holds for some
    $i<j$: in this case $m_{ij}<x_i-x_j<M_{ij}$ has no solution.}    
\end{remark}

\begin{remark}
\label{rem:extend}  
  {\em We may also introduce an edge
    $i\rightarrow j$ (respectively $i\leftarrow j$) of weight $-\infty$
    whenever $i<j$ and $m_{ij}=-\infty$ (respectively $M_{ij}=\infty$)
    holds, and we may extend the addition operation to ${\mathbb R}\cup
    \{-\infty\}$ by setting $r+(-\infty)=(-\infty)$ 
for any $r\in {\mathbb R}\cup \{-\infty\}$. This addition makes no
substantial difference as no $m$-ascending cycle could contain a
directed edge of weight $-\infty$.}  
\end{remark}
  
Theorem~\ref{thm:mincost} may be rephrased as follows. 
\begin{corollary}
\label{cor:Floyd}
Consider the associated weighted digraph $D$ encoding a system of
inequalities of the form~\eqref{eq:dp} and think of the weight $w(e)$ as money
we gain when we walk from the tail of the edge $e$ to its head. Then
the system of inequalities~\eqref{eq:dp} has a nonempty solution set if
and only if we lose money along any closed walk. 
\end{corollary}   
In other words, $m$-acyclic weighted digraphs are exactly the ones to
which the Floyd-Warshall algorithm may be applied to find a unique
minimum cost directed path between any pair of vertices.  

\begin{example}
  \label{ex:sat}
  {\em Each region of the Linial arrangement $\mathcal{A}^{0,2}_{n-1}$
    is described by a set of inequalities
$$
m_{ij}<x_i-x_j<M_{ij}, \quad 1\leq i< j\leq n 
$$
where each of these inequalities is either $-\infty<x_i-x_j<1$ or
$1<x_i-x_j<\infty$. The associated weighted digraph is a {\em
  tournament}: for each pair $i<j$, exactly one of the directed
edges $i\rightarrow j$ (of weight $1$, an {\em ascent}) and $i\leftarrow
j$ (of weight $-1$, a {\em descent}) belongs to the digraph. In this
setting, a directed cycle is $m$-ascending if and only if the number
of directed edges of weight $-1$ does not exceed the number of directed
edges of weight $1$ in it. This is precisely the
definition of an {\em ascending cycle} in~\cite{Postnikov-Stanley}, they
state it in terms of the numbers of ascents and descents. 
By the definition of~\cite{Postnikov-Stanley} a tournament is {\em
  semiacyclic} if and only if it contains no ascending cycle. The observation
that semiacyclic tournaments are in bijection with the regions of the
Linial arrangement was independently made by Postnikov and Stanley and
by Shmulik Ravid.}  
\end{example}  

Our next result helps identify the bounded regions in a deformation of
the braid arrangement. 

\begin{theorem}
\label{thm:br}
A weighted digraphical polytope, given by a system of inequalities of the form
\eqref{eq:dp}, is not empty and bounded if and only if the associated
weighted digraph is $m$-acyclic and it is strongly connected. 
\end{theorem}
\begin{proof}
By Theorem~\ref{thm:mincost} we may assume that the associated weighted
digraph is $m$-acyclic: this is equivalent to assuming
that our polytope is not the empty set. 

Assume first that the associated weighted digraph is not strongly
connected. Then the set $\{1,2,\ldots,n\}$ may be partitioned into two
disjoint subsets $V_1$ and $V_2$ such that for each $v_1\in V_1$ and
$v_2\in V_2$ the directed edge $v_1\leftarrow v_2$ does not belong to
the associated weighted digraph. Let $(x_1,\ldots,x_n)$ be any point
satisfying \eqref{eq:dp}. We claim that the point $(x_1',\ldots,x_n')$
given by
$$
x_v'=
\begin{cases}
  x_v+\frac{t}{|V_1|},&\mbox{if $v\in V_1$;}\\
    x_v-\frac{t}{|V_2|},&\mbox{if $v\in V_2$.}\\
\end{cases}  
$$
also satisfies \eqref{eq:dp} for all $t>0$. Indeed $x_i'-x_j'=x_i-x_j$
holds if both $i$ and $j$ belong to $V_1$ or both of them belong to
$V_2$. We are left to consider the inequalities where one of the indices
belongs to $V_1$ and the other to $V_2$. Without loss of generality we
may assume $i\in V_1$ and $j\in V_2$. If there is any bound on
$x_i-x_j$, it is of the form $m_{ij}<x_i-x_j$ if $i<j$ and it is of the
form $x_j-x_i<M_{ji}$ if $i>j$.  In either case, the inequality is even
more valid if we increase $x_i$ to $x_i'$ and we decrease $x_j$ to
$x_j'$. Observe finally that we increased $|V_1|$ coordinates by
$t/|V_1|$ and we decreased $|V_2|$ coordinates by
$t/|V_2|$, hence $\sum_{i=1}^n x_i'=\sum_{i=1}^n x_i=0$. 
The value of $t$ is not bounded from above, hence our
digraphical polytope must be unbounded.

Assume next that the associated weighted digraph $D$ defining a weighted
digraphical polytope $P\subset V_{n-1}$ is strongly
connected. We show by induction on $n$ that $P$ is bounded. There is
nothing to prove for $n=1$: $V_0$ consists of a single point and every
directed graph with a single vertex is strongly connected. For $n=2$,
the region defined by $m_{1,2}<x_1-x_2<M_{1,2}$ is an open interval, and
it is bounded exactly when both $m_{1,2}\neq -\infty$ and $M_{1,2}\neq
\infty$ hold, which is precisely the case when the vertices $1$ and $2$
are linked by directed edges both ways. Assume from now on that $n>2$
holds. Since the digraph is strongly connected, neither the set
$\operatorname{In}(n)=\{i\: : \: i\rightarrow n\}$ nor the set
$\operatorname{Out}(n)=\{j\: : \: n\rightarrow j\}$ is empty. Each weighted
arrow  $i\rightarrow n$ represents an inequality $m_{in}<x_i-x_n$ which
we rearrange as $x_n<x_i-m_{in}$. Each weighted directed edge  $n\rightarrow j$
represents an inequality $x_j-x_n<M_{jn}$ which we rearrange as
$-x_n<M_{jn}-x_j$. For each pair $(i,j)$ with
$i\in\operatorname{In}(n)$ and $j\in\operatorname{Out}(n)$, we take the
sum of the inequalities represented by $i\rightarrow n$ and $j\leftarrow
n$ and obtain $0<x_i-m_{in}+M_{jn}-x_j$. We rearrange this inequality
as $m_{in}-M_{jn}<x_i-x_j$ if $i<j$ and as $x_j-x_i<-m_{in}+M_{jn}$
if $j<i$. We add these inequalities, remove the inequalities
involving $x_n$, and we obtain the definition of a weighted digraphical
polytope $P'\subset V_{n-2}$. (Note that we may have created several
upper and lower bounds for the same $x_i-x_j$ but we only need to
consider the greatest lower bound and the least upper bound.) Its 
associated weighted digraph $D'$ is still strongly connected: if a
directed path, connecting two elements of the set $\{1,2,\ldots,n-1\}$
does not pass through $n$ then the same walk is also present in $D'$, if
it passes through $n$ then we may replace the subwalk $i\rightarrow
n\rightarrow j$ in $D$ with a single edge $i\rightarrow j$ in $D'$ as we
created the edge $i\rightarrow j$ when we took the sum of the
inequalities associated to $i\rightarrow n$ and $n\rightarrow j$. By the
induction hypothesis $P'$ is a bounded (or empty) polytope. Since $P'$
is bounded, there is a cube $[-R,R]^{n-1}$ containing it for
some $R>0$.  Consider now
any point $(x_1,x_2,\ldots,x_n)\in P$ and let
$s=(\sum_{i=1}^{n-1} x_i)/(n-1)$. We claim that the point
$(x_1-s,x_2-s,\ldots,x_{n-1}-s)$ belongs to $P'$. Indeed, for any
$i,j\in\{1,2,\ldots,n-1\}$ the difference $(x_i-s)-(x_j-s)$ is the same
as $x_i-x_j$, and every bound imposed on $x_i-x_j$ in $P'$ is a
consequence of the linear inequalities defining $P$. Furthermore, we
have $\sum_{i=1}^{n-1} (x_i-s)=0$. (In particular, $P'$ is not empty.)
Using any $i\in\operatorname{In}(n)$ and any
$j\in\operatorname{Out}(n)$, we may write 
$$
-R-M_{jn}\leq x_j-s-M_{jn}<x_n-s<x_i-s-m_{in}\leq R-s-m_{in}.
$$
As a consequence, $x_n-s$ is also bounded, there is an $R^*>0$ such that
$(x_1-s,x_2-s,\ldots,x_n-s)\in [-R^*,R^*]^n$. Observe finally that
$\sum_{i=1}^n x_i=0$ implies $(n-1)s+x_n=0$, that is, $x_n-s=-ns$. Since
$x_n-s$ is bounded, so is $s=(x_n-s)/(-n)$ and the vector
$(x_1,x_2,\ldots,x_n)$, obtained by adding $(s,s,\ldots,s)$ to
$(x_1-s,x_2-s,\ldots,x_n-s)$, is also bounded.   
\end{proof}  

\begin{remark}
{\em A variant of Theorem~\ref{thm:br} is stated in~\cite[Theorem
    1.8]{Hopkins-bigraphical}. As noted in Remark~\ref{rem:boundedeq},
  the paper of Hopkins and Perkinson considers relative bounded regions
  in $\mathbb{R}^n$ instead of bounded regions in $V_{n-1}$. Their
  result may be generalized to deformations of 
graphical arrangements as follows: if we define the associated
weighted digraphs the same way as before, then relative bounded regions
correspond to digraphs whose strongly connected components are the same as the
weakly connected components. We leave the proof of this variant to the
reader.}    
\end{remark}

\begin{corollary}
\label{cor:csat}
The bounded regions of the Linial arrangement are in bijection with the
strongly connected semiacyclic tournaments.   
\end{corollary}  

\begin{remark}
  {\em
Athanasiadis computed the number $b({\mathcal L}_n)$ of bounded regions
of the Linial arrangement~\cite[Theorem 4.2]{Athanasiadis-charpol}  
and raised the question whether there is a combinatorial interpretation
of the numbers $b({\mathcal L}_n)$ similar to the combinatorial
interpretation of the number of all regions given by Postnikov and
Stanley~\cite{Postnikov-Stanley}. Corollary~\ref{cor:csat} provides a
new response to this question. Previous models were provided by
Tewari~\cite[Theorem~1.1]{Tewari} and by Fl\'orez and Forge~\cite[Theorem
2.2]{Forge}. It is also worth pointing out that Bernardi~\cite{Bernardi}
labels the regions of the Linial arrangement with trees which arise
after considering the regions of the Linial arrangement as unions of
regions of the Catalan arrangement. Bernardi does not address the
question which trees represent bounded regions, but looking
at~\cite[Fig.\ 10]{Bernardi} it is easy to notice that a region of the
Catalan arrangement is bounded if and only if Bernardi's associated
$1$-sketch is connected in the sense that there is no way to place a
vertical line separating the points in the diagram into two nonempty
sets without intersecting at least one arc. This statement is easily shown
using the proof techniques of Theorem~\ref{thm:br}. It is an interesting
  question for future research whether this observation can be extended
  to Bernardi's labeled trees encoding the regions of the Linial
  arrangement.}    
\end{remark}

We may generalize our observations regarding the Linial arrangement
to an arbitrary deformed graphical arrangement $\mathcal{A}$ as follows. 
Assume $\mathcal{A}$ is given by \eqref{eq:dbraid}.
Without loss of generality we may assume that for each ordered pair
$(i,j)$ satisfying $i<j$ the numbers $a_{ij}^{(1)}, a_{ij}^{(2)},
\ldots, a_{ij}^{(n_{ij})}$ are listed in increasing order. 
Each region of $\mathcal{A}$ is a weighted digraphical polytope and
  it is described by a system of inequalities~\eqref{eq:dp} where each
  $m_{ij}$ is either $-\infty$ or some element of the set 
$A_{ij}=\{a_{ij}^{(1)}, a_{ij}^{(2)}, \ldots, a_{ij}^{(n_{ij})}\}$
and $M_{ij}$ is given by the following formula:
\begin{equation}
\label{eq:Mij}  
  M_{ij}=
  \begin{cases}
    a_{ij}^{(1)}, & \mbox{if $m_{ij}=-\infty$};\\
    a_{ij}^{(k+1)}, & \mbox{if $m_{ij}=a_{ij}^{(k)}$ for some $k<n_{ij}$};\\
    \infty & \mbox{if $m_{ij}=a_{ij}^{(n_{ij})}$}.\\ 
  \end{cases}  
\end{equation}
Keeping in mind these possibilities, we define a {\em valid weighted
  digraph} associated to a deformation of a graphical arrangement in
$V_{n-1}$ as
follows. 
\begin{definition}
\label{def:valid}  
Let $\mathcal{A}$ be a deformation of a graphical arrangement, given
by~\eqref{eq:dbraid}. We call a weighted digraph with vertex set
$\{1,2,\ldots,n\}$ {\em valid} if  for each $(i,j)$ satisfying $1\leq
i<j\leq n$ exactly one of the following holds:
\begin{enumerate}
\item $n_{ij}=0$ and there is no directed edge between $i$ and $j$.
\item $n_{ij}>0$, there is no directed edge $i\rightarrow j$, and there
  is exactly one directed edge $i\leftarrow j$ which has weight $-a_{ij}^{(1)}$.
\item $n_{ij}>0$, there is a directed edge $i\rightarrow j$ of weight
  $a_{ij}^{(k)}$, and there is exactly one directed edge $i\leftarrow j$
  which has weight $-a_{ij}^{(k+1)}$, for some $k<n_{ij}$.
\item $n_{ij}>0$, there is exactly one directed edge $i\rightarrow j$
  which has weight $a_{ij}^{(n_{ij})}$, and there is no directed edge
  $i\leftarrow j$. 
\end{enumerate}  
\end{definition}
Since the associated weighted digraphs contain at most one directed
edge $i\rightarrow j$ for any ordered pair $(i,j)$, we may uniquely
encode each such weighted digraph with a {\em weight function}
$$
w: \{1,2,\ldots,n\}\times \{1,2,\ldots,n\} \rightarrow {\mathbb R} \cup
\{-\infty\} 
$$
by setting $w(i,j)$ to be the weight of the directed edge $i\rightarrow j$, if
$i\rightarrow j$ is present in the weighted digraph and $w(i,j)=-\infty$ otherwise.
The $m$-acyclic condition may be then rephrased as follows:
\begin{equation}
\label{eq:m-acyclic}  
w(i_1,i_2)+w(i_2,i_3)+\cdots+w(i_{m-1},i_m)+w(i_m,i_1)<0
\end{equation}  
must hold for any cyclic list $(i_1,i_2,\ldots,i_m)$ of elements of
$\{1,2,\ldots,n\}$. Here we extend the rules of addition to ${\mathbb R} \cup
\{-\infty\}$ as in Remark~\ref{rem:extend}.     
To simplify the notation in all subsequent proofs, we introduce the
shorthand notation 
\begin{equation}
\label{eq:shorthand}  
w(i_1,i_2,\ldots,i_m)=w(i_1,i_2)+w(i_2,i_3)+\cdots+w(i_{m-1},i_m)+w(i_m,i_1)
\end{equation}  
for the total weight of all directed edges along the closed walk
$i_1\rightarrow i_2 \rightarrow \cdots \rightarrow i_m\rightarrow i_1$. 
As a consequence of Theorems~\ref{thm:mincost} and \ref{thm:br}
we have the following result.
\begin{corollary}
\label{cor:bijection}
 Let $\mathcal{A}$ be a deformation of a graphical arrangement, given
by~\eqref{eq:dbraid}. Then the regions of $\mathcal{A}$ are in bijection
with the valid $m$-acyclic weighted digraphs on
$\{1,2,\ldots,n\}$ in such a way that bounded regions correspond to
strongly connected  valid $m$-acyclic weighted digraphs. 
\end{corollary}
Indeed, each region created by the hyperplanes of $\mathcal{A}$ may be
uniquely encoded by exactly one valid $m$-acyclic weighted digraph of the
given form: for each $i,j\in \{1,2,\ldots,n\}$ satisfying
$i<j$, the value of $x_i-x_j$ belongs to exactly one interval created by
the set $A_{ij}=\{a_{ij}^{(1)}, a_{ij}^{(2)}, \ldots, a_{ij}^{(n_{ij})}\}$ and
this interval is the same for all points of the same region. The
assignment of valid weighted digraphs to regions is thus injective. By
Theorem~\ref{thm:mincost}, exactly the $m$-acyclic valid weighted digraphs
encode sets of inequalities with a nonempty solution set, and by
Theorem~\ref{thm:br} exactly the strongly connected $m$-acyclic valid
weighted digraphs represent bounded regions. 

The $m$-acyclic property can be independently
verified within each strong component of the weighted digraph as each
cycle of a digraph stays within the same strong component. Hence we
have the following structure theorem. 
\begin{theorem}
\label{thm:dbr}
Assume a deformation of a graphical arrangement given by~\eqref{eq:dbraid}
has the property that $n_{ij}>0$ holds for all $(i,j)$ satisfying $1\leq
i<j\leq n$. Then the associated weighted digraph of any region may be
uniquely constructed as follows.
\begin{enumerate}
\item We fix an ordered set partition $(N_1,N_2,\ldots,N_k)$ of the set
  $\{1,2,\ldots n\}$. The parts of the ordered set partition will be the
  vertex sets of the strong components. For any $(i,j)$ having the
  property that the part containing $i$ precedes the part containing $j$
  there is a directed edge $i\rightarrow j$ but no directed edge
  $i\leftarrow j$. The label on $i\rightarrow j$ must be
  $a_{ij}^{(n_{ij})}$ if $i<j$ and it must be $-a_{ij}^{(1)}$ if $i>j$.
\item On each strong component $N_i$ we may select a
  strongly connected valid $m$-acyclic weighted digraph independently. 
\end{enumerate}
\end{theorem}
\begin{proof}
By our assumption each valid weighted digraph has at least one
directed edge between any two vertices, hence the strong components may
be linearly ordered in such a way that for any $i$ in a preceding strong
component and for any $j$ in a succeeding strong component there is a
directed edge $i\rightarrow j$ but no directed edge $i\leftarrow j$.
The weighting of these edges can only be as stated. As noted above,
the $m$-acyclic property only needs to be verified on the strongly connected
components.  
\end{proof}  

In~\cite{Stanley-hit} Stanley considers a sequence
$\mathcal{A}=(\mathcal{A}_1,\mathcal{A}_2,\ldots)$ of deformations
 of the braid arrangement, such that each $\mathcal{A}_n$ is a hyperplane
 arrangement in $\mathbb{R}^n$ and for each $S\subseteq \{1,2,\ldots,n\}$
 he defines $\mathcal{A}^S_n$ as the subcollection of hyperplanes
 $x_i-x_j=c$ of $\mathcal{A}_n$ satisfying $\{i,j\}\subseteq S$. He
 calls such a sequence {\em exponential} if the number
 $r(\mathcal{A}^S_n)$ of regions of $\mathcal{A}^S_n$ depends only on
 $k=|S|$ and it is the number $r(\mathcal{A}_k)$ of regions of $\mathcal{A}_k$.
 Introducing the exponential generating functions
 $$
R_{\mathcal{A}}(t)=\sum_{n\geq 0} r(\mathcal{A}_n)\cdot \frac{t^n}{n!}
\quad\mbox{and}\quad
B_{\mathcal{A}}(t)=\sum_{n\geq 1} b(\mathcal{A}_n)\cdot \frac{t^n}{n!}
$$
for all, respectively the bounded regions, Stanley~\cite[Theorem
  1.2]{Stanley-hit} shows
\begin{equation}
\label{eq:br}  
B_{\mathcal{A}}(t)=1-\frac{1}{R_{\mathcal{A}}(t)}.
\end{equation}  
The outline of the proof cites Zaslavsky's
formula~\eqref{eq:Zaslavsky}, Whitney's formula~\cite{Whitney} and the
exponential formula in enumerative combinatorics. Formula~\eqref{eq:br}
may also be derived from Theorem~\ref{thm:dbr} as follows. Consider an
exponential sequence $\mathcal{A}=(\mathcal{A}_1,\mathcal{A}_2,\ldots)$
of deformations of the braid arrangements. As a consequence of Zaslavsky's
formula~\eqref{eq:Zaslavsky}, for each $n\geq 1$ and for each $S\subseteq
\{1,2,\ldots,n\}$ the number $b(\mathcal{A}^S_n)$ of bounded regions of
$\mathcal{A}^S_n$ also depends only on
 $k=|S|$ and it is the number $b(\mathcal{A}_k)$ of bounded regions of
$\mathcal{A}_k$. To construct an associated weighted digraph of a
region of $\mathcal{A}_n$ we first fix an ordered set partition
$(N_1,N_2,\ldots,N_k)$ of the set $\{1,2,\ldots n\}$. After fixing $k$
and the size $n_i$ of each $N_i$, the number of ways we can select an ordered
set partition is $\binom{n}{n_1,n_2,\ldots,n_k}$. In the next step we
must select a strongly connected valid $m$-acyclic weighted digraph on
each part of our set partition. There are $b(\mathcal{A}_{n_i})$ ways to
perform this step on the part $N_i$. Hence we obtain 
\begin{equation}
  r(\mathcal{A}_n)=\sum_{k=1}^n \sum_{\substack{n_1+\cdots +
      n_k=n\\ n_1,\ldots, n_k>0}}
  \binom{n}{n_1,n_2,\ldots,n_k} \prod_{i=1}^k
  b(\mathcal{A}_{n_i})\quad\mbox{for all $n\geq 1$}. 
\end{equation}
This formula implies $R_A(t)=\sum_{k\geq 0} B_A(t)^k$.

\begin{remark}
{\em 
 In the case when all parameters $m_{i,j}$ and $M_{i,j}$ are integers, a
weighted digraphical polytope is a special case of an {\em alcoved
  polytope} as defined by Lam and Postnikov~\cite{Lam-Postnikov-1,
  Lam-Postnikov-2}. Here we only
consider hyperplanes of the form $x_i-x_j=c_{ij}$, but we also allow
non-integer parameters. To extend our approach to type $B$ arrangements is
a subject of ongoing research. The Lam-Postnikov approach is based on counting
lattice points and unit volume simplices, not any version of the Farkas'
lemma or of Zaslavsky's theorem. This same approach is present in the
work of  Benedetti, Knauer and Valencia-Porras~\cite{Benedetti-Knauer-Valencia}
where it is used to compute the $h^*$-vector of certain matroid
polytopes. It is also worth noting
that some of the earliest examples of weighted digraphical polytopes are the
alcoves introduced in Shi's work on affine sign types corresponding to
an affine Weyl group~\cite[Section 6]{Shi}.}  
\end{remark}  

\section{The poset of gains and sparse deformations}
\label{sec:og}

In this section we introduce a key notion, the {\em poset of gains}
associated to a deformation of a graphical arrangement. Special
instances of this partial order include {\em 
  sleek posets}~\cite[Section 8.2]{Postnikov-Stanley}, encoding the
regions of the Linial arrangement, {\em
  semiorders}~\cite[Section~7]{Postnikov-Stanley} encoding the regions
of the semiorder arrangement and {\em interval orders}~\cite{Stanley-hit}
encoding hyperplane arrangements. In later sections we will also see
that whenever the poset of gains is a linear order, it encodes the
region of the braid arrangement containing the given region. 

Consider a deformation of a graphical arrangement given
by~\eqref{eq:dbraid} and one of its valid $m$-acyclic associated
weighted digraphs. We may use the weights to define a partial order on
the vertex set $\{1,2,\ldots,n\}$ as follows.

\begin{definition}
\label{def:og}  
Given a valid $m$-acyclic weighted digraph $D$ on 
$\{1,2,\ldots,n\}$, we define $i<_D j$ if there is a directed path
$i=i_0\rightarrow i_1 \rightarrow \cdots \rightarrow i_k=j$ such that
the weight of each directed edge $i_s\rightarrow i_{s+1}$ is
nonnegative. We call the set $\{1,2,\ldots,n\}$, ordered by $<_D$ the
{\em poset of gains induced by $D$}. 
\end{definition}

The fact that $i<_D j$ is a partial order is a direct consequence of
the $m$-acyclic property stated in Theorem~\ref{thm:mincost}. In terms
of Corollary~\ref{cor:Floyd}, $i<_D j$ holds if we can find a directed
path from $i$ to $j$ such that 
we do not lose money at any step by walking from $i$ to $j$ using that path.
In general, the poset of gains carries less information than
recording the actual weights, but in some special cases all information
may be reconstructed from it. One example is the
Linial arrangement: its posets of gains are the sleek posets,
see~\cite[Section 
  8.2]{Postnikov-Stanley}. Semiacyclic tournaments and sleek posets are
in bijection: for $i<j$, the relation $i<_D j$ holds exactly when the
$w(i,j)=1$, if $w(i,j)=-1$ then
$i$ and $j$ are incomparable. It has been shown in~\cite[Theorem~8.6]{Postnikov-Stanley} that the length of a minimal 
  ascending cycle is at most $4$, equivalently, sleek posets may be
  characterized by a finite set of excluded subposets on at most $4$
  elements.

  A similar approach may be taken to the {\em semiorder arrangement},
  defined as the set of hyperplanes 
$$
x_i-x_j=-1,1 \quad \mbox{for $1\leq i<j\leq n$ in $V_{n-1}$.}
$$
For these, in any associated
valid weighted digraph, regardless of the order of the numbers $i$ and
$j$, there is either a single directed edge of weight $1$ between $i$
and $j$, or there are two edges of weight $-1$, one in each
direction. The elements $i$ and $j$ are comparable in the poset 
of gains exactly when there is a single directed
edge between them, pointing toward the larger
element. These are exactly the {\em semiorders} as defined
in~\cite[Section~7]{Postnikov-Stanley}.   In analogy
to~\cite[Theorem~8.6]{Postnikov-Stanley} one may directly show
the following:
\begin{proposition}
In a valid weighted digraph associated to the semiorder arrangement, the
length of a shortest $m$-ascending cycle is at most $4$.
\end{proposition}
The equivalent statement for semiorders was first shown by Scott and
Suppes~\cite{Scott-Suppes}. The Linial arrangement and the semiorder
arrangement are both examples of the following class of hyperplane arrangements.
\begin{definition}
We call a deformation $\mathcal{A}$ of the braid arrangement, given
by~\eqref{eq:dbraid} {\em sparse} if for each $i<j$ the set
$A_{ij}=\{a_{ij}^{(1)},\ldots,a_{ij}^{(n_{ij})}\}$ has one or two elements,
the largest element $a_{ij}=a_{ij}^{(n_{ij})}$ is positive and the
other element (if it exists) is negative. We denote the smaller element by
$-a_{ji}$ if $n_{ij}=2$, and we set  $a_{ji}=\infty$ if $n_{ij}=1$. 
We call~$\mathcal{A}$  {\em Euclidean} if
$a_{ij}+a_{jk}\geq a_{ik}$ holds for each $\{i,j,k\}\subseteq
\{1,2,\ldots,n\}$. 
We call $\mathcal{A}$ a {\em generalized semiorder arrangement} if 
$n_{ij}=2$ holds for all $i<j$.
\end{definition}
The $G$-semiorder arrangements discussed in~\cite{Hopkins-orientations} are
closely related to our generalized semiorder arrangements.
Note that the hyperplane arrangements associated to interval orders
in~\cite{Stanley-hit} are generalized semiorder arrangements in our
setting: if the lengths of a collection of $n$ intervals are given by the 
vector $(\ell_1,\ldots,\ell_n)\in {\mathbb R}_+^n$, the hyperplane
arrangement associated 
to their interval order is given by the equations
$$x_i-x_j=-\ell_j,\ell_i,\quad i<j.$$
In our notation we have $a_{ij}=\ell_i$ for all $i\neq j$ and the
Euclidean property is a direct consequence of the inequality
$\ell_i+\ell_j>\ell_i$. Another example of a sparse Euclidean
deformation of the braid arrangement is the Linial arrangement.   
Regarding Euclidean sparse deformations our key tools are summarized in
the next statement. 
\begin{proposition}
\label{prop:spdesc}  
Consider a Euclidean sparse deformation of the braid arrangement and any valid
$m$-acyclic weighted digraph $D$ associated to it. In the induced poset
of gains, $i<_D j$ holds exactly when there is a single directed
edge $i\rightarrow j$ of positive weight. For any pair $\{i,j\}$ of
incomparable vertices satisfying $i<j$, the edge $j\rightarrow i$ is
always present, and any edge between $i$ and $j$ has negative weight.     
\end{proposition}
The straightforward verification is left to the reader. An interesting
property of Euclidean sparse deformations of the braid arrangement is
that a bounded region corresponds to a poset of gains with a connected
{\em incomparability graph}. 
\begin{definition}
The {\em incomparability graph} of a partially ordered set is the undirected
graph whose vertices are the elements of the poset, and the edges are
the incomparable pairs of elements. 
\end{definition}

\begin{theorem}
\label{thm:sparsebounded}  
Let $D$ be a valid $m$-acyclic weighted digraph associated to a
Euclidean sparse
deformation of the braid arrangement in $V_{n-1}$. If $D$ is strongly
connected then the incomparability graph of the induced
poset of gains is connected. The converse is also true when
$n_{ij}=2$ holds for all $1\leq i<j\leq n$.   
\end{theorem}  
\begin{proof}
Assume first $D$ is strongly connected. We show for any pair of vertices
$\{i,j\}$ that there is a path between them in the incomparability graph
of $<_D$. There is nothing to prove if $i$ and $j$ are
incomparable. Without loss of generality we may assume that $i<_D j$
holds, by Proposition~\ref{prop:spdesc}  this implies the presence of a
single directed edge $i\rightarrow j$ in $D$. Since $D$ is strongly
connected, there is a path $j=i_0\rightarrow i_1\rightarrow \cdots
\rightarrow i_k=i$ from $j$ to $i$ in $D$ in which each edge
$i_s\rightarrow i_{s+1}$ is either between a pair of incomparable
vertices, or corresponds to the relation $i_s<_D i_{s+1}$. It suffices
to show that in a shortest such path all edges correspond to
incomparable pairs. Assume, by contradiction, that $i_s<_D
i_{s+1}$ holds for some $i_s\rightarrow i_{s+1}$, and consider first the
case when this edge is succeeded by an edge  
$i_{s+1}\rightarrow i_{s+2}$. The edge $i_{s+1}\rightarrow i_{s+2}$ must
correspond to an incomparable pair of vertices, otherwise $i_s<_D
i_{s+1}<_D i_{s+2}$ holds, and the segment  $i_s\rightarrow
i_{s+1}\rightarrow i_{s+2}$ in 
our path may be shortened to $i_s\rightarrow i_{s+2}$. A similar
shortening may also occur if the vertex $i_{s+2}$ forms an incomparable
pair with $i_s$, or if $i_s<_D i_{s+2}$ holds. Hence we must have
$i_{s+2}<_D i_s$ and $i_{s+2}<_D i_s <_D i_{s+1}$, in contradiction
with the presence of the edge $i_{s+1}\rightarrow i_{s+2}$.  A similar
contradiction may be reached if $i_s\rightarrow i_{s+1}$ is preceded by
another edge. Hence we may have only one edge $j\rightarrow i$
corresponding to $j<_D i$, in contradiction with $i<_D j$.

Assume now that $n_{ij}=2$ holds for all $i<j$ and that the
incomparability graph of $<_D$ is connected. In this case the edges of
negative weight in $D$ are obtained from the incomparability graph of
$<_D$ by replacing each undirected edge of the incomparability graph by
a pair of opposing directed edges. Since the incomparability graph is
connected, there is a directed path using edges of negative weight only
from any vertex to any vertex in $D$. 
\end{proof}  

\begin{example}
{\em Consider the Linial arrangement and the semiacyclic tournament $D$
containing a directed edge $i\leftarrow j$ of weight $-1$ for each
$i<j$. This is a valid $m$-acyclic weighted digraph, it is in fact acyclic. The
induced poset of gains is an antichain, the
incomparability graph is the complete graph, it is connected. However,
$D$ is not strongly connected. } 
\end{example}

We conclude this section with presenting a class of sparse deformations
of the braid arrangement for which the properties of the associated
valid $m$-acyclic weighted digraphs may be used to give an upper bound for the
number of regions. 
\begin{definition}
Let $\underline{a}=(a_1,a_2,\ldots,a_n)\in {\mathbb R}_{\geq 0}^n$ be a
vector of nonnegative real numbers. We
define the {$\underline{a}$-generalized Linial arrangement} as the set
of hyperplanes 
$$
x_i-x_j=a_i \quad \mbox{for $1\leq i<j\leq n$ in $V_{n-1}$.}
$$
\end{definition}  
Note that setting $a_1=a_2=\cdots=a_n=1$ yields the
Linial arrangement, whereas in the special case when the real numbers
$a_1,\ldots,a_n$ are algebraically independent, we obtain a {\em
  semigeneric arrangement} as defined in~\cite{Postnikov-Stanley}. In any valid associated weighted digraph $D$ there
is exactly one directed edge for each $i<j$:
\begin{enumerate}
\item either an edge $i\rightarrow j$, of weight $a_i$, corresponding to
  $x_i-x_j>a_i$: we call such an edge an {\em ascent};
\item or an edge $i\leftarrow j$ of weight $-a_i$: we call such an edge a {\em
  descent}.
  \end{enumerate}  
As a consequence, the weighted digraph $D$ may be uniquely reconstructed
from its underlying tournament. It depends on the values of the
parameters $a_i$ which tournaments contain no $m$-ascending cycle, but
there is a bijection between the regions of an
$\underline{a}$-generalized Linial arrangement and a set of tournaments
on $\{1,2,\ldots,n\}$. That said, certain tournaments may be excluded
regardless of the choice of the
parameters. Following~\cite{Hetyei-alta}, we call a cycle {\em
  alternating} if ascents and descents alternate in it. 
\begin{proposition}
If $D$ is a valid $m$-acyclic weighted digraph associated to an
$\underline{a}$-generalized Linial arrangement, then $D$ contains no
alternating cycle.  
\end{proposition}  
\begin{proof}
Given a directed cycle in a tournament on $\{1,2,\ldots,n\}$, let us call a
vertex $i$ a {\em peak} if an descent follows a ascent at $i$ along the
cycle, and let us call $i$ a {\em valley} if ascent follows a descent at
$i$. Peaks and valleys alternate along an alternating cycle, hence we
may compute the total weight of its edges by counting the contribution
of each edge at the valley incident to it. The incoming edge at a valley
$i$ has weight $-a_i$, the outgoing edge has weight $a_i$. Hence the
total weight of all edges is zero, regardless of the values of the
parameters $a_i$. An alternating cycle is $m$-ascending.   
\end{proof}
A tournament on the set $\{1,2,\ldots,n\}$ is {\em alternation acyclic}
if it contains no alternating cycle. 
It has been shown in~\cite[Theorem~4.4]{Hetyei-alta} that the number of
alternation acyclic tournaments on the set $\{1,2,\ldots,n\}$ is the median
Genocchi number $H_{2n-1}$.
\begin{corollary}
\label{cor:alta}  
The number of regions in any $\underline{a}$-generalized Linial
arrangement in $V_{n-1}$ is less than or equal to the median
Genocchi number $H_{2n-1}$. 
\end{corollary}   
\begin{remark}
{\em It has been shown in~\cite{Hetyei-alta} that there is a bijection
between alternation acyclic tournaments and the regions of the {\em
  homogenized Linial arrangement} whose hyperplanes are defined by
equations of the form $x_i-x_j=y_i$, where the $y_i$s are also
coordinate functions. Hence we may think of an
$\underline{a}$-generalized Linial arrangement as the intersection of
the homogenized Linial arrangement with the hyperplanes $y_i=a_i$ for 
$i=1,2,\ldots, n$. Corollary~\ref{cor:alta} may also be shown using
these observations.}  
\end{remark}

\begin{remark}
{\em The {\em
  bigraphical arrangements} introduced in~\cite{Hopkins-bigraphical} are
related to the sparse deformations discussed in this section.
Given a simple graph $G=(\{1,2\ldots,n\},E)$, for each edge $\{i,j\}$ we
choose real parameters $a_{ij}$  and $a_{ji}$, such
that there is an $x\in\mathbb{R}^n$ satisfying all inequalities of the form
$x_i-x_j < a_{ij}$ and $x_j - x_i < a_{ji}$. 
The bigraphical arrangement is the set of $2|E|$
hyperplanes $\{x_i - x_j = a_{ij} \::\: \{i, j\} \in E\}$.
Since $-a_{ji}<x_i-x_j<a_{ij}$ must have a solution,
$a_{ij}+a_{ji}>0$ must hold for each $\{i,j\}\in E$. 
The valid associated weighted digraphs may be described as follows. For
each edge $\{i,j\}\in E$, assuming $i<j$, exactly one of the following
possibilities hold: 
\begin{enumerate}
\item There a directed edge $i\leftarrow j$ of weight
  $a_{ji}$ (representing $x_i-x_j<-a_{ji}$). 
\item There is a directed edge $i\rightarrow j$ of weight $-a_{ji}$
  and a directed edge $i\leftarrow j$ of weight $-a_{ij}$
(representing $-a_{ji}<x_i-x_j<a_{ij}$). 
\item There a directed edge $i\rightarrow j$ of weight
  $a_{ij}$ (representing $x_i-x_j>a_{ij}$).  
\end{enumerate}
It is part of the definition of a bigraphical arrangement that selecting
option (2) on each $\{i,j\}\in E$  should yield an $m$-acyclic
digraph (labeling the {\em central region}).
The definition of a {\em partial orientation}
in~\cite{Hopkins-bigraphical} is equivalent to the above weighted
digraph labeling. The oriented edges in~\cite{Hopkins-bigraphical} are
directed in the opposite way, the unoriented edges correspond to our
pairs of opposing directed edges. The definition of $A$-admissibility is
equivalent to our $m$-acyclic condition (the {\em score} used
in~\cite{Hopkins-bigraphical}  is the negative of our weight) and
{\em potential cycles} are directed cycles using some directed edges
whose opposite is also present. The same observations also apply to the
{\em mixed graphs} appearing in~\cite{Beck-mixed}.}    
\end{remark}

\section{Separated deformations and the weak triangle inequality}
\label{sec:separated}

\begin{definition}
Let $\mathcal{A}$ be a deformation of the braid arrangement, given
by~\eqref{eq:dbraid}. We call the arrangement $\mathcal{A}$ 
{\em  separated} if $0$ belongs to the set $A_{ij}=\{a_{ij}^{(1)}, a_{ij}^{(2)},
\ldots, a_{ij}^{(n_{ij})}\}$ for each $1\leq i<j\leq n$.    
\end{definition}
\begin{proposition}
In any valid associated weighted digraph of a separated
deformation $\mathcal{A}$ of the braid arrangement given 
by~\eqref{eq:dbraid} either $w(i,j)\geq 0$ or $w(j,i)\geq 0$ holds
for each $1\leq i<j\leq n$.  
\end{proposition}
Indeed, by our assumptions $a_{ij}^{(1)}\leq 0$ (and hence $-a_{ij}^{(1)}\geq
0$) and $a_{ij}^{(n_{ij})}\geq 0$ hold, furthermore, for each
$k<n_{ij}$,  the numbers  $a_{ij}^{(k)}$ and $a_{ij}^{(k+1)}$ can not be
both nonzero real numbers, having opposite signs. The statement is a
direct consequence of Definition~\ref{def:valid}.  
\begin{corollary}
For a separated deformation of the braid arrangement, the induced poset of gains
associated to any valid $m$-acyclic weighted digraph is a totally ordered set. 
\end{corollary}
Hence there is a unique permutation $\sigma=\sigma(1)\sigma(2)\cdots
\sigma(n)$ of the coordinates such that $w(\sigma(i),\sigma(j))\geq 0$
holds for all $i>j$. 
\begin{definition}
We call the permutation $\sigma$ the {\em order of gains} associated to
the valid $m$-acyclic weighted digraph of a separated deformation of the
braid arrangement.  We will also use the notation $i<_{\sigma^{-1}} j$ to
indicate that the label $i$ precedes the label $j$ in $\sigma$. 
\end{definition}
So far we only rephrased the observation that for a separated
deformation of the braid arrangement, each region is included in a
region $x_{\sigma(1)}>x_{\sigma(2)}>\cdots >x_{\sigma(n)}$
of the braid arrangement. Using Theorem~\ref{thm:br} we may refine this
observation as follows.
\begin{theorem}
\label{thm:sbounded}  
Let ${\mathcal R}$ be a region of a separated deformation of the braid
arrangement and let $x_{\sigma(1)}>x_{\sigma(2)}>\cdots >x_{\sigma(n)}$
be the unique region of the braid arrangement containing it. Then there
is a unique sequence $1=i_0\leq i_1<i_2<\cdots <i_k=n$ such that 
writing $\sigma$ as the concatenation 
$$
\sigma=(\sigma(i_0)\cdots\sigma(i_1))\cdot
(\sigma(i_1+1)\cdots\sigma(i_2))
\cdots
(\sigma(i_{k-1}+1)\cdots\sigma(i_k))
$$
of contiguous subwords has the following properties: 
\begin{enumerate}
\item For each $j=-1,0,\ldots,k-1$ the intersection of ${\mathcal R}$
  with the linear span of $\{e_{\sigma(i_j+1)},e_{\sigma(i_j+2)},\ldots,
  e_{\sigma(i_{j+1})}\}$ is a bounded region.
\item If a subset $S$ of $\{1,2,\ldots,n\}$ contains indices $j_1$ and
  $j_2$ such that $\sigma(j_1)$ and $\sigma(j_2)$ belong to different
  subwords in the above   decomposition then the intersection of ${\mathcal R}$
  with the linear span of $\{e_{\sigma(j)}\::\: j\in S\}$ is unbounded.    
\end{enumerate}  
\end{theorem}  
\begin{proof}
By Theorem~\ref{thm:br}, for any nonempty subset $S$ of
$\{1,2,\ldots,n\}$ the intersection of ${\mathcal R}$ with the linear
span of $\{e_{\sigma(j)}\::\: j\in S\}$ is bounded if the
restriction of the weighted digraph $D_{\mathcal{R}}$ associated to
${\mathcal R}$ to the set $\{\sigma(j)\::\: j\in S\}$ is strongly
connected. Note that the converse is not necessarily true: inequalities implied
by edges not belonging to the restriction may force the intersection to
be bounded even if the restriction is not strongly connected. 

First we show the existence of such a decomposition. 
Let $i_1$ be the largest index such that there is a directed
path from $\sigma(i_1)$ to 
$\sigma(1)$ in $D_{\mathcal{R}}$. In particular we set $i_1=1$ if $\sigma(1)$
has indegree zero in $D_{\mathcal{R}}$. Keep in mind that the
definition of $V_{n-1}$ includes the equality $x_1+\cdots+x_n=0$, 
together with $x_{\sigma(2)}=\cdots=x_{\sigma(n)}=0$ this forces
$x_{\sigma(1)}=0$. Since there is a 
directed path $\sigma(1)\rightarrow \sigma(2)\rightarrow
\cdots \rightarrow \sigma(i_1)$ in $D_{\mathcal{R}}$, the restriction of
$D_{\mathcal{R}}$ to the set
$[\sigma(1),\sigma(i_1)]=\{\sigma(1),\sigma(2),\ldots,\sigma(i_1)\}$ is
strongly connected. Furthermore, given any $j_1\in [\sigma(1),\sigma(i_1)]$ and
any  $j_2\not\in [\sigma(1),\sigma(i_1)]$, no directed edge
$j_2\rightarrow j_1$ can exist as $\sigma(1)$ is not reachable from
$j_2$. Hence any subset of $[\sigma(1),\sigma(i_1)]$ and any subset of
its complement must belong to different strong components: once we leave
$[\sigma(1),\sigma(i_1)]$, there is no arrow we could use to return to
it. As seen in the proof of Theorem~\ref{thm:br}, this implies that all
(implied) differences of the form $x_{\sigma(j_1)}-x_{\sigma(j_2)}$ are only
bounded from below in the definition of ${\mathcal R}$ which remains
unbounded even after intersecting it with the linear span of
$\{e_{\sigma(j)}\::\: j\in S\}$ for any $S$ containing both $j_1$ and
$j_2$. We may continue finding a suitable $i_2$, $i_3$, \ldots $i_k$ in
a recursive fashion: for $j=1,2,\ldots,k-1$, we may define $i_{j+1}$ as
the largest index such that $\sigma(i_j)$ may be reached from
$\sigma(i_{j+1})$.   

To show the uniqueness of our decomposition, observe first that replacing $i_1$
with any larger $i_1'$ results in a subset
$[\sigma(1),\sigma(i_1')]$ such that the restriction of
$D_{\mathcal{R}}$ to $[\sigma(1),\sigma(i_1')]$ is not strongly
connected: for any $j_1\leq i_1$ and any $j_2>i_1$ there is only a directed edge
from $\sigma(j_1)$ to $\sigma(j_2)$. As above, using the proof of
Theorem~\ref{thm:br}, we may conclude that the
differences of the form $x_{\sigma(j_1)}-x_{\sigma(j_2)}$ are only
bounded from below in the definition of ${\mathcal R}$ and the
intersection of $\mathcal{R}$ with the linear span of
$\{e_{\sigma(j)}\::\: 1\leq j\leq i_1'\}$ is still unbounded. On the
other hand, for any $i_1'<i_1$ the set $[\sigma(1),\sigma(i_1)]$ properly
contains $[\sigma(1),\sigma(i_1')]$ and the restriction of 
$D_{\mathcal{R}}$ to $[\sigma(1),\sigma(i_1)]$ is strongly
connected. 
\end{proof}
It should be stressed that Theorem~\ref{thm:sbounded} above helps
identify some bounded regions in a separated deformation of the braid
arrangement, but does not describe all regions completely. For the Shi
arrangement the associated poset of gains identifies the {\em Weyl
  chamber}  containing a given region, and there is an extensive
literature on counting the regions just inside each Weyl chamber in a
Shi arrangement~\cite{Armstrong-Reiner-Rhoades,Dermenjian-Tzanaki,Dorpalen-Stump}.

Next we define the {\em gain
  function $g:\{1,2,\ldots,n\}\rightarrow {\mathbb R}$} associated to a
valid $m$-acyclic weighted digraph as follows.
\begin{definition}
For each $i\in \{1,2,\ldots,n\}$ we define the {\em gain function}
$g(\sigma(i))$ as the maximum weight of a directed path beginning at
$\sigma(1)$ and ending at $\sigma(i)$. In particular, we set
$g(\sigma(1))=0$. Here $\sigma$ is the total order of gains.    
\end{definition}  
\begin{remark}
{\em The word ``gain'' here is used as the opposite of ``cost'' in a weighted
directed graph setting and it is not related to the notion of a gain
graph~\cite{Berthome-etc,Zaslavsky-perpendicular}.} 
\end{remark}  

\begin{lemma}
  Every gain function has the weakly increasing property
  $$
g(\sigma(1))\leq g(\sigma(2))\leq\cdots \leq g(\sigma(n)).
  $$
\end{lemma}
Indeed, for each $i<j$, every maximum weight directed path from
$\sigma(1)$ to $\sigma(i)$ may be extended to a directed walk to
$\sigma(j)$ by adding the directed edge $\sigma(i)\rightarrow \sigma(j)$
of nonnegative weight. Eliminating cycles from the walk can only
increase the total weight, hence we have
\begin{equation}
\label{eq:gbound}  
g(\sigma(j))\geq g(\sigma(i))+w(\sigma(i),\sigma(j))\geq g(\sigma(i)).  
\end{equation}
The question naturally arises: could the maximum gain be achieved by
using directed edges of nonnegative weight only? In the rest of this
section we show that the answer is yes at least in an important special
case.
\begin{definition}
We call a deformation $\mathcal{A}$ of the braid arrangement {\em
  integral} if all the numbers $a_{ij}^{(k)}$ appearing in
\eqref{eq:dbraid} are integers. We say that $\mathcal{A}$ satisfies the
      {\em weak triangle inequality} if for all triplets $(i,j,k)$, the
      inequalities $w(i,j)\geq 0$ and $w(j,k)\geq 0$ imply
      $$
w(i,k)\leq w(i,j)+w(j,k)+1
$$
in any valid $m$-acyclic associated weighted digraph.
\end{definition}
\begin{theorem}
\label{thm:posweight}  
Let $\mathcal{A}$ be a separated integral deformation of the braid
arrangement satisfying the weak triangle inequality, and let $D$ be an
associated $m$-acyclic weighted digraph. Let $\sigma$ be the total order
of gains associated to $D$ and let $g$ be the gain function. Then, for
each $i>1$ there is a directed path from $\sigma(1)$ to $\sigma(i)$ such
that all weights in the path are nonnegative and the total weight of the
edges in the path is $g(\sigma(i))-g(\sigma(1))$.  
\end{theorem}  
\begin{proof}
Consider a weighted path $\sigma(1)=i_0\rightarrow i_1\rightarrow
\cdots\rightarrow i_m=\sigma(i)$ from $\sigma(1)$ to $\sigma(i)$ such
that the weight of the path is $g(\sigma(i))-g(\sigma(1))$ and it has the least possible
number of negative weighted edges. We are done if there is no edge of
negative weight, otherwise assume that $i_{k}\rightarrow i_{k+1}$ is the
first edge satisfying $w(i_k,i_{k+1})<0$. All preceding steps being of
nonnegative weight, we must have $i_0<_{\sigma^{-1}} i_1 <_{\sigma^{-1}}
\cdots <_{\sigma^{-1}} i_k$. The inequality $w(i_k,i_{k+1})<0$ implies
$i_{k+1}<_{\sigma^{-1}} i_k$, hence there is a unique $j\leq k-1$ such
that $i_j <_{\sigma^{-1}} i_{k+1} <_{\sigma^{-1}} i_{j+1}$
holds. (Recall that, the $m$-acyclic property implies that walks
revisiting a vertex can only have lower weight than the path obtained by
eliminating the closed subwalks, hence the inequalities must be strict.)
Applying the weak triangle inequality to he triple $(i_j,i_{k+1},i_{j+1})$ we
obtain
$$
w(i_j,i_{j+1})\leq w(i_j,i_{k+1})+w(i_{k+1},i_{j+1})+1.
$$
Let us add the weight $\walkw(i_{j+1},i_{j+2},\ldots,i_k,i_{k+1})$ of the walk
$i_{j+1}\rightarrow i_{j+2}\rightarrow \cdots\rightarrow i_k\rightarrow
i_{k+1}$ to both sides. On the left hand side we obtain the weight of the walk 
$i_{j}\rightarrow i_{j+1}\rightarrow \cdots\rightarrow i_k\rightarrow
i_{k+1}$:
$$
\walkw(i_j,i_{j+1},\ldots i_k,i_{k+1})\leq
w(i_j,i_{k+1})+\walkw(i_{j+1},i_{j+2},\ldots,i_k,i_{k+1})+w(i_{k+1},i_{j+1})+1. 
$$
The sum $\walkw(i_{j+1},i_{j+2},\ldots,i_k,i_{k+1})+w(i_{k+1},i_{j+1})$ on
the right hand side is the weight of the closed walk
$i_{j+1}\rightarrow i_{j+1} \rightarrow i_{j+1} \rightarrow i_{j+2}
\rightarrow \cdots \rightarrow i_k \rightarrow i_{k+1} \rightarrow
i_{j+1}$, which is negative by the $m$-acyclic property, and it is at
most $-1$ by the integrality of $\mathcal{A}$. Thus we obtain 
$$
\walkw(i_j,i_{j+1},\ldots i_k,i_{k+1})\leq
w(i_j,i_{k+1}),
$$
which means that we may replace the subpath $i_{j}\rightarrow
i_{j+1}\rightarrow \cdots\rightarrow i_k\rightarrow i_{k+1}$ with the
nonnegative edge $i_j\rightarrow i_{k+1}$, thus decreasing the number of
negative edges without decreasing the total weight, in contradiction
with our assumptions. 
\end{proof}   
An example of a saturated integral arrangement in which
$g(\sigma(i))-g(\sigma(1))$ is attained only by a path containing a
negative weighted edge is given in Example~\ref{ex:negative}. 
{\em From now on until the rest of the section we only consider separated
integral deformations of the braid arrangement satisfying the weak
triangle inequality.}  
If we know the total order of gains, Theorem~\ref{thm:posweight} allows
us to compute the gain function in a greedy fashion:
\begin{enumerate}
\item We set $g(\sigma(1))=0$.
\item Once we computed $g(\sigma(1))), g(\sigma(2)),\ldots,
  g(\sigma(i-1))$, the value of $g(\sigma(i))$ is 
\begin{equation}
\label{eq:grec}    
g(\sigma(i))=\max_{1\leq j\leq i-1} (g(\sigma(j))+w(\sigma(j),\sigma(i))).
\end{equation}
\end{enumerate}  
The choice of $j$ in equation~(\ref{eq:grec}) may not be unique. We
eliminate the ambiguity by always selecting the rightmost possible label
$\sigma(j)$: 
\begin{definition}
\label{def:treedef}
We select the largest $j<i$ satisfying~(\ref{eq:grec}) and call the
resulting $\sigma(j)$ {\em the parent of $\sigma(i)$}, denoted by
$p(\sigma(i))$. We extend the definition to $\sigma(1)$ by setting
$p(\sigma(1))=\sigma(1)$. 
\end{definition}   
Note that the pairs $(\sigma(i),p(\sigma(i)))$ form the edges of a tree
rooted at $\sigma(1)$. We call this rooted tree the {\em tree of the
  gain function}. In the study of the properties of this tree the
following lemma plays a key role. 
\begin{lemma}
\label{lem:parentf}  
  If $i<_{\sigma^{-1}} j$ and $p(j)<_{\sigma^{-1}} i$ hold then we have
  $w(i,j)=g(j)-g(i)-1$, $w(p(j),i)=g(i)-g(p(j))$ and
  $p(j)\leq_{\sigma^{-1}} p(i)$. 
\end{lemma}  
\begin{proof}
  By our assumptions $i$ is to the right of $p(j)$, hence we have
  $$w(i,j)\leq g(j)-g(i)-1.$$
  By~\eqref{eq:gbound} we also have 
  $$w(p(j),i)\leq g(i)-g(p(j)).$$
The sum of the two inequalities, combined with the weak triangle
inequality yields
$$
g(j)-g(p(j))=w(p(j),j)\leq w(p(j),i)+w(i,j)+1\leq g(j)-g(p(j)).
$$
The left and the right end being equal all inequalities above must be
equalities. Hence we have $w(i,j)=g(j)-g(i)-1$ and
$w(p(j),i)=g(i)-g(p(j))$. Finally $p(j)\leq_{\sigma^{-1}} p(i)$ is a
direct consequence of $w(p(j),i)=g(i)-g(p(j))$.     
\end{proof}  
\begin{proposition}
The edges of the tree of the gain function are noncrossing: there is no
$i_1<_{\sigma^{-1}} i_2$ such that $p(i_1)<_{\sigma^{-1}}
p(i_2)<_{\sigma^{-1}} i_1<_{\sigma^{-1}} i_2$ would hold.  
\end{proposition} 
\begin{proof}
Assume the contrary. By Lemma~\ref{lem:parentf}, $i_1<_{\sigma^{-1}} i_2$
and $p(i_2)<_{\sigma^{-1}} i_1$ imply $p(i_2)\leq_{\sigma^{-1}} p(i_1)$,
in contradiction with $p(i_1)<_{\sigma^{-1}} p(i_2)$. 
\end{proof}  
\begin{corollary}
\label{cor:Catalan}  
The number of possible types of trees of the gain function is a Catalan
number.  
\end{corollary}

\section{Contiguous integral deformations}
\label{sec:contiguous}

\begin{definition}
We call an integral deformation of the braid arrangement in $V_{n-1}$ {\em
  contiguous}  if, for every $i<j$, the set $A_{ij}=\{a_{ij}^{(1)},
a_{ij}^{(2)}, \ldots, a_{ij}^{(n_{ij})}\}$ is a contiguous set
$[\alpha(i,j),\beta(i,j)]=\{\alpha(i,j),\alpha(i,j)+1,\ldots,\beta(i,j)\}$
of integers. 
\end{definition}  

Since the equation $x_i-x_j=c$ is equivalent to the equation
$x_j-x_i=-c$, we may consistently extend our notation by setting
\begin{equation}
  \alpha(j,i)=-\beta(i,j) \quad\mbox{and}\quad \beta(j,i)=-\alpha(i,j)
  \quad\mbox{for $1\leq i<j\leq n$.}
\end{equation}
The truncated affine arrangements $\mathcal{A}^{a,b}_{n-1}$
are contiguous deformations of the braid
arrangement: we have $\alpha(i,j)=1-a$ and $\beta(i,j)=b-1$  for $1\leq
i<j\leq n$.

We may specialize Definition~\ref{def:valid} to such arrangements as follows.
\begin{proposition}
  \label{prop:validt}
Let $\mathcal{A}$ be a contiguous integral deformation of the braid
arrangement in $V_{n-1}$, defined by the intervals
$[\alpha(i,j),\beta(i,j)]$. A valid associated digraph on the vertex set
$\{1,2,\ldots,n\}$  satisfies for each $(i,j)$ exactly one of the following:
\begin{enumerate}
\item There is no directed edge $i\rightarrow j$, and there
  is a directed edge $i\leftarrow j$ of weight $-\alpha(i,j)$.
\item There is a directed edge $i\rightarrow j$ of weight
  $w$, and there is a directed edge $i\leftarrow j$ of
  weight $-w-1$, for some for some integer $w\in [\alpha(i,j),\beta(i,j)-1]$. 
\item There is a directed edge $i\rightarrow j$ of weight
  $\beta(i,j)$, and there is no directed edge $i\leftarrow j$.
\end{enumerate}  
\end{proposition}
As we did right after Definition~\ref{def:valid}, we extend the weight
function to a function  
$w: \{1,2,\ldots,n\}\times \{1,2,\ldots,n\} \rightarrow {\mathbb R} \cup
\{-\infty\}$ by setting the following for all $i<j$: 
\begin{enumerate}
\item The weight $w(i,j)$ belongs to the set
  $[\alpha(i,j),\beta(i,j)]\cup\{-\infty\}$, and the weight $w(j,i)$ belongs
  to the set
  $[-\beta(i,j),-\alpha(i,j)]\cup\{-\infty\}=[\alpha(j,i),\beta(j,i)]\cup\{-\infty\} 
  $. 
\item If $w(i,j)=-\infty$ then $w(j,i)=-\alpha(i,j)=\beta(j,i)$ and if
  $w(i,j)=\beta(i,j)$ then $w(j,i)=-\infty$. 
\item For all other values of $w(i,j)$ we have $w(j,i)=-1-w(i,j)$. 
\end{enumerate}  
In other words, regardless of the order of $i$ and $j$,
$w(i,j)$ belongs to the set
  $[\alpha(i,j),\beta(i,j)]\cup\{-\infty\}$, and $w(i,j)=-\infty$ holds
if and only if $w(j,i)=\beta(j,i)$.

As we will see in Theorem~\ref{thm:a-bleq1} below, the $m$-acyclic
property is especially easy to check for truncated affine arrangements
satisfying $|a-b|\leq 1$. The following lemma still holds for all
contiguous integral deformations.

\begin{lemma}
  \label{lem:diagonals}
Let $C=(i_0,i_1,\ldots,i_k)$ be an $m$-ascending cycle of minimum length
in a valid weighted digraph associated to a contiguous integral deformations of
the braid arrangement. Then for any {\em diagonal pair} of
vertices of $C$, that is, for any  $\{i_r,i_s\}$ such that $i_r$ and $i_s$ are
not cyclically consecutive, the associated weight function satisfies
$-\infty\in\{w(i_r,i_s),w(i_s,i_r)\}$. 
\end{lemma}  
\begin{proof}
Assume by contradiction that there is a pair $\{i_r,i_s\}$ of
cyclically not consecutive vertices such that
$-\infty\not\in\{w(i_r,i_s),w(i_s,i_r)\}$. In this case the real numbers
$w(i_r,i_s)$ and $w(i_s,i_r)$ must satisfy $w(i_r,i_s)+w(i_s,i_r)=-1$.
The cyclic lists $(i_r,i_{r+1},\ldots,i_s)$ and
$(i_s,i_{s+1},\ldots,i_r)$ represent shorter closed
walks in the associated weighted digraph, which can not be
$m$-ascending. Since all weights are integers we must have
$$
w(i_r,i_{r+1},\ldots,i_{s})\leq
-1\quad\mbox{and}\quad 
w(i_s,i_{s+1},\ldots,i_r)\leq
-1. 
$$
Taking the sum of these inequalities, after subtracting
$w(i_r,i_s)+w(i_s,i_r)=-1$ on both sides yields
$$
w(C)
=w(i_r,i_{r+1},\ldots,i_{s})+w(i_s,i_{s+1},\ldots,i_r)
-(w(i_r,i_s)+w(i_s,i_r))\leq-1 
$$
in contradiction with the closed walk represented by $C$ being
$m$-ascending. 
\end{proof}  

The next theorem is a generalization of~\cite[Theorem
  8.6]{Postnikov-Stanley}.  

\begin{theorem}
\label{thm:4gen}  
Consider a contiguous integral deformation of the braid arrangement in
$V_{n-1}$ satisfying
\begin{equation}
\label{eq:case1}
\beta(i,k)\leq \beta(i,j)+\beta(j,k)+1 \quad\mbox{for all triplets $\{i,j,k\}$.}
\end{equation}
 
Then any valid associated weighted digraph is $m$-acyclic if and only if
it contains no $m$-ascending cycle of length at most four. 
\end{theorem}  
\begin{proof}
It suffices to show that
there is no valid associated weighted digraph in which a shortest $m$-ascending
cycle of length $k\geq 5$ would exist. Assume by contradiction that 
$C=(i_0,i_1,\ldots,i_{k-1})$ is such an is an $m$-ascending cycle.
Recall that, by Lemma~\ref{lem:diagonals}, only one of the directed edges
$i_r\rightarrow i_s$ and $i_{r}\leftarrow i_{s}$ is present for each
diagonal pair $\{i_r,i_s\}$. Let us call $i_s\in
\{i_0,i_1,\ldots,i_{k-1}\}$ a {\em diagonal tail (diagonal head) vertex} if
$i_s\rightarrow i_r$ ($i_s\leftarrow i_r$) holds for some diagonal pair
$\{i_r,i_s\}$. Clearly, for $k\geq 4$, each vertex $i_s$ is either
a diagonal tail or a diagonal head vertex, some of them may be both. If 
each $i_s\in\{i_0,i_1,\ldots,i_{k-1}\}$ is either only a diagonal
tail or only a diagonal head then we may reach a contradiction as follows.
At least one $i_r$ must be a diagonal head, without loss of generality
we may assume $r=0$, hence we must have $i_2\rightarrow i_0$,
$i_3\rightarrow i_0$, \ldots, $i_{k-2}\rightarrow i_0$, as a consequence
$i_2$ and $i_{k-2}$ are diagonal tails and by $k\geq 5$ the pair
$\{i_2,i_{k-2}\}$ is a diagonal pair for which neither $i_2\rightarrow
i_{k-2}$ nor $i_2\leftarrow i_{k-2}$ can hold, in contradiction with
Lemma~\ref{lem:diagonals}. Hence there is at least one vertex 
$i_r\in \{i_0,i_1,\ldots,i_{k-1}\}$ which is simultaneously a diagonal head
and a diagonal tail: $i_r\leftarrow i_s$ and $i_r\rightarrow i_{s'}$
hold for some $s,s'\in\{0,1,\ldots,k-1\}-\{r-1,r,r+1\}$. (Additions
and subtractions in the subscripts are performed modulo $k$.) 
Without loss of generality we may assume that $s$ and $s'$ are
cyclically consecutive. 

{\bf\noindent Case 1:} $s'=s+1$ holds. The cycles
$(i_r,i_{r+1},\ldots,i_s)$ and $(i_{s'},i_{s'+1},\ldots,\ldots i_{r-1},i_r)$
are shorter than $C$, hence they can not be $m$-ascending:
$$
w(i_r,i_{r+1},\ldots,i_s)\leq
-1\quad\mbox{and}\quad
w(i_{s'},i_{s'+1},\ldots,i_{r})\leq -1
\quad\mbox{must hold}.
$$
Adding $w(i_s,i_{s'})$ on both sides to the sum of the above
inequalities we obtain
$$
w(i_s,i_r)+w(i_r,i_s')+w(C)\leq -2+w(i_s,i_{s'}).
$$
Since $w(C)$ is nonnegative, we obtain 
\begin{equation}
\label{eq:short}  
  w(i_s,i_r)+w(i_r,i_{s'})\leq -2+w(i_s,i_{s'}).
\end{equation}  
Since $\{i_s,i_r\}$ and $\{i_r,i_{s'}\}$ are diagonals, we have
$w(i_s,i_r)=\beta(i_s,i_r)$ and  $w(i_r,i_{s'})=\beta(i_r,i_{s'})$.
Furthermore $w(i_s,i_{s'})$ is at most
$\beta(i_s,i_{s'})$. Equation~\eqref{eq:short} implies
$$
  \beta(i_s,i_r)+\beta(i_r,i_{s'})\leq -2+\beta(i_s,i_{s'}),
$$
in contradiction with \eqref{eq:case1}.   

{\bf\noindent Case 2:} $s'=s-1$ holds. This case is analogous
to the previous one. Now the cycles $(i_r,i_{r+1},\ldots,i_s)$ and
$(i_{s'},i_{s'+1},\ldots,\ldots i_{r-1},i_r)$ both contain the directed
edge $i_{s'}\rightarrow i_s$, and, in analogy to (\ref{eq:short}) we
obtain the inequality  
\begin{equation}
\label{eq:short2}  
  w(i_s,i_r)+w(i_r,i_s')\leq -2-w(i_{s'},i_s).
\end{equation}
Once again, $\{i_s,i_r\}$ and $\{i_r,i_{s'}\}$ are diagonals, hence the
left hand side of \eqref{eq:short2} is
$\beta(i_s,i_r)+\beta(i_r,i_{s'})$ whereas the right hand side is at
most $-2-\alpha(i_{s'},i_s)=-2+\beta(i_s,i_{s}')$. Once again, we obtain
  a contradiction with \eqref{eq:case1}. 

\end{proof}

\begin{theorem}
  \label{thm:4}
If the truncated affine arrangement $\mathcal{A}^{a,b}_{n-1}$ satisfies   
$a,b\geq 0$, then a valid associated weighted digraph is $m$-acyclic if and only if
it contains no $m$-ascending cycle of length at most four. 
\end{theorem}  
\begin{proof}
As noted in Remark~\ref{rem:symmetry}, without loss of generality we may 
assume that $a\leq b$ holds. By Theorem~\ref{thm:4gen}, it suffices to
verify the validity of the inequality \eqref{eq:case1}.
We distinguish several cases depending on the relative order of
$i$, $j$ and $k$. If $i<j<k$ holds then
$\beta(i,j)=\beta(j,k)=\beta(i,k)=b-1$, and Equation~\eqref{eq:case1} is
equivalent to  $b-1\leq 2(b-1)+1$, which is equivalent to $b\geq 0$.  

If $i>j>k$ holds then $\beta(i,j)=\beta(j,k)=\beta(i,k)=a-1$, and
Equation~\eqref{eq:case1} is equivalent to $a-1\leq 2(a-1)+1$, which is
equivalent to $a\geq 0$.   

In all other subcases at least one of $\beta(i,j)$ and $\beta(j,k)$ is
$b-1$, the right hand side of (\ref{eq:case1}) is at least $(a-1)+(b-1)+1$,
and the left hand side of (\ref{eq:case1}) is at most $b-1$. The
inequality $b-1\leq (a-1)+(b-1)+1$ is equivalent to $a\geq 0$.
\end{proof}

It may be possible to extend Theorem~\ref{thm:4} to further cases, but
this can not be done by inspecting the value of $\min(a,b)=a$ alone. The
next two examples 
illustrate this fact and indicate some of the 
difficulties in discussing the truncated affine arrangements which do
not fall under the validity of Theorem~\ref{thm:4}.

\begin{example}
\label{ex:a=-1b=4}
      {\em In the case when $a=-1$ and $b=4$ there is a valid
        weighted digraph associated to the truncated affine arrangement
        $\mathcal{A}^{-1,4}_{4}$ whose shortest 
        $m$-ascending cycle has length $5$. An 
        example is shown in Figure~\ref{fig:m5}. This weighted digraph has $5$
        vertices and for each $i,j$ we have either $i<j$, $w(i,j)=b-1=3$
        and $w(j,i)=-\infty$, or we have $i>j$, $w(i,j)=a-1=-2$
        and $w(j,i)=-\infty$. In other words, only one of the arrows
        $i\rightarrow j$ or $j\rightarrow i$ is present, the digraph is
        a tournament. The weight of the directed cycle $C=(1,3,5,4,2)$
        is $2\cdot 3+3\cdot (-2)=0$, this cycle is $m$-ascending. To show
        that all other directed cycles have         negative weight,
        consider the following ``potential function'' $\phi$ on the vertices:
        $\phi(1)=0$, $\phi(2)=2$, $\phi(3)=3$, $\phi(4)=4$ and
        $\phi(5)=6$. The values 
        of this function are the circled numbers in
        Figure~\ref{fig:m5}. For any arrow $i\rightarrow j$ that belongs to 
$C$, we have $w(i,j)=\phi(j)-\phi(i)$ and for all other arrows $i\rightarrow
        j$ we have $w(i,j)<\phi(j)-\phi(i)$. Hence any cycle containing at
        least one directed edge that does not belong to $C$ must have
        negative weight. 
      }
\end{example}  
\begin{figure}[h]
\centering
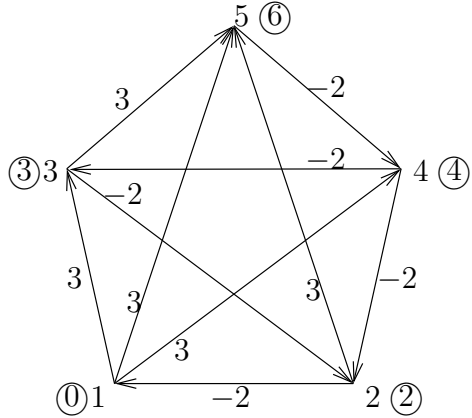
\caption{A shortest $m$-ascending cycle of length $5$}
\label{fig:m5}
\end{figure}  
\begin{example}
\label{ex:a=-1b=3}
      {\em In the case when $a=-1$ and $b=3$, the truncated affine
        arrangement $\mathcal{A}^{-1,3}_{n-1}$ consists of the
        hyperplanes $x_i-x_j= 2$ for $1\leq i<j\leq n$. A
        dilation by a factor of $1/2$ puts the regions of
        this hyperplane arrangement in bijection with the regions of
        the Linial arrangement $\mathcal{A}^{0,2}_{n-1}$ to which
        Theorem~\ref{thm:4}, as well as~\cite[Theorem
  8.6]{Postnikov-Stanley}, are applicable.}
\end{example}  
        
Looking at the formulas for the number of regions and the characteristic
polynomial $\mathcal{A}^{ab}_{n-1}$ in~\cite{Postnikov-Stanley},
counting regions in a combinatorial way promises to be easier in the
cases when $|a-b|\leq 1$ and $a+b\geq 2$ hold for the parameters $a$ and
$b$. As we will see in Proposition~\ref{prop:trawt}  below, these
are exactly the cases when the truncated 
affine arrangement is separated and satisfies the weak triangle
inequality (as defined in Section~\ref{sec:separated}). As a direct
consequence of the definitions we obtain:
\begin{corollary}
\label{cor:contsep}
A contiguous integral deformation of the braid arrangement in $V_{n-1}$
is separated if and only if $0$ belongs to $[\alpha(i,j),\beta(i,j)]$
for all $1\leq i< j\leq n$.
\end{corollary}
\begin{corollary}
\label{cor:trasep}
The truncated affine arrangement $\mathcal{A}^{ab}_{n-1}$ satisfying
$a\leq b$ and $a+b\geq 2$ is separated if and only if $a\geq 1$ holds.
\end{corollary}   
Next we characterize the cases when the weak triangle inequality is
satisfied by a separated contiguous integral deformation of the braid
arrangement.  
\begin{theorem}
\label{thm:wtc}  
Consider a separated contiguous integral deformation of the braid
arrangement in $V_{n-1}$. This 
satisfies the weak triangle inequality if and only if
\begin{equation}
\label{eq:wtc}  
\beta(i,j)\leq \beta(i,k)+1 \quad\mbox{and}\quad
\beta(i,j)\leq\beta(k,j)+1
\end{equation}
hold for all $\{i,j,k\}\subseteq \{1,2,\ldots,n\}.$
\end{theorem}  
\begin{proof}
Assume first that~\eqref{eq:wtc} is satisfied and let us compare
$w(i,k)$ with $w(i,j)+w(j,k)+1$, where we assume that the weights
$w(i,j)$, $(j,k)$  and $w(i,k)$ are nonnegative. If the reverse arrows
$i\leftarrow j$ (of weight $-1-w(i,j)$) and $j\leftarrow k$ (of weight
$-1-w(j,k)$) both exist, then the fact that the cycle $(i,k,j)$ is
$m$-acyclic implies
$$
-1-w(i,j)-1-w(j,k)+w(i,k)\leq -1,
$$
which is equivalent to the weak triangle inequality
$$
w(i,k)\leq w(i,j)+w(j,k)+1.
$$
We are left to consider the cases when at least one of the arrows
$i\leftarrow j$ and $j\leftarrow k$ does not exist. In this case either
$w(i,j)=\beta(i,j)$ or $w(j,k)=\beta(j,k)$. In either case the sum
$w(i,j)+w(j,k)+1$ is at least $\min(\beta(i,j),\beta(j,k))+1$, whereas
$w(i,k)$ is at most $\beta(i,k)$. The weak triangle inequality follows
from~\eqref{eq:wtc}.

Next we prove the contrapositive of the converse. Assume
that~\eqref{eq:wtc} fails for some $\{i,j,k\}$, without loss of
generality we may assume that $\beta(i,j)>\beta(i,k)+1$ holds. We
construct a valid associated weighted digraph violating the weak triangle
inequality as follows.We set $w(i,j)=\beta(i,j)$ (hence there is no
arrow $i\leftarrow j$) and $w(i,k)=\beta(i,k)$ (hence there is no
arrow $i\leftarrow k$). We set $w(j,k)=0$ and $w(k,j)=-1$. Note that
there is no $m$-ascending cycle on the restriction of our weighted digraph to
$\{i,j,k\}$ as the only cycle is $(j,k)$, which has weight $(-1)$. We
fix a linear order on $\{1,2,\ldots,n\}$ in such a way that $i$, $j$,
$k$ are the smallest vertices in this order. For any pair of vertices
$\{i',j'\}$ not
contained in $\{i,j,k\}$ we set $w(i',j')=\beta(i',j')$ where $i'$ is
the smaller vertex in our order. This choice guarantees that there is no
arrow $i'\leftarrow j'$, hence our weighted digraph does not contain any other
cycle than $(j,k)$. 
\end{proof}  

\begin{proposition}
\label{prop:trawt}  
Assume the truncated affine arrangement $\mathcal{A}^{ab}_{n-1}$
satisfies $1\leq a\leq b$ and $n\geq 3$. This integral and separated arrangement
satisfies the weak triangle inequality if and only of $b\leq a+1$ holds.
\end{proposition}  
\begin{proof}
The arrangement is separated by Corollary~\ref{cor:trasep}. Assume first
that $b\leq a+1$ holds. For all $\{i,j,k\}$ the value of $\beta(i,j)$ is
at most $b-1$, whereas the value of $\beta(i,k)$ and of $\beta(k,j)$ is at
least $a-1$. The inequality~\eqref{eq:wtc} is a direct consequence of
$b\leq a+1$.

Conversely, if $a\leq b-2$ holds, consider $i=2$, $j=3$ and $k=1$.
Then $\beta(i,j)=\beta(2,3)=b-1$, $\beta(i,k)=\beta(2,1)=a-1$ and
$\beta(2,3)\leq \beta(2,1)+1$ fails as $b-1\not\leq (a-1)+1$. 
\end{proof}  

\begin{example}
\label{ex:negative}  
{\em Consider the truncated affine arrangement $\mathcal{A}^{13}_{2}$ and the
associated $m$-acyclic weighted digraph given by $w(1,2)=-1$,
$w(2,1)=0$, $w(1,3)=0$, 
$w(3,1)=-1$, $w(2,3)=2$ and $w(3,2)=-\infty$. The arrangement is
saturated, the total order 
of gains is $\sigma=213$. Not only the weak triangle inequality fails
because of $w(2,3)>w(2,1)+w(1,3)+1$, but the largest weight path form
$\sigma(1)=2$ to $\sigma(2)=1$ is $2\rightarrow 3\rightarrow 1$, which
contains an edge of negative weight.}     
\end{example}  
For truncated affine arrangements, a combinatorial approach to the case
$1\leq a\leq b\leq a+1$ is further facilitated by the following result. 
\begin{theorem}
\label{thm:a-bleq1}  
Consider a truncated affine arrangement $\mathcal{A}^{ab}_{n-1}$
satisfying $1\leq a\leq b\leq a+1$. A valid associated weighted digraph is
$m$-acyclic if and only if it contains no $m$-ascending cycle of length three.
\end{theorem}  
\begin{proof}
  As a consequence of Theorem~\ref{thm:4}, it suffices
  to show that there is no $ab$-weighted digraph in which a shortest
  $m$-ascending cycle of length $4$ would exist. 
Assume by contradiction that $(i_0,i_1,i_2,i_3)$ is such a
cycle. By Lemma~\ref{lem:diagonals}, for each diagonal pair
$\{i_r,i_s\}$, only one of the directed edges $i_r\rightarrow i_s$ and
$i_{r}\leftarrow i_{s}$ is present. After cyclic rotation of the labels,
if necessary, we may assume that $i_0\rightarrow i_2$ and $i_1\rightarrow
i_3$ are the directed edges present in our weighted digraph. By our
assumption, the directed cycles $(i_0,i_2,i_3)$ and $(i_1,i_3,i_0)$ are
not $m$-acyclic, we have
$$
w(i_0,i_2,i_3)\leq -1 \quad\mbox{and}\quad w(i_1,i_3,i_0)\leq -1.
$$
In analogy to the derivation of~(\ref{eq:short}), adding $w(i_1,i_2)-w(i_3,i_0)$
on both sides of the sum of the two inequalities yields
$$
w(i_0,i_2)+w(i_1,i_3)+w(i_0,i_1,i_2,i_3)\leq -2+w(i_1,i_2)-w(i_3,i_0).
$$
Since $(i_0,i_1,i_2,i_3)$ is $m$-ascending, the last inequality implies  
$$
w(i_0,i_2)+w(i_1,i_3)\leq -2+w(i_1,i_2)-w(i_3,i_0).
$$
The left hand side is at least $2(a-1)$, the right hand side is at most
$-2+(b-1)+(b-1)$. Hence we obtain
$$
2(a-1)\leq -2+(b-1)+(b-1)
$$
which is equivalent to $a\leq b-1$. Together with $a\geq b-1$ we obtain
$a=b-1$ and all inequalities used must be equalities. But this implies
$i_0>i_2>i_1>i_3>i_0$, a contradiction. 
\end{proof}
  
Theorem~\ref{thm:a-bleq1} may be extended to contiguous integral
deformations of the braid arrangement using the following auxiliary
terminology.

\begin{definition}
A separated contiguous integral deformation of the braid arrangement in
$V_{n-1}$ satisfying the 
weak triangle inequality has {\em a crossing $\beta$-pattern} if there is a
$4$-element subset $\{i_1,i_2,j_1,j_2\}$ of $\{1,2,\ldots,n\}$ such that
\begin{equation}
\label{eq:crossingp}
\beta(i_1,j_1)<\beta(i_1,j_2)\quad\mbox{and}\quad \beta(i_2,j_1)>\beta(i_2,j_2)
\quad\mbox{hold}.   
\end{equation}
\end{definition}

\begin{theorem}
Let $\mathcal{A}$ be a  separated  contiguous integral deformation of
the braid arrangement in
$V_{n-1}$ satisfying the weak triangle inequality. Then $\mathcal{A}$
has a valid associated weighted digraph containing a minimal
$m$-ascending cycle of length $4$, if and only if $\mathcal{A}$ has a
crossing $\beta$-pattern. 
\end{theorem}  
\begin{proof}
Assume first that $\mathcal{A}$ has a crossing $\beta$-pattern, that is,
a set $\{i_1,i_2,j_2,j_3\}$ satisfying~\eqref{eq:crossingp}. Introducing
$\beta=\beta(i_1,j_1)$, as a consequence of Theorem~\ref{thm:wtc}
and~\eqref{eq:crossingp} we obtain $\beta(i_1,j_2)=\beta+1$. Similarly,
introducing  $\beta'=\beta(i_2,j_1)$, Theorem~\ref{thm:wtc}
and~\eqref{eq:crossingp} yield
$\beta(i_1,j_2)=\beta'+1$. Theorem~\ref{thm:wtc} also implies
$\beta(i_1,j_2)-\beta(i_2,j_2)\leq 1$, that is, $\beta+1-\beta'\leq 1$      
and $\beta(i_2,j_1)-\beta(i_1,j_1)\leq 1$, that is,
$\beta'+1-\beta\leq 1$. We obtain that $\beta-\beta'$ and $\beta'-\beta$
are both at most zero, hence $\beta$ must equal $\beta'$.
\begin{figure}[h]
\centering
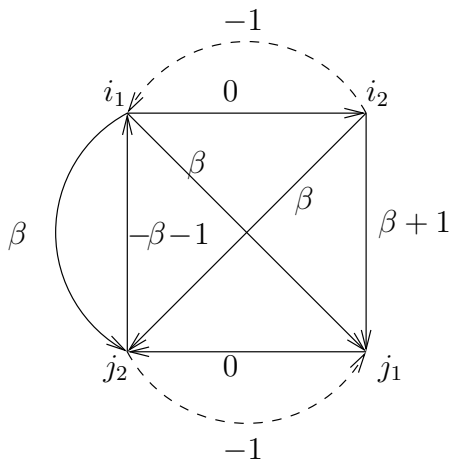
\caption{A valid weighted digraph with a minimal $m$-ascending $4$-cycle}
\label{fig:crossingbeta}
\end{figure}  
Consider the weighted digraph shown in
Figure~\ref{fig:crossingbeta}. The weights $w(i_1,j_1)=\beta(i_1,j_1)$,
$w(i_2,j_2)=\beta(i_2,j_2)$ and $w(i_2,j_1)=\beta(i_2,j_1)$ are largest
possible, hence there is no arrow $i_1\leftarrow j_1$, $i_2\leftarrow
j_2$ or $i_2\leftarrow j_1$. The weight $w(i_1,j_1)=\beta(i_1,j_1)-1$ is
less than $\beta(i_1,j_1$ hence there is also an arrow $i_1\leftarrow
j_1$ of weight $(j_1,i_1)=-\beta-1$. The choices $w(i_1,i_2)=0$ and
$w(j_1,i_2)=0$ are valid since the arrangement is separated, 
the (dashed) reverse arrows $i_1\rightarrow i_2$, respectively
$j_1\rightarrow j_2$ are only present if $\beta(i_1,i_2)>0$, respectively
$\beta(j_1,j_2)>0$ hold. The weight of each directed cycle of length $2$
is $(-1)$ (as it should), and the same holds for the $3$-cycles
$(i_1,j_1,j_2)$ and $(i_1,i_2,j_2)$. The $4$-cycle $(i_1,i_2,j_1,j_2)$
has weight zero. Just like in the proof of Theorem~\ref{thm:wtc}, we may
select a linear order on $\{1,2,\ldots,n\}$ in such a way that
$i_1,i_2,j_1,j_2$ are the least elements, and for any pair of
vertices $\{i',j'\}$ not contained in $\{i_1,i_2,j_1,j_2\}$ we set
$w(i',j')=\beta(i',j')$ where $i'$ is the smaller vertex in our order. 
The resulting weighted digraph is valid, and its only $m$-ascending
cycle is $(i_1,i_2,j_1,j_2)$. 

Assume next that a valid weighted digraph associated to $\mathcal{A}$
has a minimal $m$-ascending cycle $(i_0,i_1,i_2,i_3)$ of length four. We will 
adapt the proof of Theorem~\ref{thm:a-bleq1} to tackle this case. As
before, we may use 
Lemma~\ref{lem:diagonals}, and after a cyclic rotation we may assume
that $i_0\rightarrow i_2$ and $i_1\rightarrow 
i_3$ are the diagonals of finite weight present. Using the 
fact that the directed cycles $(i_0,i_2,i_3)$ and $(i_1,i_3,i_0)$ are
not $m$-acyclic, the same calculation yields
$$
w(i_0,i_2)+w(i_1,i_3)\leq -2+w(i_1,i_2)-w(i_3,i_0).
$$
The left hand side is exactly $\beta(i_0,i_2)+\beta(i_1,i_3)$, as there
is no arrow $i_0\rightarrow i_2$ or $i_1\rightarrow i_3$. The
right hand side is at most 
$-2+\beta(i_1,i_2)-\alpha(i_3,i_0)=-2+\beta(i_1,i_2)+\beta(i_0,i_3)$. We
obtain the inequality
\begin{equation}
\label{eq:betaineq}  
\beta(i_0,i_2)+\beta(i_1,i_3)\leq -2+\beta(i_1,i_2)+\beta(i_0,i_3).
\end{equation}
By Theorem~\ref{thm:wtc}, we have $\beta(i_1,i_2)\leq \beta(i_1,i_3)+1$
and $\beta(i_0,i_3)\leq \beta(i_0,i_2)+1$. Using these as upper bounds
for the terms on the right hand side of~\eqref{eq:betaineq} we obtain
$$
\beta(i_0,i_2)+\beta(i_1,i_3)\leq -2+\beta(i_1,i_2)+\beta(i_0,i_3)\leq
\beta(i_0,i_2)+\beta(i_1,i_3). 
$$
All inequalities must be equalities, in particular we get
$\beta(i_1,i_2)=\beta(i_1,i_3)+1$ and
$\beta(i_0,i_3)=\beta(i_0,i_2)+1$. As a consequence 
$\beta(i_1,i_2)>\beta(i_1,i_3)$ and $\beta(i_0,i_3)<\beta(i_0,i_2)$
hold, that is, we have a crossing $\beta$-pattern. 
\end{proof}

\section{Extended Shi arrangements}
\label{sec:eShi}

In this section we show how labeled weighted digraphs may be used to fill in
the omitted details in Stanley's proof of the injectivity of the {\em
  Pak-Stanley labeling} of the regions of the extended Shi arrangement
in~\cite[2.1 Theorem]{Stanley-tinv}. We also revisit and generalize the
Athanasiadis-Linusson labeling~\cite{Athanasiadis-Linusson} of its
regions. A connection between this labeling and our weighted digraph
model will be indicated in Section~\ref{sec:ceiling}.  

By Corollary~\ref{cor:trasep} the extended Shi arrangement is separated:
for each region there is a unique permutation
$\sigma(1)\sigma(2)\ldots\sigma(n)$ such 
that $w(\sigma(i),\sigma(j))\geq 0$ holds for all $i<j$. Equivalently,
every region that is represented by a weighted digraph whose total order of
gains is $\sigma$ is a subset of the cone
$x_{\sigma(1)}>x_{\sigma(2)}>\cdots>x_{\sigma(n)}$. For
$i<_{\sigma^{-1}}j$, there is always a 
directed edge $i\rightarrow j$, and there is no arrow
$i\leftarrow j$ exactly when $w(i,j)$ is as large as possible, that is,
$w(i,j)=\beta(i,j)$ where
$$
\beta(i,j)=
\begin{cases}
  a, &\mbox{when $i<j$};\\
  a-1, &\mbox{when $i>j$}.\\
\end{cases}  
$$
By Theorem~\ref{thm:a-bleq1} we may verify
the $m$-acyclic condition in a valid associated weighted digraph by
only checking the directed cycles of length three.  Consider a set
$\{i,j,k\}\subseteq \{1,2,\ldots,n\}$, without loss of generality we may assume
$i<_{\sigma^{-1}}j<_{\sigma^{-1}}k$. There is no $m$-ascending cycle
$(i,j,k)$ if and only if either there is no arrow $i\leftarrow k$ or
$w(i,j)+w(j,k)-1-w(i,k)\leq -1$ holds. This condition may be
equivalently rewritten as
\begin{equation}
  \label{eq:wiklb}
w(i,k)\geq \min(\beta(i,j),w(i,j)+w(j,k))\quad\mbox{for
  $i<_{\sigma^{-1}}j<_{\sigma^{-1}}k$}.   
\end{equation}  
There is no $m$-ascending cycle $(i,k,j)$ exactly when either one of the
arrows $k\rightarrow j$, $j\rightarrow i$ is missing, or we have
$w(i,k)-1-w(i,j)-1-w(j,k)\leq -1$. The last inequality is equivalent
to
\begin{equation}
  \label{eq:wikub}
w(i,k)\leq w(i,j)+w(j,k)+1\quad\mbox{for
  $i<_{\sigma^{-1}}j<_{\sigma^{-1}}k$}.   
\end{equation}  
Note that \eqref{eq:wikub} also holds when one of the
arrows $k\rightarrow j$, $j\rightarrow i$ is missing, as $\beta(i,k)\leq
\beta(i,j)+1$ and $\beta(i,k)\leq \beta(j,k)+1$ always hold, regardless
of the relative order of the numbers $i$, $j$ and $k$. To summarize,
given a valid associated digraph satisfying
$w(i,j)\geq 0$ for all $i<_{\sigma^{-1}}j$, the weighted digraph is
$m$-acyclic if and only if \eqref{eq:wiklb} and \eqref{eq:wikub} hold.

Recall that the {\em Pak-Stanley labeling} associates to each region $R$
a vector $\lambda(R)=(f(1),f(2),\ldots,f(n))$ in the following way:
\begin{enumerate}
\item The label of the central region $x_1>x_2>\cdots>x_n-1$ is
  $(0,0,\ldots,0)$.
\item Suppose $R$ is an already labeled region, and the region $R'$
  separated from $R$ by the unique hyperplane $x_i-x_j=m$ for some
  $i<j$. Then $\lambda(R')=\lambda(R)+e_j$ if $m\leq 0$ and
  $\lambda(R')=\lambda(R)+e_i$ if $m>0$. Here $e_1,\ldots,e_n$ are the
  standard basis vectors of $\mathbb{R}^n$.   
\end{enumerate}  
In~\cite{Stanley-tinv} Stanley gives a detailed equivalent definition of
the function $f(i)$ in the case when $a=1$. We now extend this
equivalent definition to all extended Shi arrangements.
\begin{definition}
\label{def:SPl}  
Consider a valid $m$-acyclic associated weighted digraph of
$\mathcal{A}^{a,a+1}_{n-1}$ with weight function $w$. We define the {\em
  Pak-Stanley label $(f(1),\ldots,f(n))$} of the corresponding region as
$$
f(i)=\sum_{i<_{\sigma^{-1}} j} w(i,j)+|\{(i,j)\::\: i<_{\sigma^{-1}} j
\mbox{ and } i>j\}|.
$$
\end{definition}  
We let the reader verify that Definition~\ref{def:SPl} is equivalent to
the usual definition of the Pak-Stanley labeling given in the literature. For
example, in the case when $i<j$ and $0\leq w\leq a-1$ hold, 
$w(i,j)=w$ is equivalent to stating $w<x_i-x_j<w+1$, hence $w(i,j)$ is
the number of hyperplanes of the form $x_i-x_j=m$ crossed while reaching
our region from the central region.
Following~\cite{Stanley-tinv} we call the sum $\sum_{i<_{\sigma^{-1}} j}
w(i,j)$ the number of {\em separations}, whereas $|\{(i,j)\::\:
i<_{\sigma^{-1}} j \mbox{ and } i>j\}|$ is the number of {\em
  inversions}. We do not need to prove the equivalence of the above
definitions, as we may easily show Lemma~\ref{lem:aparking} below
directly. Recall that an {\em $a$-parking function} is a sequence 
$(f(1),\ldots,f(n))\in \mathbb{N}^n$ whose monotonic rearrangement
$\tilde{f}(1)\leq \tilde{f}(2)\leq \cdots \leq \tilde{f}(n)$ satisfies
$0\leq \tilde{f}(i)\leq a(i-1)$. 
\begin{lemma}
\label{lem:aparking}  
The Pak-Stanley labels $(f(1),\ldots,f(n))$ are $a$-parking functions.
\end{lemma}  
\begin{proof}
  This is a direct consequence of the fact that, for each $i<_{\sigma^{-1}}
j$, the number $w(i,j)$ is at most $a$ if $(i,j)$ is not an inversion
and at most $a-1$ if $(i,j)$ is an inversion. Hence the total number of
inversions and separations contributing to $f(i)$ is at most $a$ times
the number of labels $j$ succeeding $i$ in $\sigma$.
\end{proof}

Next we rephrase Stanley's key lemma~\cite[p.\ 363]{Stanley-tinv} to
match our terminology.
\begin{lemma}
\label{lem:pfkey}
Given $i<_{\sigma^{-1}} j$, if $i>j$ or $w(i,j)>0$ holds then we
have $f(i)>f(j)$.
\end{lemma}  
\begin{proof}
Assume first $i>j$ holds, i.e., $(i,j)$ is an inversion. Then for each
inversion $(j,k)$ the pair $(i,k)$ is also inversion, hence we may
identify the set of all inversions of $j$ with those inversions $(i,k)$
of $i$ which satisfy the stronger inequality $j>k$. Regarding
separations, note that~\eqref{eq:wiklb} implies $w(i,k)\geq w(j,k)$ in
almost all cases: either we have $w(i,k)\geq w(i,j)+w(j,k)\geq w(j,k)$;
or we have $w(i,k)\geq \beta(i,j)$ forcing $w(i,k)\geq a-1$ and
$w(j,k)\leq a$. The inequality $w(i,k)< w(j,k)$ can only hold when we have
$j<k<i$, $w(j,k)=a$ and $w(i,k)=a-1$. In this exceptional case $w(i,k)$
contributes one less to the separations of $i$ than $w(j,k)$ to the
separations of $j$, but this lag is offset by the presence of the
inversion $(i,k)$ which has no corresponding inversion $(j,k)$. Hence we
have $f(i)\geq f(j)$, and the presence of the additional inversion
$(i,j)$ (contributing to $f(i)$ only) makes the inequality strict. 

Assume next $w(i,j)>0$ holds. Once again, we may use~\eqref{eq:wiklb} to
state $w(i,k)\geq w(j,k)$ almost always. In the 
exceptional case $j<k<i$, $w(j,k)=a$ and $w(i,k)=a-1$ must hold, and
$(i,k)$ is an additional inversion. Regarding an inversion $(j,k)$, 
we may match it to $(i,k)$ if $(i,k)$ is an inversion, the unmatched
inversions $(j,k)$ satisfy $j<k<i$. For these~\eqref{eq:wiklb} implies
$$
w(i,k)\geq \min(a,w(i,j)+w(j,k))>w(j,k), 
$$
as $w(j,k)\leq a-1$ and $w(i,j)>0$ hold, hence $w(i,k)$ contributes more 
separations to $f(i)$ than $w(j,k)$ to $f(j)$. Therefore we have
$f(i)\geq f(j)$, and the presence of the additional separations
contributed by $w(i,j)>0$ to $f(i)$ makes the inequality strict.
\end{proof}  

Now we are ready to fill in the omitted details in the proof
of~\cite[2.1 Theorem]{Stanley-tinv}.

\begin{theorem}[Stanley]
The labels of the regions of the extended Shi arrangement are the
$a$-parking functions of length $n$, each occurring exactly once.
\end{theorem}  
\begin{proof}
The fact that the labels are $a$-parking functions have been shown in
Lemma~\ref{lem:aparking} above. As in the proof of~\cite[2.1
  Theorem]{Stanley-tinv}, we will only show the injectivity of the
labeling, and then we will rely on existing results in the literature to
confirm that the number of regions equals the number of $a$-parking
functions.

Given an $a$-parking function $(f(1),\ldots,f(n))$, we insert the labels
$i$ into $\sigma$ one by one and show the uniqueness of the place and of the
function values $w(i,j)$ one step at a time. As in the proof of~\cite[2.1
  Theorem]{Stanley-tinv}, we insert the labels $i$ in increasing order
of the value $f(i)$, and we insert labels with the same $f$-value
in decreasing order of their numerical value. Suppose $i$
is the most recently inserted label. If this label is preceded by any
label $i'$ in $\sigma$, then by Lemma~\ref{lem:pfkey} we must have
$w(i',i)=0$ and $i'<i$, otherwise we get $f(i')>f(i)$ in contradiction
with our defined order of insertion. Note that we must also have
$f(i')=f(i)$ in this case. We also obtained that the label $i$ can
not contribute any inversion or separation to any preceding $i'$. This
guarantees that at the insertion of $i$ we only need to make sure that
the instances of~\eqref{eq:wiklb}  and ~\eqref{eq:wikub} involving $i$
are satisfied and that the values $w(i,j)$ for all previously inserted $j$
succeeding $i$ in $\sigma$ are consistently defined. 

First we show that $i$ can not be inserted at two different places into
$\sigma$. Assume that the word representing the currently inserted
labels is $\widetilde{\sigma}=i_1\cdots i_m$ and $i$ could be inserted
either right 
after $i_r$ or right after $i_s$ for some $r<s$ (we set $r=0$ if $i$ is
inserted as the first letter). Let $w_r$ respectively $w_s$ be a
weight function arising with the insertion of $i$ right after $i_r$,
respectively $i_s$. As in the proof of Lemma~\ref{lem:pfkey},  
~\eqref{eq:wiklb} implies $w_r(i,k)\geq w_r(i_s,k)$
for almost all $k$ satisfying $i_s<_{\sigma^{-1}}k$, the conclusion
$w_r(i,k)\geq w_s(i,k)$ fails to hold only
when we have $w_r(i_s,k)=a$, $w_r(i,k)=a-1$ and $i_s<k<i$. 
On the other hand, ~\eqref{eq:wiklb} implies 
$$
w_s(i_s,k)\geq w_s(i_s,i)+w_s(i,k)\geq w_s(i,k)
$$
for all $k$ satisfying $i_s<_{\sigma^{-1}}k$: the exceptional case
$w_s(i_s,k)=a-1$ , $w_s(i,k)=a$, $i<k<i_s$ cannot occur, as it would
create the inversion $(i_s,i)$, in contradiction with the possibility of
inserting $i$ right after $i_s$. Combining the last two observations and
using the fact that $w_r(i_s,k)=w_s(i_s,k)$ (it is the same weight
function created before the insertion of $i$), we obtain
\begin{equation}
\label{eq:rs}  
w_r(i,k)\geq w_s(i,k)
\end{equation}
for almost all $k$ satisfying $i_s<_{\sigma^{-1}}k$. We only need to
check the case when $w_r(i_s,k)=a$, $w_r(i,k)=a-1$ and $i_s<k<i$
hold. In this exceptional case $i>k$ implies that the maximum value of
$w_s(i,k)$ is $a-1=w_r(i,k)$, hence~\eqref{eq:rs} holds in this exceptional
case as well. As a consequence the total number of separations of $i$
can not be greater when we insert it right after $i_s$ than in the case when
we insert it right after $i_r$. The same holds for the number of inversions,
finally the inversion $(i,i_s)$ is only present when we insert $i$ right
after $i_r$. In conclusion we can not obtain the same value $f(i)$ in
both scenarios.

We are left to show that inserting $i$ at the same place can not be
continued by extending the weight function in two different ways. By
Lemma~\ref{lem:pfkey} we must set $w(i',i)=0$ for all previously
inserted $i'<_{\sigma^{-1}}i$, variation may only occur in the
definition of $w(i,j)$ for some $i<_{\sigma^{-1}}j$. Assume there are
two different extensions: $w_1$ and $w_2$ of the weight function. Then
there is a leftmost $j$ in the order $<_{\sigma^{-1}}$ for which
$w_1(i,j)\neq w_2(i,j)$, without loss of generality we may assume
$w_1(i,j)< w_2(i,j)$. By definition, for all $k$ satisfying
$i<_{\sigma^{-1}}k<_{\sigma^{-1}}j$ we have $w_1(i,k)=w_2(i,k)$, these
$k$'s contribute the same separation to $f(i)$ in either case. On the
other hand, for all $k$ satisfying $j<_{\sigma^{-1}}k$, the
inequality~\eqref{eq:wiklb} implies
$$
w_2(i,k)\geq \min(\beta(i,k), w_2(i,j)+w_2(j,k)).
$$
Using $w_2(i,j)\geq w_1(i,j)+1$ and $w_2(j,k)=w_1(j,k)$ (since $j$ and
$k$ were inserted before $i$) we may rewrite this as  
$$
w_2(i,k)\geq \min(\beta(i,k), w_1(i,j)+w_1(j,k)+1).
$$
The inequality~\eqref{eq:wikub} implies
$$
w_2(i,k)\geq \min(\beta(i,k), w_1(i,k)).
$$
Since $\beta(i,k)$ is the maximum value of $w_1(i,k)$, we obtain
$w_2(i,k)\geq  w_1(i,k)$ for all $k$ satisfying
$j<_{\sigma^{-1}}k$. Keeping in mind $w_1(i,j)<w_2(i,j)$, we obtain a
strictly larger separation in the computation of $f(i)$ when we use
$w_2$ instead of $w_1$. The set of inversions being the same, we can not
obtain the same $f(i)$ using either extension of the weight function. 
\end{proof}

\begin{remark}
{\em Mazin~\cite{Mazin} has shown that the Pak-Stanley labeling of the
  regions of the extended Shi arrangement is surjective. Together with
  Stanley's above result we have a self-contained proof of the fact that
  the Pak-Stanley labeling is a bijection between the regions of the
  extended Shi arrangement and the $a$-parking functions.  
}
\end{remark}

We conclude this section by revisiting and generalizing the construction
of Athanasiadis and Linusson~\cite{Athanasiadis-Linusson}.
\begin{definition}
We say that the regions of a contiguous, separated and integral
deformation of the braid arrangement in $V_{n-1}$ given by the equations
$$
x_i-x_j=m\quad m\in [-\beta(j,i),\beta(i,j)]\quad\mbox{for $1\leq i<j<n$}  
$$
{\em have Athanasiadis-Linusson diagrams} if for each $j\in
\{1,2,\ldots,n\}$ the set 
$\{\beta(i,j)\::\: i\neq j\}$ has at most two elements and these
elements are consecutive nonnegative integers. We set
$\beta(j)=\min_{i\neq j} \beta(i,j)$ for all $j$. 
\end{definition}
Note that for the extended Shi arrangement $\mathcal{A}^{a,a+1}_{n-1}$
we have $\beta(1)=\cdots=\beta(n)=a-1$. More generally, as a consequence
of Theorem~\ref{thm:wtc}, the regions of every contiguous, separated and
integral deformation of the braid arrangement satisfying the weak triangle
inequality have Athanasiadis-Linusson diagrams. 
When a contiguous, separated and integral deformation of the
braid arrangement in $V_{n-1}$ has Athanasiadis-Linusson
  diagrams, we define one for each of its regions essentially the same
  way as it is done in~\cite{Athanasiadis-Linusson}:

\begin{enumerate}
\item We fix a representative $\underline{x}=(x_1,x_2,\ldots,x_n)$ of the
  region. These appear on the reversed number line in the order
  $x_{\sigma(1)}>x_{\sigma(2)}>\cdots > x_{\sigma(n)}$. (The arrangement
  being separated forces $\sigma$ to be independent from the choice of
  $\underline{x}$.)  
\item For each $j\in \{1,2,\ldots,n\}$ satisfying $\beta(j)>0$ we also
  mark $x_j+\beta(j), x_j+\beta(j)-1,\ldots,x_j+1$ on the reversed
  number line and we draw an arc connecting $x_j+k+1$ with $x_j+k$ for
  $k=0,1,\ldots,\beta(j)-1$. We label all of these points with $j$.
\item For each $\{i,j\}\subseteq \{1,2,\ldots,n\}$ we
  also draw an arc between $x_i$ and $x_j+\beta(j)$ if $\beta(i,j)=\beta(j)+1$
  and $x_i-x_j>\beta(i,j)$ hold.
\item We remove all nested arcs, that is, all arcs that contain another arc.  
\end{enumerate}   

In terms of weighted digraphs, the permutation $\sigma$ is the total
order of gains, and in step (3) we add an arc between $x_i$ and
$x_j+\beta(j)$  exactly when $w(i,j)=\beta(i,j)=\beta(j)+1$ holds. In such a
case $w(i,j)>0$ implies $i<_{\sigma^{-1}} j$ regardless of the choice of
$\underline{x}$. In all other cases, assuming
$i<_{\sigma^{-1}} j$, we have $w(i,j)=w$ if and only if the point
representing $x_i$ is to the left of the point $x_i+w$, but not to the
left of $x_i+w+1$ (if $w<\beta(i,j)$ and there is a point representing
$x_i+w+1$. Since the valid $m$-acyclic weighted digraphs bijectively
represent the regions of our hyperplane arrangement and they can be
uniquely reconstructed from the associated Athanasiadis-Linusson
diagrams, there is always a bijection between the regions and their
Athanasiadis-Linusson diagrams. It seems hard, however, to characterize
which diagrams may occur in general. 

\begin{example}
\label{ex:athanalin}  
{\em The diagram shown in Figure~\ref{fig:athanalin} is obtained
  from~\cite[Fig.\ 6]{Athanasiadis-Linusson} by adding a single point
  $5$ and an arc between the rightmost copy of $3$ and this newly added point. 
\begin{figure}[h]
\centering
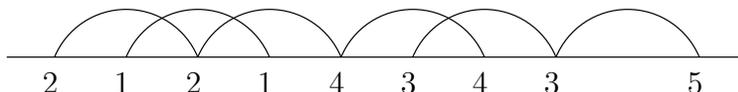
\caption{An Athanasiadis-Linusson diagram}
\label{fig:athanalin}
\end{figure}  
Without this addition, the original diagram represents a region in
$\mathcal{A}^{1,2}_{3}$, we have $\beta(i,j)=2$ for $i<j$ and
$\beta(i,j)=1$ for $i>j$ for all pairs $\{i,j\}\subset \{1,2,3,4\}$. We
add $\beta(i,5)=\beta(5,i)=0$ for $i=1,2,4$, and we add $\beta(3,5)=1$
and $\beta(5,3)=0$.    
}  
\end{example}

As in~\cite{Athanasiadis-Linusson}, for each
$i\in\{1,2,\ldots,n\}$ we define $f(i)$ as the position of the leftmost
element of the continuous component of~$i$. We call the resulting 
$(f(1),f(2),\ldots,f(n))$ the {\em $\beta$-parking function} of the
region. For example, in Example~\ref{ex:athanalin} we have $f(1)=2$,
$f(2)=f(4)=1$ and $f(3)=f(5)=6$. The proof of Athanasiadis and Linusson
showing that we may reconstruct the diagram from its $\beta$-parking
function still applies: we insert the components in increasing order of
the position of the left end of their components and we interlace the
rest of the inserted component with the already inserted components in
such a way that no pair of nested arcs is formed. It seems hard to
characterize the resulting $\beta$-functions in general, but they can be
easily visualized, using an idea implicit in~\cite[Ch1. Exercise
  (32)(c)]{Sagan}.  
\begin{figure}[h]
\centering
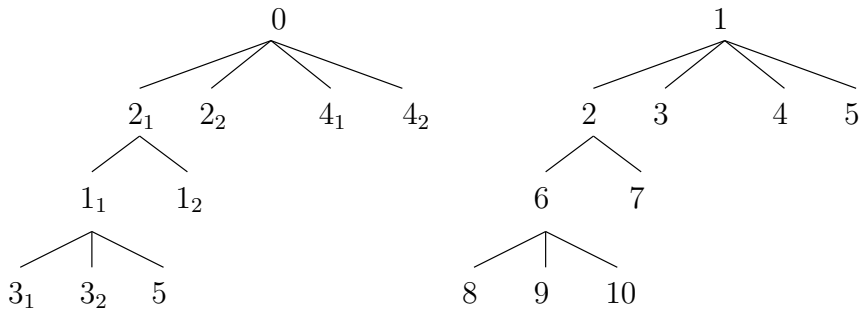
\caption{A rooted tree encoding an Athanasiadis-Linusson diagram}
\label{fig:alintree}
\end{figure}  
\begin{definition}
Given an Athanasiadis-Linusson diagram, we define the {\em parking tree}
representing it as follows:
\begin{enumerate}
\item Replace the labels $j$ with $j_1$, $j_2$,\ldots, $j_{\beta(j)+1}$,
  numbered left to right, so that we can distinguish the copies.
\item The copies of the labels satisfying $f(j)=1$ become the children
  of the root $0$.
\item We number the nodes in the tree level by level and in
  increasing order of the labels (breadth-first search order).
\item Once we inserted the copies of all labels $j$ satisfying $f(j)<i$,
  all copies of the labels $j$ satisfying $f(j)=i$ will be the children
  of the node whose number is $i$.    
\end{enumerate}  
\end{definition}  
The parking tree associated to the Athanasiadis-Linusson diagram shown
in Figure~\ref{fig:athanalin} is the tree on the left hand side in
Figure~\ref{fig:alintree}. The tree on the right hand side shows the
numbering of the positions in a breadth-first search order: we call this
the {\em position numbering}. 
The description of such a parking tree depends on the characterization
of the set of labels which can be siblings. This seems hard in general,
but easy in the following special case.

\begin{definition}
For a sequence $\underline{\beta}=(\beta(1),\beta(2),\ldots,\beta(n))\in
\mathbb{N}^n$ we define the {\em $\underline{\beta}$-extended Shi
  arrangement} as the hyperplane arrangement
\begin{equation}
x_i-x_j=-\beta(i),-\beta(i)+1,\ldots,\beta(j)+1 \quad 1\leq i<j\leq n\quad\mbox{in $V_{n-1}$.}
\end{equation}
\end{definition}  
In particular, setting  $\beta(1)=\cdots=\beta(n)=a-1$ yields the extended Shi
arrangement $\mathcal{A}^{a,a+1}_{n-1}$.
\begin{theorem}
The number of regions in a $\underline{\beta}$-extended Shi
arrangement $\mathcal{A}$ is
$$r(\mathcal{A})=\left(\sum_{j=1}^n (\beta(j)+1)+1\right)^{n-1}.$$
\end{theorem}
\begin{proof}
Obviously, the regions of a $\underline{\beta}$-extended Shi arrangement have
Athanasiadis-Linusson diagrams. When we build these diagrams,
$\beta(i,j)=\beta(j)+1$ holds exactly when $i<j$, hence we may draw an
arc between the points representing $x_i$ and $x_{j}+\beta(j)$ if and
only if $i<j$ and $x_i-x_j>\beta(j)+1$ hold, and these are all the
potential arcs added in step~(3). As a consequence, a set of labels can
be a set of 
siblings exactly when for each $j$ it contains either all copies of $j$,
or neither of them. We may count all such rooted trees using a colored
variant of the Pr\"ufer code algorithm as follows. We consider the
elements of $\{1,2,\ldots,n\}$ as colors, and we say that a color $j$ is
{\em exposed} if all labels $j_i$ are leaves in the tree. In each step
we remove the vertices of the least exposed color and record the common
parent $p(j_i)$ in our Pr\"ufer code. We never remove the root
$0$: we consider it the highest numbered vertex. We never remove
$p(j_i)$ before all copies of $j_i$ as our algorithm removes leaves only
in each step, hence the unique path between any not yet removed vertex and the
root must still be present. For example, for the tree on the
left hand side of Figure~\ref{fig:alintree}, the least exposed
color is $3$. The vertex $1_2$ is also a leaf, but $1_1$ is not, hence the
color $1$ is not yet exposed. There is always at least one exposed
color, because a leaf whose distance from the root is maximum has only
leaf siblings. We stop when
all remaining non-root vertices have the same color. For the tree on the
left hand side of Figure~\ref{fig:alintree}, we remove the colors
$3,4,5,1$ in this order and we record the Pr\"ufer code
$(1_1,0,1_1,2_1)$. The tree can be uniquely reconstructed from its
colored Pr\"ufer code, as follows. The least nonzero color not present
in the code is $3$: this must be the first removed color, and the common
parent of all vertices of color $3$ is the first coordinate of the code,
that is, $1_1$. We remove the first
coordinate from the Pr\"ufer code and record $3$ as a color in a
separate list of already reinserted vertices: we obtain the pair of
lists $((0,1_1,2_1), (3))$. The least nonzero color not present in the
current pair of lists is $4$, hence the vertices of color $4$ have
parent $0$, and we get the pair of lists $((1_1,2_1), (3,4))$. Now the
least nonzero color not present in our pair of lists is $5$ and its
parent is $1_1$. We continue the reconstruction of our tree in a similar
fashion. Our colored Pr\"ufer code has $n-1$ coordinates and each
coordinate can be either the root or any of the labeled vertices.  
\end{proof}

\section{The simplification lemma and ceiling hyperplanes}
\label{sec:ceiling}

Each edge in a weighted digraph we use to encode a region in a
deformation of the braid arrangement corresponds to a linear inequality
satisfied by the region. Not all of these inequalities define facets of the
boundary of the closure of region, some of them are consequences of the other
inequalities. The key result in this section is the Simplification Lemma
(Lemma~\ref{lem:simplification}) and its converse, which allow us to
identify the facet inequalities. These results allow us to shed new
light on and generalize some of the results of Armstrong and
Rhoades~\cite{Armstrong-Rhoades} on ceiling diagrams for the deleted Shi
and Ish arrangements (defined in Definitions~\ref{def:delShi} and
\ref{def:genIsh} below). The results in this section were inspired by 
suggestions of an anonymous referee. 

\begin{example}
\label{ex:nonwall}  
{\em 
  Let $n=3$, and consider the arrangement given by all
  hyperplanes of the form $x_i-x_j=0$ and by the hyperplanes $x_1-x_2=1$ and
  $x_1-x_3=1$ in $V_2$. The hyperplanes 
  of this arrangement, together with some additional hyperplanes (to be
  ignored at this point), are shown in Figure~\ref{fig:ish}. Let $R$ be
  the triangular region given by $x_1-x_2>0$, $x_2-x_3>0$, $x_1-x_2<1$  and
  $x_1-x_3<1$. The vertices  of $R$ are $(1/3,1/3,1/3)$, $(2/3,2/3,-1/3)$ and
  $(1,0,0)$. The associated weighted digraph is shown in
  Figure~\ref{fig:nonwall}. This is a separated deformation, the linear order of
  gains is the identity permutation for $R$, and $w(i,j)=0$ 
  holds for all $i<j$. The directed edge $2\rightarrow 1$ represents
  $x_1-x_2<1$. This inequality is not a facet of the boundary 
  of the topological closure of $R$: it is a consequence of the
  inequalities $x_2-x_3>0$ and $x_1-x_3<1$. Its only common point with
  the closure of $R$ is the vertex $(1,0,0)$.   
}  
\end{example}
\begin{figure}[h]
\centering
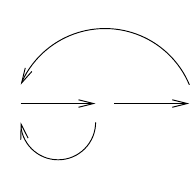
\caption{The weighted digraph associated to Example~\ref{ex:nonwall}.}
\label{fig:nonwall}
\end{figure}  
This example illustrates the fact that not all
  weighted edges in our weighted digraphs encode facet 
inequalities for the corresponding regions. The
$m$-acyclic property only ensures that each digraph encodes a nonempty
region. Furthermore, different weighted digraphs assign a different
weight to at least one directed edge, this difference between the
implied inequalities guarantees that the corresponding regions are disjoint,
thus different. Example~\ref{ex:nonwall} illustrates the following
phenomenon.
\begin{lemma}[Simplification lemma]
\label{lem:simplification}
Let ${\mathcal A}$ be a deformation of a braid arrangement in $V_{n-1}$, and
consider any valid weighted $m$-acyclic digraph representing a region of
${\mathcal A}$. If there is a directed path $i_1\rightarrow
i_2\rightarrow \cdots \rightarrow i_k$ and a directed edge
$i\rightarrow j$ such that
\begin{equation}
\label{eq:reduction}
w(i_1,i_2)+w(i_2,i_3)+\cdots + w(i_{k-1},i_k)\geq w(i_1,i_k)
\end{equation}
holds, then the linear inequality represented by the weighted edge
$i_1\rightarrow i_k$ is a consequence of the inequalities represented by
the edges in the directed path $i_1\rightarrow
i_2\rightarrow \cdots \rightarrow i_k$. 
\end{lemma}
Indeed, the directed edge $i\rightarrow j$ represents the linear
inequality $x_i-x_j>w(i,j)$, and the proof of
Lemma~\ref{lem:simplification}  is a simple addition of such
inequalities. Lemma~\ref{lem:simplification} also has the following
converse. 
\begin{proposition}
\label{prop:reduction}
Let ${\mathcal A}$ be any deformation of the braid arrangement in $V_{n-1}$, and
consider any valid weighted $m$-acyclic digraph representing a region of
${\mathcal A}$. If there is a directed edge $i\rightarrow j$ such that
the inequality $x_i-x_j>w(i,j)$ is a consequence of linear inequalities
represented by other directed edges, then there is a directed path
$i=i_1\rightarrow i_2\rightarrow \cdots \rightarrow i_k=j$ which,
combined with the edge $i\rightarrow j$, satisfies \eqref{eq:reduction}. 
\end{proposition}  
\begin{proof}
If $x_i-x_j>w(i,j)$ is a consequence of the inequalities encoded by the
other directed edges, then the contrary assumption $x_i-x_j\leq w(i,j)$
must lead to a contradiction. Strengthening this contrary assumption to
$x_i-x_j<w(i,j)$ cannot remove this contradiction. If we delete the
directed edge $j\rightarrow i$ (if present), and replace the
directed edge $i\rightarrow j$, of weight $w(i,j)$, with an edge
$j\rightarrow i$, of weight $-w(i,j)$, the resulting weighted digraph
must contain an $m$-ascending cycle $C=(i_1,i_2,\ldots,i_k)$ by
Theorem~\ref{thm:mincost}. This directed cycle  was not present in the original
weighted digraph, hence it must contain the newly added edge
$j\rightarrow i$. Without loss of generality we may assume $i_1=i$ and
$i_k=j$. The directed path $i=i_1\rightarrow i_2\rightarrow \cdots
\rightarrow i_k=j$ is also present in the weighted digraph we started
with, and \eqref{eq:reduction} is equivalent to stating that $C$ is
$m$-ascending. 
\end{proof}  
By Proposition~\ref{prop:reduction}, repeated use of
Lemma~\ref{lem:simplification} allows us to simplify the weighted
digraphs representing the regions by removing all  
``redundant'' directed edges representing linear inequalities which are
direct consequences of the linear inequalities represented by the
remaining edges. The remaining edges necessarily represent
facets of the boundary of the topological closure of the region defined by the
represented inequalities. This is a simple and general idea, but it seems a
serious challenge to describe the resulting weighted digraphs in
general. This section is about a class of deformations of the braid
arrangements where this task appears to be more manageable.    

\begin{definition}
\label{def:Shiish} We call a deformation of the
braid arrangement in $V_{n-1}$ given by \eqref{eq:dbraid} {\em
  Shiish}  if for every $i<j$ the set
$A_{ij}=\{a_{ij}^{(1)}, a_{ij}^{(2)}, \ldots,
a_{ij}^{(n_{ij})}\}$ satisfies $n_{ij}>0$ and $a_{ij}^{(1)}=0$. 
Here we assume that $a_{ij}^{(1)}< a_{ij}^{(2)}< \cdots<
a_{ij}^{(n_{ij})}$ holds. 
\end{definition}  
By definition a Shiish arrangement is separated,
the partial order of gains associated to each region is a linear order
$\sigma$. As noted in Section~\ref{sec:eShi}, a region whose linear
order of gains is $\sigma$ is contained in the region
$x_{\sigma(1)}>x_{\sigma(2)}>\cdots > x_{\sigma(n)}$ of the braid
arrangement. In terms of $\sigma$, we may describe the weights as follows.
\begin{lemma}
\label{lem:ceiling}  
Consider a valid weighted $m$-acyclic digraph of a region in a Shiish
arrangement in $V_{n-1}$ whose linear order of gains is $\sigma$. Then
for each $i<j$ we have the following: 
\begin{enumerate}
\item If $\sigma(i)>\sigma(j)$, then
$w(\sigma(i),\sigma(j))=0$ and $w(\sigma(j),\sigma(i))=-\infty$ hold.
\item If $\sigma(i)<\sigma(j)$, then $w(\sigma(i),\sigma(j))$
belongs to the set $A_{\sigma(i),\sigma(j)}$ and $w(\sigma(j),\sigma(i))$ is
given by
$$w(\sigma(j),\sigma(i))=
\begin{cases}
-a_{\sigma(i)\sigma(j)}^{(k+1)}, &\text{if
  $w(\sigma(i),\sigma(j))=a_{\sigma(i)\sigma(j)}^{(k)}$ for some $k<n_{\sigma(i)\sigma(j)}$;}\\
-\infty,&\text{if $w(\sigma(i),\sigma(j))=a_{\sigma(i)\sigma(j)}^{n_{\sigma(i)\sigma(j)}}$.}\\
\end{cases}  
$$
\end{enumerate}
\end{lemma}  
The straightforward verification is left to the reader.
\begin{corollary}
\label{cor:ceiling1}  
The weighted $m$-acyclic digraph encoding a region in a Shiish
arrangement in $V_{n-1}$ may be uniquely reconstructed from the linear
order of gains $\sigma$ and the value of each $w(\sigma(i),\sigma(j))$
such that $i<j$, $\sigma(i)<\sigma(j)$, and
$w(\sigma(i),\sigma(j))<a_{\sigma(i)\sigma(j)}^{(n_{\sigma(i)\sigma(j)})}$ hold.
\end{corollary}
The sufficient information described in Corollary~\ref{cor:ceiling1} is
related to the information encoded in the {\em ceiling diagrams} defined
in~\cite{Armstrong-Rhoades} for the (deleted) Shi and Ish arrangements.
Given a region $R$ in a hyperplane arrangement, a hyperplane $H$ is a
{\em wall} of R if it is the affine span of facet of the boundary of the 
closure of $R$. The wall $H$ is called a {\em ceiling} if $H$ does
not contain the origin and the region $R$ and the origin lie in the
same half-space of $H$.
\begin{proposition}
\label{prop:ceiling} 
Let $R$ be a region of a Shiish arrangement in $V_{n-1}$, encoded by an
weighted $m$-acyclic digraph whose linear order of gains is
$\sigma$. Then all ceilings of the region are hyperplanes
given by equations of the form
$$
x_{\sigma(j)}-x_{\sigma(i)}=w(\sigma(j),\sigma(i))
$$
where $i<j$, $\sigma(i)<\sigma(j)$ and $w(\sigma(j),\sigma(i))>-\infty$ hold. 
\end{proposition}
\begin{proof}
Let us fix $i<j$. Observe first that the directed edge
$\sigma(j)\rightarrow \sigma(i)$ is 
present with a finite weight if and only $\sigma(i)<\sigma(j)$ and
$w(\sigma(i),\sigma(j))=a_{\sigma(i)\sigma(j)}^{(k)}$ is not the maximal
element of $A_{\sigma(i)\sigma(j)}$. If these
conditions are satisfied then $w(\sigma(j),\sigma(i))=-a_{\sigma(i)\sigma(j)}^{(k+1)}$.
Hence the
statement may be rephrased as follows: all ceiling equations are of the form  
$$
x_{\sigma(i)}-x_{\sigma(j)}=a_{\sigma(i)\sigma(j)}^{(k+1)}
$$
where $i<j$, $\sigma(i)<\sigma(j)$ and
$w(\sigma(i),\sigma(j))=a_{\sigma(i)\sigma(j)}^{(k)}$ for some $k<n_{\sigma(i)\sigma(j)}$. We will prove
this rephrased statement. 

All walls must have an equation of the form
$x_{\sigma(i)}-x_{\sigma(j)}=a_{\sigma(i)\sigma(j)}^{(k)}$. If 
for some $i<j$ we have $\sigma(i)>\sigma(j)$ then the only hyperplane of
the form $x_{\sigma(i)}-x_{\sigma(j)}=a_{\sigma(i)\sigma(j)}^{(k)}$
involved in the 
definition of $R$ is the hyperplane
$x_{\sigma(i)}-x_{\sigma(j)}=0$. This hyperplane is not a 
ceiling as it contains the origin. Consider next the case when
$\sigma(i)<\sigma(j)$ holds. If
$w(\sigma(i),\sigma(j))=a_{\sigma(i)\sigma(j)}^{(n_{ij})}$,
then the hyperplane $x_{\sigma(i)}-x_{\sigma(j)}=0$ contains the origin
and all other hyperplanes
$x_{\sigma(i)}-x_{\sigma(j)}=a_{\sigma(i)\sigma(j)}^{(k)}$ for
$2\leq k\leq n_{ij}$ separate the origin from $R$. Finally,
if $w(\sigma(i),\sigma(j))=a_{\sigma(i)\sigma(j)}^{(k)}$ holds for some $k<n_{ij}$ then all
points in $R$ satisfy 
$$a_{\sigma(i)\sigma(j)}^{(k)}<x_{\sigma(i)}-x_{\sigma(j)}<a_{\sigma(i)\sigma(j)}^{(k+1)}$$ 
Only the hyperplane
$x_{\sigma(i)}-x_{\sigma(j)}=a_{\sigma(i)\sigma(j)}^{(k+1)}$ can be a ceiling. 
\end{proof}  
Not all hyperplanes satisfying the criteria of
Proposition~\ref{prop:ceiling} are ceilings. They all have the
property that they contain the region and the origin on the same side,
but some of them are not walls, as we have seen in Example~\ref{ex:nonwall}.
As a consequence of Proposition~\ref{prop:reduction} we obtain the
following description.
\begin{corollary}
  \label{cor:ceiling2}
Let $R$ be a region of a Shiish arrangement in $V_{n-1}$, encoded by an
weighted $m$-acyclic digraph whose linear order of gains is
$\sigma$. Then the ceilings of the region are those hyperplanes
$$
x_{\sigma(j)}-x_{\sigma(i)}=w(\sigma(j),\sigma(i))
$$
satisfying $i<j$, $\sigma(i)<\sigma(j)$ and
$w(\sigma(j),\sigma(i))>-\infty$ for which the edge
$\sigma(j)\rightarrow \sigma(i)$ can not be eliminated using
Lemma~\ref{lem:simplification}.   
\end{corollary}  

It seems an interesting question for future research, how to identify
the ceilings of the regions in an arbitrary Shiish arrangement using
Corollary~\ref{cor:ceiling2}.  In general it seems also a challenge
to verify the $m$-acyclic property, given $\sigma$ and the weight function
(reconstructed from Corollaries~\ref{cor:ceiling1} and
\ref{cor:ceiling2}). The work of finding the ceiling hyperplanes is done 
in~\cite{Armstrong-Rhoades} for the deleted Shi and Ish arrangements. We
begin with having a closer look at the deleted Shi arrangement.
\begin{definition}
\label{def:delShi} We call a Shiish arrangement in $V_{n-1}$ a {\em
deleted Shi arrangement} if it is a contiguous integral deformation of
the braid arrangement, and
$\alpha(i,j)=0$ and $\beta(i,j)\in \{0,1\}$ hold for every $i<j$ . 
\end{definition}  
In other words, a deleted Shi arrangement consists of all hyperplanes of
the braid arrangement and equations of the form $x_i-x_j=1$ for some
$i<j$; we can think of the latter ones as marking the edges of a simple loopless
graph $G$. This is the definition given in~\cite{Armstrong-Rhoades}.
The directed edges $\sigma(j)\rightarrow \sigma(i)$ of weight
$w(\sigma(j),\sigma(i))$ appearing in
Proposition~\ref{prop:ceiling} are precisely the edges of this graph
$G$, oriented in the direction that is opposite to the order of
$\sigma$. After fixing an 
order $x_{\sigma(1)}>\cdots >x_{\sigma(n)}$, we may represent the edges
of $G$ as arcs, indicating inequalities of the form
$x_{\sigma(i)}-x_{\sigma(j)}>1$ for $i<j$. An application of
Lemma~\ref{lem:simplification} is the following, well-known observation,
used in the construction of the Athanasiadis-Linusson diagrams
(discussed in Section~\ref{sec:eShi}): nested edges of $G$ may be
removed. We omit the details, as our approach in this case appears to
end up leading to the same ideas as the one already present
in~\cite{Armstrong-Rhoades,Athanasiadis-Linusson}.  We only add the
following consequence of Theorem~\ref{thm:wtc}. 
\begin{corollary}
Every deleted Shi arrangement satisfies the weak triangle inequality.  
\end{corollary}
Our approach allows considering a generalization of the deleted Ish
arrangements discussed in~\cite{Armstrong-Rhoades}. 
\begin{definition}
\label{def:genIsh}
We call a Shiish arrangement in $V_{n-1}$ a {\em generalized Ish
  arrangement} if for all 
$i<j$, $n_{ij}>1$ implies $i=1$. A generalized Ish arrangement is a
  {\em deleted Shi arrangement} if for each $j\in
  \{2,\ldots,n\}$,  the set
  $A_{1j}=\{a_{1j}^{(1)},a_{1j}^{(2)},\ldots,a_{1j}^{(n_{1j})}\}$ is a subset of
  $[0,j-1]=\{0,1,\ldots,j-1\}$.    
\end{definition}  
The original {\em Ish arrangement}, first introduced in the
work of Armstrong~\cite{Armstrong}, is the deleted Ish arrangement for
which $n_{1j}=j$, and thus $A_{1j}=[0,j-1]$ holds for $2\leq j\leq
n$. It is an easy consequence of Theorem~\ref{thm:wtc} that the Ish
arrangement does not satisfy the weak triangle inequality for $n\geq
3$. On the other hand, the verification of the $m$-acyclic property is
particularly easy for generalized Ish arrangements because of the
following result. 
\begin{theorem}
\label{thm:Ishcycles}  
A valid weighted digraph associated to a generalized Ish arrangement in
$V_{n-1}$ is $m$-acyclic if and only if it satisfies the
following for some permutation $\sigma$ of $\{1,2,\ldots,n\}$:
\begin{enumerate}
\item For $\{i,j\}\subseteq \{2,\ldots,n\}$ satisfying
  $\sigma^{-1}(i)<\sigma^{-1}(j)$ there is only a directed edge
  $i\rightarrow j$ of weight $w(i,j)=0$ (and we have $w(j,i)=-\infty)$);
\item every directed cycle of length $3$ and containing $1$ has negative
  weight.   
\end{enumerate}  
\end{theorem}  
\begin{proof}
  Given a valid $m$-acyclic weighted digraph, condition~(1) is a
  consequence of $A_{ij}=\{0\}$ for $1<i<j$, and condition~(2) is a special
case of the $m$-acyclic property. To prove the converse, assume by
contradiction that a valid weighted digraph satisfies the stated
conditions, but a directed cycle $(i_1,\ldots, i_k)$ of length $k>3$ has
nonnegative weight. By condition~(1), the restriction of our weighted
digraph to the set $\{2,\ldots,n\}$ is acyclic, hence we must
have $1\in\{i_1,\ldots,i_k\}$. Without loss of generality we may assume
$1=i_1$, the directed path $i_2\rightarrow \cdots \rightarrow i_k$ has
only edges of weight $0$, and it is easy to deduce from condition~(1)
that there is an edge $i_2\rightarrow i_k$ of weight zero. As a
consequence, the directed cycle $(i_1,i_2,i_k)$ of length $3$ contains
$1$ and has the same weight, in contradiction with condition~(2).    
\end{proof}  
Combining Corollary~\ref{cor:ceiling1} with Theorem~\ref{thm:Ishcycles}
we obtain the following simple way to encode each region of a
generalized Ish arrangement. 
\begin{definition}
\label{def:Ishcode}  
Consider a generalized Ish arrangement in $V_{n-1}$.  A {\em code}  of a
region is a pair $(\sigma,\omega)$ where $\sigma$ is a
permutation of 
$\{1,2,\ldots,n\}$ and
$$\omega=(\omega(1),\ldots,\omega(n))$$ 
is a vector in such a way that $\omega(j)=0$ holds for $j\leq
\sigma^{-1}(1)$ , and for $j>\sigma^{-1}(1)$ the number $\omega(j)$
belongs to the set $A_{1,\sigma(j)}$. We associate to the pair
$(\sigma,\omega)$ a weighted digraph by fixing $\sigma$ as the linear
order of gains, and setting $w(\sigma(j),\sigma^{-1}(1))=\omega(j)=0$ for
$j<\sigma^{-1}(1)$ and $w(1,\sigma(j))=\omega(j)$ for
$j>\sigma^{-1}(1)$. We extend the definition of the weights using
Corollary~\ref{cor:ceiling1}.  
\end{definition}  
Clearly each code encodes a valid weighted digraph associated to the
arrangement and each valid weighted digraph may be encoded in the above
way. Using Theorem~\ref{thm:Ishcycles} it is easy to check which codes
encode $m$-acyclic weighted digraphs.
\begin{proposition}
\label{prop:macycIsh}
For an element $a\in A_{ij}=\{a_{ij}^{(1)}, a_{ij}^{(2)}, \ldots,
a_{ij}^{(n_{ij})}\}$ let us set
$$
a^+=
\begin{cases}
  a_{ij}^{(k+1)}, & \text{if $a=a_{ij}^{(k)}$ for some $k<n_{ij}$};\\
  \infty, & \text{if $a=a_{ij}^{(n_{ij})}$.}\\
\end{cases}  
$$
A code $(\sigma,\omega)$ encodes a valid $m$-acyclic digraph if and only
if $\omega(j_1)<\omega(j_2)^+$ holds whenever $\sigma^{-1}(1)<j_1<j_2$.
\end{proposition}  
\begin{proof}
By Theorem~\ref{thm:Ishcycles} we only need to check that every directed
cycle of length $3$ containing $1$ has negative weight. Since no
directed edge ends in a $\sigma(j)$ preceding $1$ in the order of
$\sigma$, any directed cycle of length $3$ must be of the form 
$(1,\sigma(j_1),\sigma(j_2))$ where $\sigma^{-1}(1)<j_1<j_2$.
Here we have $w(1,\sigma(j_1))=\omega(j_1)$ and $w(j_1,j_2)=0$.
The directed edge $\sigma(j_2)\rightarrow 1$ is not present (has weight
$-\infty$) exactly when $w(1,\sigma(j_2))=\omega(j_2)$ is the largest
element of $\max A_{1,\sigma(j_2)}$, or equivalently when
$\omega(j_2)^+=\infty$. 
If $w(1,\sigma(j_2))=\omega(j_2)$ is not the largest
element of $A_{1,\sigma(j_2)}$ then
$w(\sigma(j_2),1)=-\omega(j_2)^+$. The
weight of the directed cycle is $\omega(j_1)-\omega(j_2)^+$. 
\end{proof}
\begin{remark}
\label{rem:contintv}
{\em Proposition~\ref{prop:macycIsh} takes a particularly nice form if the
generalized Ish arrangement is contiguous and integral. The original Ish
arrangement does have this property. In this case the
condition $\omega(j_1)<\omega(j_2)^+$ may be restated as follows: either
$\omega(j_2)=\beta(1,\sigma(j_2))$, that is, $\omega(j_2)=\max
A_{1,\sigma(j_2)}$ or $\omega(j_1)\leq \omega(j_2)$ holds. 
}
\end{remark}

\begin{figure}[h]
\centering
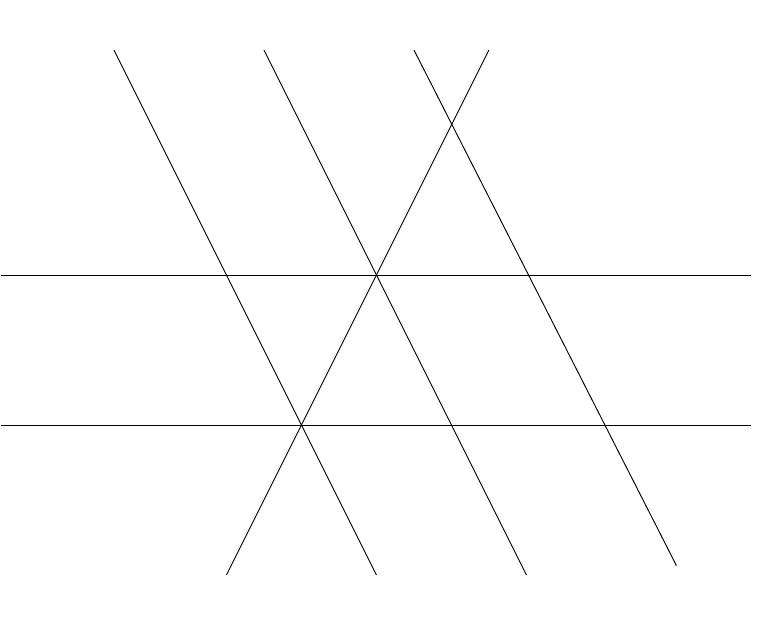
\caption{Codes labeling the regions of the Ish arrangement}
\label{fig:ish}
\end{figure}  

The labeling of the Ish arrangement in $V_2$ is shown in Figure~\ref{fig:ish}. 
A region is called {\em dominant} if it is contained in the braid cone
$x_1>x_2>\cdots> x_n$, equivalently, $\sigma$ is the identity
permutation. As a direct consequence of Corollary~\ref{cor:ceiling1},
dominant regions are in bijection with all vectors
$\omega=(\omega(1),\ldots,\omega(n))$ such that $0\leq \omega(j)\leq
j-1$ holds for $1\leq j\leq n$ and 
$\omega(1)\leq\cdots\leq\omega(n)$ is satisfied. It is easy to show that
the number of such vectors is a Catalan number. 
We conclude this section with a description of the ceiling hyperplanes
of a region in a generalized Ish arrangement.
\begin{proposition}
\label{prop:ceilingf}  
Let $R$ be a region in a generalized Ish arrangement in $V_{n-1}$,
whose associated weighted digraph has the code $(\sigma,\omega)$. Then a
hyperplane is a ceiling hyperplane if and only if it satisfies an
equation of the form 
\begin{equation}
\label{eq:ishceiling}  
x_{\sigma(1)}-x_{\sigma(j)}=\omega(j)^+
\end{equation}
for some $1<j$ satisfying $\sigma(1)<\sigma(j)$ and
$\omega(j)^+< \infty$, and there is no $j'$ satisfying $j<j'$ and
$\omega(j)^{+}\geq\omega(j')^+$. 
\end{proposition}  
\begin{proof}
The fact that all ceiling hyperplanes must be of the form
\eqref{eq:ishceiling} for some $1<j$ satisfying $\sigma(1)<\sigma(j)$ and 
$\omega(j)^+< \infty$ is just 
Corollary~\ref{cor:ceiling1}, restated using the $\omega$-notation.
In particular, $\omega(j)^+< \infty$ is equivalent to
$\omega(j)<a_{1j}^{n_{1j}}$.  Assume first that there is a $j'$
satisfying $j<j'$ and $\omega(j)^{+}\geq \omega(j')^+$ . The latter
inequality is equivalent to $w(\sigma(j),\sigma(1))\leq 
w(\sigma(j'),\sigma(1))$. In this case we may apply 
Lemma~\ref{lem:simplification} to the path $\sigma(j)\rightarrow
\sigma(j')\rightarrow 1$, since we have
$$
w(\sigma(j),\sigma(j'))+w(\sigma(j',\sigma(1)))=0-\omega(j')^+\geq
-\omega(j)^+=w(\sigma(j),\sigma(1)). 
$$
Therefore \eqref{eq:ishceiling} is not the equation of a ceiling
hyperplane. Conversely, assume \eqref{eq:ishceiling} is not the equation
of a ceiling hyperplane. Then, by Proposition~\ref{prop:macycIsh}, there
is a directed path $\sigma(j)=i_1\rightarrow\cdots\rightarrow
i_k=\sigma(1)$ from $\sigma(j)$ to $\sigma(1)$ whose total weight is at
least $w(\sigma(j),\sigma(1))$. Let us draw our valid weighted digraph
in such a way that the vertices are aligned on a horizontal line left to
right in the order $\sigma(1), \sigma(2),\ldots,\sigma(n)$, as shown in
Figure~\ref{fig:nonwall}. Introducing $\sigma(j')=i_{k-1}$, only
the weight of the edge $i_{k-1}\rightarrow i_k=\sigma(j')\rightarrow
\sigma(1)$ can be nonzero, this edge is directed right to left. All
other edges of $i_1\rightarrow\cdots\rightarrow i_k$ point in the left
to right direction. Hence we must have $j<j'$ and 
$$
-\omega(j')^+=w(\sigma(j'),\sigma(1))\geq
w(\sigma(j),\sigma(1))=-\omega(j)^+,
$$
implying $\omega(j)^+\geq \omega(j')^+$.
\end{proof}  
The coordinates of $\omega$ representing ceiling hyperplanes are
underlined in Figure~\ref{fig:ish}. For example, the region whose code
is $(132,011)$ has the ceiling hyperplane $x_3-x_1=\omega(2)^+=2$. Note
that $\omega(3)=\omega(2)=1$ but $\sigma(3)=2$, hence $\omega(3)=1$ is
the largest element of $\{0,1\}$ and $\omega(3)^+=\infty$. 
Proposition~\ref{prop:ceilingf} takes a particularly nice form for
the dominant regions in certain contiguous integral generalized Ish
arrangements. 
\begin{proposition}
Let $(12\cdots n,\omega)$ be the code of a dominant region in a
contiguous integral generalized  
Ish arrangement in $V_{n-1}$ satisfying $\beta(1,2)\leq \beta(1,3)\leq
\cdots \leq \beta(1,n)$. Then
$\omega(j)$ represents a ceiling hyperplane if and only if
$\omega(j)< \beta(1,j)$ and either $j=n$ or $\omega(j)<\omega(j+1)$ hold. 
\end{proposition}  
\begin{proof}
The condition $\omega(j)< \beta(1,j)$ is equivalent to $\omega(j)^+\neq
\infty$. If this holds then $\omega(j)^+=\omega(j)+1$. For all $j'>j$ we
have $\beta(1,j)\leq \beta(1,j')$, hence $\omega(j)^+\geq \omega(j')^+$
is equivalent to $\omega(j)+1\geq \omega(j')+1$, that is $\omega(j)\geq
\omega(j')$. As a consequence of Remark~\ref{rem:contintv}, a valid code must
also satisfy $\omega(1)\leq \omega(2)\leq \omega(n)$ in this
case. Therefore $\omega(j)\geq \omega(j')$ for some $j'>j$ is equivalent
to $\omega(j)=\omega(j+1)=\cdots=\omega(j')$ for some $j'>j$. 
\end{proof}  
\begin{corollary}
For the code $(12\ldots n,\omega)$ of a dominant region in the
Ish arrangement, ceiling hyperplanes correspond to those coordinates
$\omega(j)$ for which $\omega(j)\leq j-2$ and either $j=n$ or
$\omega(j)<\omega(j+1)$ hold.  
\end{corollary}

\section{The $a$-Catalan arrangements}
\label{sec:FC}

In this section we discuss two labelings of the regions of the
$a$-Catalan arrangement $\mathcal{A}^{a,a}_{n-1}$ for $a\geq 1$. First
we point out that Bernardi's {\em annotated
  sketches}~\cite[Section~8.1]{Bernardi} 
also arise as special instances of the generalized Athanasiadis-Linusson
diagrams defined in Section~\ref{sec:eShi}. The second labeling provides
a simple direct definition of the weighted digraphs using labeled $a$-Catalan
paths. These labelings are different from each other and also from the
labeling of Duarte and Guedes de Oliveira~\cite{Duarte-Oliveira} which
encode the Pak-Stanley labeling using labeled rational Dyck paths. It
seems an interesting question for future research to find direct
bijections between these labelings. 

The structure of the Athanasiadis-Linusson diagrams is the simplest
possible for the $a$-Catalan arrangement $\mathcal{A}^{a,a}_{n-1}$ for
$a\geq 1$. In this case 
$\beta(i,j)=a-1=\beta(j)$ holds for all $i\neq j$ and we may omit step
(3) in constructing the diagrams. 
\begin{figure}[h]
\centering
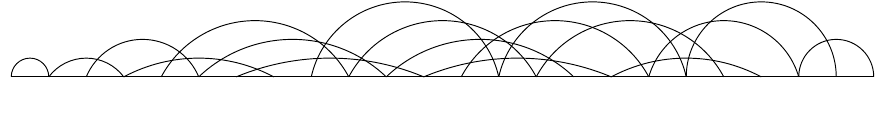
\caption{Athanasiadis-Linusson diagram of a region in $\mathcal{A}^{4,4}_5$}
\label{fig:alinfc}
\end{figure}  
Each continuous component consist of all $a$ copies of a single
$j\in\{1,2,\ldots,n\}$ as shown in Figure~\ref{fig:alinfc}. In this
figure we have $x_1>x_4>x_2>x_6>x_3>x_5$, but we can freely permute the labels
as these components can never be ``glued''
together. By the same reason, the diagram may be reconstructed from the
{\em Athanasiadis-Linusson word} 
$1141246\underline{1}3246532\underline{4}653\underline{2}\underline{6}5\underline{3}\underline{5}$
of the labels in the diagram. (It is helpful but not necessary to
underline the last appearance of each label.)
After fixing the order $x_{\sigma(1)}>x_{\sigma(2)}>\cdots
>x_{\sigma(n)}$, the parking trees associated to these diagrams are in
bijection with the rooted incomplete $a$-ary trees on $(a-1)n+1$
vertices. Their number is the $a$-Catalan number
$\frac{1}{(a-1)n+1}\binom{an}{n}$. Multiplying it with $n!$, which is the number
of possible choices of $\sigma$, we get 
\begin{equation}
\label{eq:FCregions}  
r(\mathcal{A}^{a,a}_{n-1})=an(an-1)\cdots ((a-1)n+2),
\end{equation}
which is the same as~\cite[Corollary 9.2]{Postnikov-Stanley}.
\begin{remark}
{\em    The Athanasiadis-Linusson words defined above are bijectively
    equivalent to the {\em annotated $m$-sketches} defined by
    Bernardi~\cite[Section~8.1]{Bernardi} as follows: we set $m=a$, replace the
    last (underlined) appearance of each $i$ with $\beta_{i}$ and all
    other appearances of the letter $i$ with $\alpha_{i}$. The careful
    reader will note that Bernardi's symbols represent real numbers that
    increase left to right, whereas here we follow the convention of the
    Athanasiadis-Linusson diagrams: the represented real numbers
    increase right to left. Bernardi states and 
    proves an explicit description of the annotated $m$-sketches and
    reduces the problem of counting them to counting $m$-parenthesis
    systems whose number is known to be an $m$-Catalan number. Above we
    relate essentially the same words to the parking trees defined in
    Section~\ref{sec:eShi}, and note that the same counting result may
    also be attained by counting rooted incomplete $a$-ary trees
    instead.}  
\end{remark}   

We may also represent the regions of the $a$-Catalan arrangement
$\mathcal{A}^{a,a}_{n-1}$ using all pairs $(\pi,\Lambda)$, where 
$\pi$ is a permutation of the set $\{1,2,\ldots,n\}$ and $\Lambda$ is an
{\em $a$-Catalan path} with $n$ up steps. 
\begin{definition}
An {\em $a$-Catalan path} with $n$ up steps  is a lattice
path containing $n$ up steps $(1,(a-1))$ and $(a-1)n$ down steps $(1,-1)$
from $(0,0)$ to $(an,0)$ that stays weakly above the horizontal
axis.
\end{definition}
It is well-known that the number of $a$-Catalan paths with $n$ up steps is
the $a$-Catalan number
$\frac{1}{(a-1)n+1}\binom{an}{n}$~\cite[Eq. (7.67)]{Graham-Knuth-Patashnik}.
This representation is easier to visualize but somewhat
mysterious at this time. We use the permutation $\pi$ to number the up
steps of $\Lambda$ from 
left to right. In the example shown in Figure~\ref{fig:alinfcp} the
permutation $\pi$ is simply $123456$, but we may reuse the same lattice
path with any other permutation. As it was also the case for the
Athanasiadis-Linusson diagrams, the set of valid representations is
invariant under permuting the labels.

\begin{figure}[h]
\centering
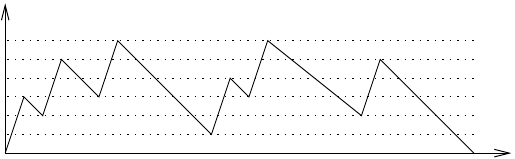
\caption{$a$-Catalan path corresponding to the Athanasiadis-Linusson
  diagram shown in Figure~\ref{fig:alinfc}}
\label{fig:alinfcp}
\end{figure}  
We identify the $i$th up step from the right with the variable
$x_{\pi(i)}$ and denote the level of the lower end of the up step with
$\ell(\pi(i))$. We use this {\em level function} $\ell:
\{1,2,\ldots,n\}\rightarrow \mathbb{N}$ to define the weight function,
as follows.
\begin{equation}
\label{eq:wdef}  
  w(\pi(i),\pi(j))=
  \begin{cases}
\ell(\pi(j))-\ell(\pi(i)), &\mbox{if $\ell(\pi(j))-\ell(\pi(i))\in
  [1-a,a-1]$;}\\
-\infty, &\mbox{if $\ell(\pi(j))-\ell(\pi(i))<1-a$;}\\
a-1, &\mbox{if $\ell(\pi(j))-\ell(\pi(i))>a-1$.}\\
  \end{cases}  
\end{equation}
Recall that the notation $w(\pi(i),\pi(j))=-\infty$ indicates the case
when there is no arrow $\pi(i)\rightarrow \pi(j)$ in the corresponding
weighted digraph, only an arrow $\pi(i)\leftarrow \pi(j)$ of weight
$a-1$. Equation~\eqref{eq:wdef} defines exactly one of $w(i,j)$ and
$w(j,i)$ for each pair $\{i,j\}$, we may uniquely extend this definition
to all ordered pairs using the rule
\begin{equation}
\label{eq:wdef2}  
  w(i,j)=
  \begin{cases}
-1-w(j,i), &\mbox{if $w(j,i)\in
  [1-a,a-2]$;}\\
-\infty, &\mbox{if $w(j,i)=a-1$;}\\
a-1, &\mbox{if $w(j,i)=-\infty$.}\\
  \end{cases}  
\end{equation}
Clearly we obtain the weight function of a valid weighted digraph.
\begin{proposition}
Equations~\eqref{eq:wdef} and~\eqref{eq:wdef2} define an $m$-acyclic
weighted digraph. 
\end{proposition}  
\begin{proof}
Consider a directed cycle $(\pi(i_1),\pi(i,2),\ldots,\pi(i_k)))$. In the
proof we add the indices modulo $k$, that is we set $i_{k+1}=i_1$.
For $i_s<i_{s+1}$ we have $w(\pi(i_s),\pi(i_{s+1})\leq
\ell(\pi(i_{s+1}))-\ell(\pi(i_{s+1}))$. For $i_s>i_{s+1}$ we must have
$w(\pi(i_{s+1}),\pi(i_s))<a-1$ otherwise
$w(\pi(i_{s}),\pi(i_{s+1}))=-\infty$ holds. Hence 
we have
\begin{align*}
w(\pi(i_s),\pi(i_{s+1})&=-1-w(\pi(i_{s+1}),\pi(i_s))
=-1-(\ell(\pi(i_{s}))-\ell(\pi(i_{s+1}))\\
&<\ell(\pi(i_{s+1})-\ell(\pi(i_{s}).
\end{align*}
To summarize $w(\pi(i_s),\pi(i_{s+1})\leq
\ell(\pi(i_{s+1})-\ell(\pi(i_{s})$ always holds, hence the total weight
of the cycle is at most $\sum_{j=1}^k
(\ell(\pi(i_{j+1})-\ell(\pi(i_j))=0$. The sum is strictly negative
because there is at least one $s$ such that $i_s>i_{s+1}$ holds.
\end{proof}  
\begin{definition}
We call the $m$-acyclic valid weighted digraph associated to   
$(\pi,\Lambda)$ by the Equations~\eqref{eq:wdef}
and~\eqref{eq:wdef2} the {\em weighted digraph encoded by $(\pi,\Lambda)$}.
\end{definition}  

\begin{lemma}
\label{lem:sigma} 
Consider the weighted digraph encoded by 
$(\pi,\Lambda)$. The total order of gains $\sigma=\gamma\circ\pi$
is the order of the labels $\pi(1),\ldots,\pi(n)$ in increasing order of
their levels, where $\pi(i)$ is listed before $\pi(j)$ if
$\ell(\pi(i))=\ell(\pi(j))$ and $i<j$ hold.
\end{lemma}  
For example, for the labeled lattice path shown in
Figure~\ref{fig:alinfcp}  we get $\sigma=142635$. It is an immediate
consequence of  Equation~\eqref{eq:wdef}
that $w(\sigma(i),\sigma(j))\geq 0$ holds for all $i<j$, and there is
exactly one permutation with this property. 

\begin{proposition}
\label{prop:gl}
For the weighted digraph encoded by $(\pi,\Lambda)$ the gain function is
the level function: we have $g(\sigma(i))=\ell(\sigma(i))$.
\end{proposition}
\begin{proof}
We proceed by induction on $i$. For $i=1$ we have
$g(\sigma(1))=\ell(\sigma(1))=0$. Assume that
$g(\sigma(i))=\ell(\sigma(i))$ holds for all $i<j$ and let us compute
$g(\sigma(j))$. By Theorem~\ref{thm:posweight} there is a $j'<j$
satisfying $g(\sigma(j))=g(\sigma(j'))+w(\sigma(j'),\sigma(j))$. By the
induction hypothesis $g(\sigma(j'))=\ell(\sigma(j'))$ holds, hence we have
$$
g(\sigma(j))=\ell(\sigma(j'))+w(\sigma(j'),\sigma(j))\leq
\ell(\sigma(j'))+\ell(\sigma(j))-\ell(\sigma(j'))=\ell(\sigma(j)). 
$$
On the other hand, in an $a$-Catalan path, between two consecutive 
up steps the level increases by at most $a-1$ (it may also decrease all
the way to zero). If we project all up
steps to the horizontal axis, we obtain a set of intervals whose union
is an interval of integers. Therefore, between any two up steps there is
such a sequence of up steps that the level function can increase by at
most $a-1$ between two subsequent up steps $(1,a-1)$, hence there is at
least one up step labeled with $\sigma(j'')$ preceding the up step labeled
with $\sigma(j)$ whose level is in the interval
$[-a+1+\ell(\sigma(j)),\ell(\sigma(j))]$. By 
Lemma~\ref{lem:sigma} $j''<j$ holds, hence by Equation~\eqref{eq:wdef} we have
$w(\sigma(j''),\sigma(j))=\ell(\sigma(j))-\ell(\sigma(j''))$. Combining
this with the induction hypothesis we also obtain
$$
g(\sigma(j))\geq g(\sigma(j''))+w(\sigma(j''),\sigma(j))=
\ell(\sigma(j''))+\ell(\sigma(j))-\ell(\sigma(j''))=\ell(\sigma(j)). 
$$
\end{proof}  

\begin{theorem}
The correspondence between the pairs $(\pi,\Lambda)$ and the valid
weighted $m$-acyclic digraphs encoded by them is a bijection.
\end{theorem}
\begin{proof}
It suffices to show that the correspondence is injective: by
\eqref{eq:FCregions} the number of regions (and of valid $m$-acyclic
weighted digraphs) is the same as the number of labeled $a$-Catalan paths. 
Consider a labeled $a$-Catalan path $(\pi,\Lambda)$ and the weighted
digraph (described by the weight function $w(i,j)$) it encodes. By the
results of Section~\ref{sec:separated}, we can compute the total order
of gains $\sigma(i)$ and the gain function $g(i)$ from the function
$w(i,j)$. By Lemma~\ref{lem:sigma} and Proposition~\ref{prop:gl} the
gain function is the same as the level function and $\sigma$ lists the
up steps level by level, and left to right within the same level. Using
this information we can determine the level of each labeled up step, and
also the relative order of two up steps at the same level. Consider
finally a pair of up steps, labeled with $i$ and $j$, respectively,
whose level is different. If $|\ell(i)-\ell(j)|\leq a-1$ holds then, by
Equations~\eqref{eq:wdef} and ~\eqref{eq:wdef2},
$w(i,j)=\ell(j)-\ell(i)$ holds exactly when the up step labeled $i$
precedes the up step labeled $j$ in $\Lambda$. As noted in the proof of
Proposition~\ref{prop:gl}, the difference of levels between any two
subsequent up steps in the sequence is at most $a-1$. As a consequence,
we can determine the relative position of any two labeled up steps whose
level is different. 
\end{proof}  

It is not hard to create the Athanasiadis-Linusson diagram of a region
given by its labeled $a$-Catalan path $(\pi,\Lambda)$: the weight
function is easily computed, the linear order of gains $\sigma$ can be
read off the labeled lattice path, and a table such as
Table~\ref{tab:alinfcw} is easily created. Using the table of the weight
function, we can insert the letters into the Athanasiadis-Linusson word
in the order of $\sigma$. Finding the inverse of this map appears to be
a much more difficult task.

\begin{table}[h]
\begin{center}
  \begin{tabular}{c|ccccccccccc}
$w(\sigma(i),\sigma(i+5))$  &     &     &     &  &     & $3$ &
    &&&&\\
$w(\sigma(i),\sigma(i+4))$  &     &     &     &     & $3$ &   &$2$ &&&&\\     
$w(\sigma(i),\sigma(i+3))$  &     &     &     & $2$ &     & $1$ & &$1$&&&\\         
$w(\sigma(i),\sigma(i+2))$  &     &     & $2$ &     & $1$ &   &$1$ &&$0$&&\\     
$w(\sigma(i),\sigma(i+1))$  &     & $1$ &     & $0$ &     & $0$ & &$0$&&$0$&\\     
\hline
vertices $\sigma(1)\sigma(2)\cdots\sigma(6)$
                            & $1$ &     & $4$ &     & $2$ & & $6$& &$3$&
& $5$\\ 
\hline
$g(\sigma(i))$              & $0$ &     & $1$ &     & $2$  &  & $2$ &&$3$&&$3$\\     
\end{tabular}
\end{center}
\caption{Weight function table associated to the Athanasiadis-Linusson
  diagram shown in Figure~\ref{fig:alinfc}}
\label{tab:alinfcw}
\end{table}
We conclude this section with a description of the bounded regions.

\begin{proposition}
\label{prop:fcbounded}  
A region of $\mathcal{A}^{a,a}_{n-1}$ is bounded if and only if the
total order of gains $\sigma$ satisfies $w(\sigma(i),\sigma(i+1))<a-1$
for $1\leq i\leq n-1$. 
\end{proposition}  
\begin{proof}
By definition we have $w(\sigma(i),\sigma(i+1))\geq 0$ for $1\leq i\leq
n-1$. If $w(\sigma(i),\sigma(i+1))<a-1$ for $1\leq i\leq
n-1$, then the directed edges $\sigma(i)\rightarrow
\sigma(i+1)$ and $\sigma(i)\leftarrow \sigma(i+1)$ are both present, and
the associated weighted digraph is strongly connected. Assume now that
$w(\sigma(i_0),\sigma(i_0+1))=a-1$ holds for some $1\leq i_0\leq n-1$.  
In analogy to the case of the extended Shi arrangement, the fact that
there is no $m$-ascending cycle implies \eqref{eq:wiklb}, that is,
$$
w(i_1,i_3)\geq \min(a-1,w(i_1,i_2)+w(i_2,i_3))\quad\mbox{holds for
  $i_1<_{\sigma^{-1}}i_2<_{\sigma^{-1}}i_3$}.   
$$
Using this inequality it is easy to show that $w(i_1,i_2)=a-1$ holds for
$i_1\leq_{\sigma^{-1}}i<_{\sigma^{-1}}i_2$. In other words, there is no
directed edge from the set $\{i_2\::\: i<_{\sigma^{-1}} i_2\}$ into the set
$\{i_1\::\: i_1\leq_{\sigma^{-1}} i\}$. The associated weighted digraph
is not strongly connected. 
\end{proof}  
Using Lemma~\ref{lem:sigma}, Proposition~\ref{prop:fcbounded} may be
rephrased as follows.
\begin{corollary}
A region of $\mathcal{A}^{a,a}_{n-1}$, encoded by the labeled lattice
path $(\pi,\Lambda)$, is unbounded if and only there are two successive
levels $\ell$ and $\ell+a-1$ whose difference is $a-1$ and the last up
step of $\Lambda$ at level $\ell$ precedes the first up step at level
$\ell+a-1$. 
\end{corollary}  
As noted in Corollary~\ref{cor:Catalan}, the number of possible types of
the trees of the gain function is a Catalan number.
\begin{conjecture}
\label{conj:Catpoly}  
For a fixed $n$ and a fixed tree of gain functions, the number of regions of
$\mathcal{A}^{a,a}_{n-1}$ associated to it is a polynomial of $a$.
\end{conjecture}  
Conjecture~\ref{conj:Catpoly} implies that the
$n$-th $a$-Catalan number, considered as a polynomial of $a$, could be
written as a sum of $C_n$ polynomials, where $C_n$ is the $n$-th Catalan
number.  
\begin{example}
\label{ex:Catpoly}
{\em Consider the case $n=3$, that is, the $a$-Catalan arrangement
  $\mathcal{A}^{a,a}_{2}$. Without loss of generality we may restrict our
  attention to the case when $\sigma$ is the identity permutation, that
  is, $x_1>x_2>x_3$ holds. There are $C_2=2$ noncrossing trees
  rooted at $1$ on the set $\{1,2,3\}$ that do not contain backward
  edges. The first rooted tree $T_1$ is underlying to the directed path
  $1\rightarrow 2\rightarrow 3$, the second rooted tree $T_2$ is
  underlying to the pair of 
  of edges $\{1\rightarrow 2,1\rightarrow 3\}$. Note that both trees
  contain the edge $1\rightarrow 2$: 
  without using backward edges, this is the only way we can move from
  vertex $1$ to vertex $2$. The distinction we need to make is based on
  the cost of the most gainful path from $1$ to $3$. A valid $m$-acyclic
  weighted digraph belongs to the tree $T_1$ if and only if
  \begin{equation}
  \label{eq:T1}  
  w(1,2)+w(2,3)\geq w(1,3) \quad\mbox{holds}.
  \end{equation}
  Equality is
  allowed for the weighted digraphs associated to $T_1$, because by 
  Definition~\ref{def:treedef} in case of equality we would still select
  $2$ as the parent of $3$. For any valid
  weighted digraph, the directed cycles of length $2$ are not
  $m$-ascending (the sum of the weights is $-1$), and there are at most
  $2$ directed cycles of length $3$: the cycle $1\rightarrow
  3\rightarrow 2\rightarrow 1$ and the cycle $1\rightarrow
  2\rightarrow 3\rightarrow 1$. The first cycle $1\rightarrow
  3\rightarrow 2\rightarrow 1$  exists if and only if $0\leq w(1,2),
  w(2,3)\leq a-2$ holds. The $m$-acyclic property for this cycle is
  equivalent to requiring $-1-w(1,2)-1-w(2,3)+w(1,3)\leq -1$, that is,
  $w(1,3)\leq w(1,2)+w(2,3)+1$. Note that this condition is a
  consequence of \eqref{eq:T1}, hence when we count the valid weighted
  digraphs associated to $T_1$, we do not need to check the $m$-acyclic
  property for this directed cycle.  The second cycle exist if and only
  if  $0\leq   w(1,3)\leq a-2$ holds, if it exists, its weight is 
  $$w(1,2)+w(2,3)+w(3,1)=w(1,2)+w(2,3)-1-w(1,3). 
  $$
This weight is at most $-1$ if and only if $w(1,3)\geq w(1,2)+w(2,3)$
holds.

{\bf\noindent Case 1:} $w(1,3)=a-1$. In this case the directed
cycle $1\rightarrow 2\rightarrow 3\rightarrow 1$ is not present.
Let us choose the values of $w(1,2)$ and $w(2,3)$ 
  from the set $\{0,\ldots,a-1\}$ independently. These values yield a
  valid $m$-acyclic digraph associated to $T_1$ if and only if
  $w(1,2)+w(2,3)\geq a-1$ holds. There are $a^2$ ways to select the pair
  $(w(1,2),w(2,3))$ and of these $\binom{a}{2}$ satisfy
  $w(1,2)+w(2,3)\leq a-2$. (This binomial coefficient is the number of
  nonnegative integer solutions of $z_1+z_2+z_3=a-2$.) We obtain that in
  this case the number of weighted $m$-acyclic digraphs associated to $T_1$ is $a^2-\binom{a}{2}$.   

  {\bf\noindent Case 2:} $0\leq w(1,3)\leq a-2$ holds. In this case the directed
cycle $1\rightarrow 2\rightarrow 3\rightarrow 1$ is present and the
$m$-acyclic property for this cycle, combined with \eqref{eq:T1} yields
\begin{equation}
\label{eq:T1c2}  
  w(1,3)=w(1,2)+w(2,3).
\end{equation}
  The number of valid $m$-acyclic weighted digraphs associated to $T_1$
  in this case is the number of ways to select $w(1,3)\in
  \{0,\ldots,a-2\}$ and then write it as a sum of two nonnegative
  integers. This number is again the number of nonnegative integer
  solutions of $z_1+z_2+z_3=a-2$, that is, $\binom{a}{2}$.

Summing over the two cases yields that the number of valid $m$-acyclic
digraphs associated to $T_1$ is simply $a^2$, a polynomial of $a$. Since
the number of all valid $m$-acyclic weighted digraphs is also a
polynomial of $a$, the validity of Conjecture~\ref{conj:Catpoly} for
$T_2$ follows by subtraction.   
} 
\end{example}
The reader is invited to find a simpler proof for the case when
$n=3$. The author used a similar, but more tedious
case-by-case analysis, to show directly that the number of valid
weighted digraphs associated to $T_2$ is also a polynomial of $a$. He also
worked out the contribution of some trees in the $n=4$ and $n=5$
case. In each step of each such logical analysis, polynomials of $a$
are added or subtracted.

\section*{Acknowledgments}
This work was partially supported by a grant from the Simons Foundation
(\#514648 to G\'abor Hetyei). The author thanks Christos
Athanasiadis, Samuel Hopkins, Alex Lazar and Thomas Zaslavsky for useful
advice. The author is indebted to two anonymous referees for suggesting
substantial improvements to this paper.

\end{document}

%% file: m5.pdf_t
\begin{picture}(0,0)%
\includegraphics{m5.pdf}%
\end{picture}%
\setlength{\unitlength}{3947sp}%
\begingroup\makeatletter\ifx\SetFigFont\undefined%
\gdef\SetFigFont#1#2#3#4#5{%
  \reset@font\fontsize{#1}{#2pt}%
  \fontfamily{#3}\fontseries{#4}\fontshape{#5}%
  \selectfont}%
\fi\endgroup%
\begin{picture}(2730,2611)(4411,-4925)
\put(6826,-4861){\makebox(0,0)[lb]{\smash{{\SetFigFont{12}{14.4}{\familydefault}{\mddefault}{\updefault}{\color[rgb]{0,0,0}$\textcircled{2}$}%
}}}}
\put(5851,-2461){\makebox(0,0)[lb]{\smash{{\SetFigFont{12}{14.4}{\familydefault}{\mddefault}{\updefault}{\color[rgb]{0,0,0}$5$}%
}}}}
\put(6976,-3436){\makebox(0,0)[lb]{\smash{{\SetFigFont{12}{14.4}{\familydefault}{\mddefault}{\updefault}{\color[rgb]{0,0,0}$4$}%
}}}}
\put(6676,-4861){\makebox(0,0)[lb]{\smash{{\SetFigFont{12}{14.4}{\familydefault}{\mddefault}{\updefault}{\color[rgb]{0,0,0}$2$}%
}}}}
\put(4951,-4861){\makebox(0,0)[lb]{\smash{{\SetFigFont{12}{14.4}{\familydefault}{\mddefault}{\updefault}{\color[rgb]{0,0,0}$1$}%
}}}}
\put(4651,-3436){\makebox(0,0)[lb]{\smash{{\SetFigFont{12}{14.4}{\familydefault}{\mddefault}{\updefault}{\color[rgb]{0,0,0}$3$}%
}}}}
\put(6301,-3361){\makebox(0,0)[lb]{\smash{{\SetFigFont{12}{14.4}{\familydefault}{\mddefault}{\updefault}{\color[rgb]{0,0,0}$-2$}%
}}}}
\put(6301,-2911){\makebox(0,0)[lb]{\smash{{\SetFigFont{12}{14.4}{\familydefault}{\mddefault}{\updefault}{\color[rgb]{0,0,0}$-2$}%
}}}}
\put(5701,-4861){\makebox(0,0)[lb]{\smash{{\SetFigFont{12}{14.4}{\familydefault}{\mddefault}{\updefault}{\color[rgb]{0,0,0}$-2$}%
}}}}
\put(6301,-4186){\makebox(0,0)[lb]{\smash{{\SetFigFont{12}{14.4}{\familydefault}{\mddefault}{\updefault}{\color[rgb]{0,0,0}$3$}%
}}}}
\put(5101,-2986){\makebox(0,0)[lb]{\smash{{\SetFigFont{12}{14.4}{\familydefault}{\mddefault}{\updefault}{\color[rgb]{0,0,0}$3$}%
}}}}
\put(6751,-4111){\makebox(0,0)[lb]{\smash{{\SetFigFont{12}{14.4}{\familydefault}{\mddefault}{\updefault}{\color[rgb]{0,0,0}$-2$}%
}}}}
\put(4801,-4111){\makebox(0,0)[lb]{\smash{{\SetFigFont{12}{14.4}{\familydefault}{\mddefault}{\updefault}{\color[rgb]{0,0,0}$3$}%
}}}}
\put(5176,-4261){\makebox(0,0)[lb]{\smash{{\SetFigFont{12}{14.4}{\familydefault}{\mddefault}{\updefault}{\color[rgb]{0,0,0}$3$}%
}}}}
\put(5476,-4561){\makebox(0,0)[lb]{\smash{{\SetFigFont{12}{14.4}{\familydefault}{\mddefault}{\updefault}{\color[rgb]{0,0,0}$3$}%
}}}}
\put(5026,-3586){\makebox(0,0)[lb]{\smash{{\SetFigFont{12}{14.4}{\familydefault}{\mddefault}{\updefault}{\color[rgb]{0,0,0}$-2$}%
}}}}
\put(4426,-3436){\makebox(0,0)[lb]{\smash{{\SetFigFont{12}{14.4}{\familydefault}{\mddefault}{\updefault}{\color[rgb]{0,0,0}$\textcircled{3}$}%
}}}}
\put(7126,-3436){\makebox(0,0)[lb]{\smash{{\SetFigFont{12}{14.4}{\familydefault}{\mddefault}{\updefault}{\color[rgb]{0,0,0}$\textcircled{4}$}%
}}}}
\put(6001,-2461){\makebox(0,0)[lb]{\smash{{\SetFigFont{12}{14.4}{\familydefault}{\mddefault}{\updefault}{\color[rgb]{0,0,0}$\textcircled{6}$}%
}}}}
\put(4726,-4861){\makebox(0,0)[lb]{\smash{{\SetFigFont{12}{14.4}{\familydefault}{\mddefault}{\updefault}{\color[rgb]{0,0,0}$\textcircled{0}$}%
}}}}
\end{picture}%

%% file: crossingbeta.pdf_t
\begin{picture}(0,0)%
\includegraphics{crossingbeta.pdf}%
\end{picture}%
\setlength{\unitlength}{3947sp}%
\begingroup\makeatletter\ifx\SetFigFont\undefined%
\gdef\SetFigFont#1#2#3#4#5{%
  \reset@font\fontsize{#1}{#2pt}%
  \fontfamily{#3}\fontseries{#4}\fontshape{#5}%
  \selectfont}%
\fi\endgroup%
\begin{picture}(2355,2895)(4486,-5446)
\put(5251,-4036){\makebox(0,0)[lb]{\smash{{\SetFigFont{12}{14.4}{\familydefault}{\mddefault}{\updefault}{\color[rgb]{0,0,0}$-\!\beta\!-\!1$}%
}}}}
\put(5101,-3136){\makebox(0,0)[lb]{\smash{{\SetFigFont{12}{14.4}{\familydefault}{\mddefault}{\updefault}{\color[rgb]{0,0,0}$i_1$}%
}}}}
\put(5101,-4861){\makebox(0,0)[lb]{\smash{{\SetFigFont{12}{14.4}{\familydefault}{\mddefault}{\updefault}{\color[rgb]{0,0,0}$j_2$}%
}}}}
\put(5626,-3586){\makebox(0,0)[lb]{\smash{{\SetFigFont{12}{14.4}{\familydefault}{\mddefault}{\updefault}{\color[rgb]{0,0,0}$\beta$}%
}}}}
\put(6301,-3811){\makebox(0,0)[lb]{\smash{{\SetFigFont{12}{14.4}{\familydefault}{\mddefault}{\updefault}{\color[rgb]{0,0,0}$\beta$}%
}}}}
\put(6826,-3961){\makebox(0,0)[lb]{\smash{{\SetFigFont{12}{14.4}{\familydefault}{\mddefault}{\updefault}{\color[rgb]{0,0,0}$\beta+1$}%
}}}}
\put(4501,-4036){\makebox(0,0)[lb]{\smash{{\SetFigFont{12}{14.4}{\familydefault}{\mddefault}{\updefault}{\color[rgb]{0,0,0}$\beta$}%
}}}}
\put(6751,-3136){\makebox(0,0)[lb]{\smash{{\SetFigFont{12}{14.4}{\familydefault}{\mddefault}{\updefault}{\color[rgb]{0,0,0}$i_2$}%
}}}}
\put(6826,-4861){\makebox(0,0)[lb]{\smash{{\SetFigFont{12}{14.4}{\familydefault}{\mddefault}{\updefault}{\color[rgb]{0,0,0}$j_1$}%
}}}}
\put(5851,-3136){\makebox(0,0)[lb]{\smash{{\SetFigFont{12}{14.4}{\familydefault}{\mddefault}{\updefault}{\color[rgb]{0,0,0}$0$}%
}}}}
\put(5851,-2686){\makebox(0,0)[lb]{\smash{{\SetFigFont{12}{14.4}{\familydefault}{\mddefault}{\updefault}{\color[rgb]{0,0,0}$-1$}%
}}}}
\put(5851,-4861){\makebox(0,0)[lb]{\smash{{\SetFigFont{12}{14.4}{\familydefault}{\mddefault}{\updefault}{\color[rgb]{0,0,0}$0$}%
}}}}
\put(5851,-5386){\makebox(0,0)[lb]{\smash{{\SetFigFont{12}{14.4}{\familydefault}{\mddefault}{\updefault}{\color[rgb]{0,0,0}$-1$}%
}}}}
\end{picture}%

%% file: athanalin.pdf_t
\begin{picture}(0,0)%
\includegraphics{athanalin.pdf}%
\end{picture}%
\setlength{\unitlength}{3947sp}%
\begingroup\makeatletter\ifx\SetFigFont\undefined%
\gdef\SetFigFont#1#2#3#4#5{%
  \reset@font\fontsize{#1}{#2pt}%
  \fontfamily{#3}\fontseries{#4}\fontshape{#5}%
  \selectfont}%
\fi\endgroup%
\begin{picture}(4674,594)(5239,-3796)
\put(9526,-3736){\makebox(0,0)[lb]{\smash{{\SetFigFont{12}{14.4}{\familydefault}{\mddefault}{\updefault}{\color[rgb]{0,0,0}$5$}%
}}}}
\put(6376,-3736){\makebox(0,0)[lb]{\smash{{\SetFigFont{12}{14.4}{\familydefault}{\mddefault}{\updefault}{\color[rgb]{0,0,0}$2$}%
}}}}
\put(5476,-3736){\makebox(0,0)[lb]{\smash{{\SetFigFont{12}{14.4}{\familydefault}{\mddefault}{\updefault}{\color[rgb]{0,0,0}$2$}%
}}}}
\put(5926,-3736){\makebox(0,0)[lb]{\smash{{\SetFigFont{12}{14.4}{\familydefault}{\mddefault}{\updefault}{\color[rgb]{0,0,0}$1$}%
}}}}
\put(6826,-3736){\makebox(0,0)[lb]{\smash{{\SetFigFont{12}{14.4}{\familydefault}{\mddefault}{\updefault}{\color[rgb]{0,0,0}$1$}%
}}}}
\put(7276,-3736){\makebox(0,0)[lb]{\smash{{\SetFigFont{12}{14.4}{\familydefault}{\mddefault}{\updefault}{\color[rgb]{0,0,0}$4$}%
}}}}
\put(7726,-3736){\makebox(0,0)[lb]{\smash{{\SetFigFont{12}{14.4}{\familydefault}{\mddefault}{\updefault}{\color[rgb]{0,0,0}$3$}%
}}}}
\put(8176,-3736){\makebox(0,0)[lb]{\smash{{\SetFigFont{12}{14.4}{\familydefault}{\mddefault}{\updefault}{\color[rgb]{0,0,0}$4$}%
}}}}
\put(8626,-3736){\makebox(0,0)[lb]{\smash{{\SetFigFont{12}{14.4}{\familydefault}{\mddefault}{\updefault}{\color[rgb]{0,0,0}$3$}%
}}}}
\end{picture}%

%% file: alintree.pdf_t
\begin{picture}(0,0)%
\includegraphics{alintree.pdf}%
\end{picture}%
\setlength{\unitlength}{3947sp}%
\begingroup\makeatletter\ifx\SetFigFont\undefined%
\gdef\SetFigFont#1#2#3#4#5{%
  \reset@font\fontsize{#1}{#2pt}%
  \fontfamily{#3}\fontseries{#4}\fontshape{#5}%
  \selectfont}%
\fi\endgroup%
\begin{picture}(5352,1924)(6211,-2975)
\put(9976,-2911){\makebox(0,0)[lb]{\smash{{\SetFigFont{12}{14.4}{\familydefault}{\mddefault}{\updefault}{\color[rgb]{0,0,0}$10$}%
}}}}
\put(8176,-1786){\makebox(0,0)[lb]{\smash{{\SetFigFont{12}{14.4}{\familydefault}{\mddefault}{\updefault}{\color[rgb]{0,0,0}$4_1$}%
}}}}
\put(8701,-1786){\makebox(0,0)[lb]{\smash{{\SetFigFont{12}{14.4}{\familydefault}{\mddefault}{\updefault}{\color[rgb]{0,0,0}$4_2$}%
}}}}
\put(7426,-1786){\makebox(0,0)[lb]{\smash{{\SetFigFont{12}{14.4}{\familydefault}{\mddefault}{\updefault}{\color[rgb]{0,0,0}$2_2$}%
}}}}
\put(6676,-2311){\makebox(0,0)[lb]{\smash{{\SetFigFont{12}{14.4}{\familydefault}{\mddefault}{\updefault}{\color[rgb]{0,0,0}$1_1$}%
}}}}
\put(7276,-2311){\makebox(0,0)[lb]{\smash{{\SetFigFont{12}{14.4}{\familydefault}{\mddefault}{\updefault}{\color[rgb]{0,0,0}$1_2$}%
}}}}
\put(7876,-1186){\makebox(0,0)[lb]{\smash{{\SetFigFont{12}{14.4}{\familydefault}{\mddefault}{\updefault}{\color[rgb]{0,0,0}$0$}%
}}}}
\put(6976,-1786){\makebox(0,0)[lb]{\smash{{\SetFigFont{12}{14.4}{\familydefault}{\mddefault}{\updefault}{\color[rgb]{0,0,0}$2_1$}%
}}}}
\put(9826,-1786){\makebox(0,0)[lb]{\smash{{\SetFigFont{12}{14.4}{\familydefault}{\mddefault}{\updefault}{\color[rgb]{0,0,0}$2$}%
}}}}
\put(10126,-2311){\makebox(0,0)[lb]{\smash{{\SetFigFont{12}{14.4}{\familydefault}{\mddefault}{\updefault}{\color[rgb]{0,0,0}$7$}%
}}}}
\put(10651,-1186){\makebox(0,0)[lb]{\smash{{\SetFigFont{12}{14.4}{\familydefault}{\mddefault}{\updefault}{\color[rgb]{0,0,0}$1$}%
}}}}
\put(11026,-1786){\makebox(0,0)[lb]{\smash{{\SetFigFont{12}{14.4}{\familydefault}{\mddefault}{\updefault}{\color[rgb]{0,0,0}$4$}%
}}}}
\put(11476,-1786){\makebox(0,0)[lb]{\smash{{\SetFigFont{12}{14.4}{\familydefault}{\mddefault}{\updefault}{\color[rgb]{0,0,0}$5$}%
}}}}
\put(10276,-1786){\makebox(0,0)[lb]{\smash{{\SetFigFont{12}{14.4}{\familydefault}{\mddefault}{\updefault}{\color[rgb]{0,0,0}$3$}%
}}}}
\put(9526,-2311){\makebox(0,0)[lb]{\smash{{\SetFigFont{12}{14.4}{\familydefault}{\mddefault}{\updefault}{\color[rgb]{0,0,0}$6$}%
}}}}
\put(6226,-2911){\makebox(0,0)[lb]{\smash{{\SetFigFont{12}{14.4}{\familydefault}{\mddefault}{\updefault}{\color[rgb]{0,0,0}$3_1$}%
}}}}
\put(6676,-2911){\makebox(0,0)[lb]{\smash{{\SetFigFont{12}{14.4}{\familydefault}{\mddefault}{\updefault}{\color[rgb]{0,0,0}$3_2$}%
}}}}
\put(7126,-2911){\makebox(0,0)[lb]{\smash{{\SetFigFont{12}{14.4}{\familydefault}{\mddefault}{\updefault}{\color[rgb]{0,0,0}$5$}%
}}}}
\put(9076,-2911){\makebox(0,0)[lb]{\smash{{\SetFigFont{12}{14.4}{\familydefault}{\mddefault}{\updefault}{\color[rgb]{0,0,0}$8$}%
}}}}
\put(9526,-2911){\makebox(0,0)[lb]{\smash{{\SetFigFont{12}{14.4}{\familydefault}{\mddefault}{\updefault}{\color[rgb]{0,0,0}$9$}%
}}}}
\end{picture}%

%% file: nonwall.pdf_t
\begin{picture}(0,0)%
\includegraphics{nonwall.pdf}%
\end{picture}%
\setlength{\unitlength}{3947sp}%
\begin{picture}(1530,1561)(1486,-3950)
\put(2626,-3136){\makebox(0,0)[lb]{\smash{\fontsize{12}{14.4}\normalfont {\color[rgb]{0,0,0}$0$}%
}}}
\put(1501,-3286){\makebox(0,0)[lb]{\smash{\fontsize{12}{14.4}\normalfont {\color[rgb]{0,0,0}$1$}%
}}}
\put(2251,-3286){\makebox(0,0)[lb]{\smash{\fontsize{12}{14.4}\normalfont {\color[rgb]{0,0,0}$2$}%
}}}
\put(3001,-3286){\makebox(0,0)[lb]{\smash{\fontsize{12}{14.4}\normalfont {\color[rgb]{0,0,0}$3$}%
}}}
\put(2251,-2536){\makebox(0,0)[lb]{\smash{\fontsize{12}{14.4}\normalfont {\color[rgb]{0,0,0}$-1$}%
}}}
\put(1801,-3886){\makebox(0,0)[lb]{\smash{\fontsize{12}{14.4}\normalfont {\color[rgb]{0,0,0}$-1$}%
}}}
\put(1876,-3136){\makebox(0,0)[lb]{\smash{\fontsize{12}{14.4}\normalfont {\color[rgb]{0,0,0}$0$}%
}}}
\end{picture}%

%% file: ish.pdf_t
\begin{picture}(0,0)%
\includegraphics{ish.pdf}%
\end{picture}%
\setlength{\unitlength}{3947sp}%
\begin{picture}(6252,4958)(889,-6127)
\put(2626,-2836){\makebox(0,0)[lb]{\smash{\fontsize{12}{14.4}\normalfont {\color[rgb]{0,0,0}$\omega=0\underline{0}1$}%
}}}
\put(7126,-4561){\makebox(0,0)[lb]{\smash{\fontsize{12}{14.4}\normalfont {\color[rgb]{0,0,0}$x_1-x_2=0$}%
}}}
\put(7051,-3436){\makebox(0,0)[lb]{\smash{\fontsize{12}{14.4}\normalfont {\color[rgb]{0,0,0}$x_1-x_2=1$}%
}}}
\put(3826,-6061){\makebox(0,0)[lb]{\smash{\fontsize{12}{14.4}\normalfont {\color[rgb]{0,0,0}$x_1-x_3=0$}%
}}}
\put(5026,-6061){\makebox(0,0)[lb]{\smash{\fontsize{12}{14.4}\normalfont {\color[rgb]{0,0,0}$x_1-x_3=1$}%
}}}
\put(6226,-6061){\makebox(0,0)[lb]{\smash{\fontsize{12}{14.4}\normalfont {\color[rgb]{0,0,0}$x_1-x_3=2$}%
}}}
\put(2626,-6061){\makebox(0,0)[lb]{\smash{\fontsize{12}{14.4}\normalfont {\color[rgb]{0,0,0}$x_2-x_3=0$}%
}}}
\put(4351,-3736){\makebox(0,0)[lb]{\smash{\fontsize{12}{14.4}\normalfont {\color[rgb]{0,0,0}$\sigma=123$}%
}}}
\put(3601,-4261){\makebox(0,0)[lb]{\smash{\fontsize{12}{14.4}\normalfont {\color[rgb]{0,0,0}$\sigma=123$}%
}}}
\put(3601,-4486){\makebox(0,0)[lb]{\smash{\fontsize{12}{14.4}\normalfont {\color[rgb]{0,0,0}$\omega=00\underline{0}$}%
}}}
\put(3601,-2386){\makebox(0,0)[lb]{\smash{\fontsize{12}{14.4}\normalfont {\color[rgb]{0,0,0}$\sigma=132$}%
}}}
\put(4276,-1336){\makebox(0,0)[lb]{\smash{\fontsize{12}{14.4}\normalfont {\color[rgb]{0,0,0}$\sigma=132$}%
}}}
\put(1501,-2611){\makebox(0,0)[lb]{\smash{\fontsize{12}{14.4}\normalfont {\color[rgb]{0,0,0}$\sigma=312$}%
}}}
\put(1501,-3811){\makebox(0,0)[lb]{\smash{\fontsize{12}{14.4}\normalfont {\color[rgb]{0,0,0}$\sigma=312$}%
}}}
\put(1576,-4936){\makebox(0,0)[lb]{\smash{\fontsize{12}{14.4}\normalfont {\color[rgb]{0,0,0}$\sigma=321$}%
}}}
\put(2926,-5461){\makebox(0,0)[lb]{\smash{\fontsize{12}{14.4}\normalfont {\color[rgb]{0,0,0}$\sigma=231$}%
}}}
\put(4951,-5086){\makebox(0,0)[lb]{\smash{\fontsize{12}{14.4}\normalfont {\color[rgb]{0,0,0}$\sigma=213$}%
}}}
\put(6151,-5086){\makebox(0,0)[lb]{\smash{\fontsize{12}{14.4}\normalfont {\color[rgb]{0,0,0}$\sigma=213$}%
}}}
\put(1501,-4036){\makebox(0,0)[lb]{\smash{\fontsize{12}{14.4}\normalfont {\color[rgb]{0,0,0}$\omega=000$}%
}}}
\put(1576,-5161){\makebox(0,0)[lb]{\smash{\fontsize{12}{14.4}\normalfont {\color[rgb]{0,0,0}$\omega=000$}%
}}}
\put(2926,-5686){\makebox(0,0)[lb]{\smash{\fontsize{12}{14.4}\normalfont {\color[rgb]{0,0,0}$\omega=000$}%
}}}
\put(3751,-5086){\makebox(0,0)[lb]{\smash{\fontsize{12}{14.4}\normalfont {\color[rgb]{0,0,0}$\sigma=213$}%
}}}
\put(3751,-5311){\makebox(0,0)[lb]{\smash{\fontsize{12}{14.4}\normalfont {\color[rgb]{0,0,0}$\omega=00\underline{0}$}%
}}}
\put(2926,-3811){\makebox(0,0)[lb]{\smash{\fontsize{12}{14.4}\normalfont {\color[rgb]{0,0,0}$\omega=00\underline{0}$}%
}}}
\put(2926,-3586){\makebox(0,0)[lb]{\smash{\fontsize{12}{14.4}\normalfont {\color[rgb]{0,0,0}$\sigma=132$}%
}}}
\put(1501,-2836){\makebox(0,0)[lb]{\smash{\fontsize{12}{14.4}\normalfont {\color[rgb]{0,0,0}$\omega=001$}%
}}}
\put(4951,-5311){\makebox(0,0)[lb]{\smash{\fontsize{12}{14.4}\normalfont {\color[rgb]{0,0,0}$\omega=00\underline{1}$}%
}}}
\put(2551,-2611){\makebox(0,0)[lb]{\smash{\fontsize{12}{14.4}\normalfont {\color[rgb]{0,0,0}$\sigma=132$}%
}}}
\put(6151,-5311){\makebox(0,0)[lb]{\smash{\fontsize{12}{14.4}\normalfont {\color[rgb]{0,0,0}$\omega=002$}%
}}}
\put(5551,-3736){\makebox(0,0)[lb]{\smash{\fontsize{12}{14.4}\normalfont {\color[rgb]{0,0,0}$\sigma=123$}%
}}}
\put(4126,-3286){\makebox(0,0)[lb]{\smash{\fontsize{12}{14.4}\normalfont {\color[rgb]{0,0,0}$\omega=01\underline{1}$}%
}}}
\put(4126,-3061){\makebox(0,0)[lb]{\smash{\fontsize{12}{14.4}\normalfont {\color[rgb]{0,0,0}$\sigma=123$}%
}}}
\put(3601,-2611){\makebox(0,0)[lb]{\smash{\fontsize{12}{14.4}\normalfont {\color[rgb]{0,0,0}$\omega=0\underline{1}1$}%
}}}
\put(5401,-3061){\makebox(0,0)[lb]{\smash{\fontsize{12}{14.4}\normalfont {\color[rgb]{0,0,0}$\sigma=123$}%
}}}
\put(5401,-3286){\makebox(0,0)[lb]{\smash{\fontsize{12}{14.4}\normalfont {\color[rgb]{0,0,0}$\omega=012$}%
}}}
\put(5551,-3961){\makebox(0,0)[lb]{\smash{\fontsize{12}{14.4}\normalfont {\color[rgb]{0,0,0}$\omega=0\underline{0}2$}%
}}}
\put(4276,-1561){\makebox(0,0)[lb]{\smash{\fontsize{12}{14.4}\normalfont {\color[rgb]{0,0,0}$\omega=021$}%
}}}
\put(4351,-3961){\makebox(0,0)[lb]{\smash{\fontsize{12}{14.4}\normalfont {\color[rgb]{0,0,0}$\omega=0\underline{01}$}%
}}}
\end{picture}%

%% file: alinfc.pdf_t
\begin{picture}(0,0)%
\includegraphics{alinfc.pdf}%
\end{picture}%
\setlength{\unitlength}{3947sp}%
\begingroup\makeatletter\ifx\SetFigFont\undefined%
\gdef\SetFigFont#1#2#3#4#5{%
  \reset@font\fontsize{#1}{#2pt}%
  \fontfamily{#3}\fontseries{#4}\fontshape{#5}%
  \selectfont}%
\fi\endgroup%
\begin{picture}(7002,894)(5161,-3796)
\put(9976,-3736){\makebox(0,0)[lb]{\smash{{\SetFigFont{12}{14.4}{\familydefault}{\mddefault}{\updefault}{\color[rgb]{0,0,0}$6$}%
}}}}
\put(5176,-3736){\makebox(0,0)[lb]{\smash{{\SetFigFont{12}{14.4}{\familydefault}{\mddefault}{\updefault}{\color[rgb]{0,0,0}$1$}%
}}}}
\put(6376,-3736){\makebox(0,0)[lb]{\smash{{\SetFigFont{12}{14.4}{\familydefault}{\mddefault}{\updefault}{\color[rgb]{0,0,0}$2$}%
}}}}
\put(7876,-3736){\makebox(0,0)[lb]{\smash{{\SetFigFont{12}{14.4}{\familydefault}{\mddefault}{\updefault}{\color[rgb]{0,0,0}$2$}%
}}}}
\put(8176,-3736){\makebox(0,0)[lb]{\smash{{\SetFigFont{12}{14.4}{\familydefault}{\mddefault}{\updefault}{\color[rgb]{0,0,0}$4$}%
}}}}
\put(5476,-3736){\makebox(0,0)[lb]{\smash{{\SetFigFont{12}{14.4}{\familydefault}{\mddefault}{\updefault}{\color[rgb]{0,0,0}$1$}%
}}}}
\put(5776,-3736){\makebox(0,0)[lb]{\smash{{\SetFigFont{12}{14.4}{\familydefault}{\mddefault}{\updefault}{\color[rgb]{0,0,0}$4$}%
}}}}
\put(6076,-3736){\makebox(0,0)[lb]{\smash{{\SetFigFont{12}{14.4}{\familydefault}{\mddefault}{\updefault}{\color[rgb]{0,0,0}$1$}%
}}}}
\put(6676,-3736){\makebox(0,0)[lb]{\smash{{\SetFigFont{12}{14.4}{\familydefault}{\mddefault}{\updefault}{\color[rgb]{0,0,0}$4$}%
}}}}
\put(6976,-3736){\makebox(0,0)[lb]{\smash{{\SetFigFont{12}{14.4}{\familydefault}{\mddefault}{\updefault}{\color[rgb]{0,0,0}$6$}%
}}}}
\put(7276,-3736){\makebox(0,0)[lb]{\smash{{\SetFigFont{12}{14.4}{\familydefault}{\mddefault}{\updefault}{\color[rgb]{0,0,0}$1$}%
}}}}
\put(7576,-3736){\makebox(0,0)[lb]{\smash{{\SetFigFont{12}{14.4}{\familydefault}{\mddefault}{\updefault}{\color[rgb]{0,0,0}$3$}%
}}}}
\put(8476,-3736){\makebox(0,0)[lb]{\smash{{\SetFigFont{12}{14.4}{\familydefault}{\mddefault}{\updefault}{\color[rgb]{0,0,0}$6$}%
}}}}
\put(8776,-3736){\makebox(0,0)[lb]{\smash{{\SetFigFont{12}{14.4}{\familydefault}{\mddefault}{\updefault}{\color[rgb]{0,0,0}$5$}%
}}}}
\put(9076,-3736){\makebox(0,0)[lb]{\smash{{\SetFigFont{12}{14.4}{\familydefault}{\mddefault}{\updefault}{\color[rgb]{0,0,0}$3$}%
}}}}
\put(9376,-3736){\makebox(0,0)[lb]{\smash{{\SetFigFont{12}{14.4}{\familydefault}{\mddefault}{\updefault}{\color[rgb]{0,0,0}$2$}%
}}}}
\put(9676,-3736){\makebox(0,0)[lb]{\smash{{\SetFigFont{12}{14.4}{\familydefault}{\mddefault}{\updefault}{\color[rgb]{0,0,0}$4$}%
}}}}
\put(10276,-3736){\makebox(0,0)[lb]{\smash{{\SetFigFont{12}{14.4}{\familydefault}{\mddefault}{\updefault}{\color[rgb]{0,0,0}$5$}%
}}}}
\put(10576,-3736){\makebox(0,0)[lb]{\smash{{\SetFigFont{12}{14.4}{\familydefault}{\mddefault}{\updefault}{\color[rgb]{0,0,0}$3$}%
}}}}
\put(10876,-3736){\makebox(0,0)[lb]{\smash{{\SetFigFont{12}{14.4}{\familydefault}{\mddefault}{\updefault}{\color[rgb]{0,0,0}$2$}%
}}}}
\put(11176,-3736){\makebox(0,0)[lb]{\smash{{\SetFigFont{12}{14.4}{\familydefault}{\mddefault}{\updefault}{\color[rgb]{0,0,0}$6$}%
}}}}
\put(11476,-3736){\makebox(0,0)[lb]{\smash{{\SetFigFont{12}{14.4}{\familydefault}{\mddefault}{\updefault}{\color[rgb]{0,0,0}$5$}%
}}}}
\put(11776,-3736){\makebox(0,0)[lb]{\smash{{\SetFigFont{12}{14.4}{\familydefault}{\mddefault}{\updefault}{\color[rgb]{0,0,0}$3$}%
}}}}
\put(12076,-3736){\makebox(0,0)[lb]{\smash{{\SetFigFont{12}{14.4}{\familydefault}{\mddefault}{\updefault}{\color[rgb]{0,0,0}$5$}%
}}}}
\end{picture}%

%% file: alinfcp.pdf_t
\begin{picture}(0,0)%
\includegraphics{alinfcp.pdf}%
\end{picture}%
\setlength{\unitlength}{3947sp}%
\begingroup\makeatletter\ifx\SetFigFont\undefined%
\gdef\SetFigFont#1#2#3#4#5{%
  \reset@font\fontsize{#1}{#2pt}%
  \fontfamily{#3}\fontseries{#4}\fontshape{#5}%
  \selectfont}%
\fi\endgroup%
\begin{picture}(4104,1254)(5209,-3553)
\put(8251,-2686){\makebox(0,0)[lb]{\smash{{\SetFigFont{12}{14.4}{\familydefault}{\mddefault}{\updefault}{\color[rgb]{0,0,0}$6$}%
}}}}
\put(5701,-2686){\makebox(0,0)[lb]{\smash{{\SetFigFont{12}{14.4}{\familydefault}{\mddefault}{\updefault}{\color[rgb]{0,0,0}$2$}%
}}}}
\put(5401,-2986){\makebox(0,0)[lb]{\smash{{\SetFigFont{12}{14.4}{\familydefault}{\mddefault}{\updefault}{\color[rgb]{0,0,0}$1$}%
}}}}
\put(7051,-2836){\makebox(0,0)[lb]{\smash{{\SetFigFont{12}{14.4}{\familydefault}{\mddefault}{\updefault}{\color[rgb]{0,0,0}$4$}%
}}}}
\put(7351,-2536){\makebox(0,0)[lb]{\smash{{\SetFigFont{12}{14.4}{\familydefault}{\mddefault}{\updefault}{\color[rgb]{0,0,0}$5$}%
}}}}
\put(6151,-2536){\makebox(0,0)[lb]{\smash{{\SetFigFont{12}{14.4}{\familydefault}{\mddefault}{\updefault}{\color[rgb]{0,0,0}$3$}%
}}}}
\end{picture}%

%% file: genlin-v5.bbl
\begin{thebibliography}{99}

\bibitem{Armstrong}
D.\ Armstrong,
Hyperplane arrangements and diagonal harmonics, preprint  (2010),
arXiv:1005.1949 [math.CO].

\bibitem{Armstrong-Reiner-Rhoades}
D.\ Armstrong, V.\ Reiner and B.\ Rhoades, 
Parking spaces, 
{\it Adv. Math.} {\bf 269} (2015), 647--706.  
  
\bibitem{Armstrong-Rhoades}
D.\ Armstrong, and B.\ Rhoades, 
The Shi arrangement and the Ish arrangement,
{\it Trans.\ Amer.\ Math.\ Soc.\ } {\bf 364} (2012), 1509--1528.  

\bibitem{Athanasiadis-charpol}
C.\ A.\ Athanasiadis, 
Characteristic polynomials of subspace arrangements and finite fields, 
{\it Adv.\ Math.\ } {\bf 122} (1996), 193--233.   

\bibitem{Athanasiadis-Linusson}
C.\ A.\  Athanasiadis and S.\ Linusson, 
A simple bijection for the regions of the Shi arrangement of hyperplanes,
{\it Discrete Math.\ } {\bf 204} (1999), 27--39. 

\bibitem{Beck-mixed}
M.\ Beck, A.\ Berrizbeitia, M.\ Dairyko, C.\ Rodriguez, A.\ Ruiz,
S.\ Veeneman, 
Parking functions, Shi arrangements, and mixed graphs, 
{\it Amer.\ Math.\ Monthly} {\bf 122} (2015), 660--673.

\bibitem{Benedetti-Knauer-Valencia}
C.\ Benedetti, K.\ Knauer and J.\ Valencia-Porras,
On lattice path matroid polytopes: alcoved triangulations and snake
decompositions, preprint 2023, arXiv:2303.10458 [math.CO]. 


\bibitem{Bernardi}
O.\ Bernardi, 
Deformations of the braid arrangement and trees,
{\it Adv.\ Math.\ } {\bf 335} (2018), 466--518.
  

\bibitem{Berthome-etc}
P.\ Berthom\'e, R.\ Cordovil, D.\ Forge, V.\ Ventos and T.\ Zaslavsky,
An elementary chromatic reduction for gain graphs and special 
hyperplane arrangements, 
{\it Electron.\  J.\  Combin.\ } {\bf 16} (2009), no.1, Research Paper
121, 31 pp. 

\bibitem{Carver}
W.\ B.\ Carver, 
Systems of linear inequalities.
{\it Ann.\ of Math.\ } {\bf(2) 23} (1922), 212--220.  

\bibitem{Dermenjian-Tzanaki}
A.\ Dermenjian and E,\ Tzanaki, 
Enumerating regions of Shi arrangements per Weyl cone, 
{\it European J.\ Combin.\ } {\bf 122} (2024), Paper No. 104002, 37 pp.  

\bibitem{Dorpalen-Stump}
  G.\ Dorpalen-Barry and C.\ Stump,
  Shi arrangements restricted to Weyl cones,
   	arXiv:2204.05829 [math.CO]. 

\bibitem{Drake}
B.\ Drake, An Inversion Theorem for Labeled Trees and Some Limits of
Areas under Lattice Paths, PhD thesis, Brandeis University, 2008.
  

\bibitem{Duarte-Oliveira}
R.\ Duarte, A.\ Guedes de Oliveira, 
Pak-Stanley labeling of the $m$-Catalan hyperplane arrangement,
{\it Adv.\ Math.\ } {\bf 387} (2021), Paper No. 107827, 17 pp.


\bibitem{Forge}
R.\ Fl\'orez and D.\ Forge, 
Activity from matroids to rooted trees and beyond, 
{\it J.\ Combin.\ Theory Ser.\ A} {\bf 198} (2023), Paper No. 105755, 21 pp.

\bibitem{Forge-Zaslavsky}
D.\ Forge and T.\ Zaslavsky, 
Lattice point counts for the Shi arrangement and other affinographic
hyperplane arrangements, 
{\it J.\ Combin.\ Theory Ser.\ A} {\bf 114} (2007), 97--109.  

\bibitem{Gallai}
T.\ Gallai, 
Maximum-minimum S\"atze \"uber Graphen, 
{\it Acta Math.\ Acad.\ Sci.\ Hungar.\ } {\bf 9} (1958), 395--434.

\bibitem{Gessel-Tewari}
I.\ M.\ Gessel, S.\ T.\ Griffin and V. Tewari, 
Labeled binary trees, subarrangements of the Catalan arrangements, and
Schur positivity,
{\it Adv.\ Math.\ } {\bf 356} (2019), 106814, 67 pp.  

\bibitem{Graham-Knuth-Patashnik}
R.\ Graham, D.\ E.\ Knuth and O.\ Patashnik,
``Concrete mathematics. A foundation for computer
           science, Second Edition,'' Addison--Wesley Publishing
           Company, Reading, MA, 1994. 

\bibitem{Hetyei-alta}
G.\ Hetyei, 
Alternation acyclic tournaments, 
{\it European J.\  Combin.\ } {\bf 81} (2019), 1--21.           

\bibitem{Hopkins-bigraphical}
S.\ Hopkins and D.\ Perkinson,
Bigraphical arrangements, 
{\it Trans.\ Amer.\ Math.\ Soc.\ } {\bf 368} (2016), 709--725.

\bibitem{Hopkins-orientations}
S.\ Hopkins and D.\ Perkinson,
Orientations, semiorders, arrangements, and parking functions, 
{\it Electron.\ J.\  Combin.\ } {\bf 19} (2012), no.4, Paper 8, 31 pp.  

\bibitem{Kalikow}
L.\ H.\ Kalikow, 
Symmetries in trees and parking functions, 
LaCIM 2000 Conference on Combinatorics, Computer Science and
Applications (Montreal, QC) 
{\it Discrete Math.\ } {\bf 256} (2002), 719--741.  

\bibitem{Korte-Vygens}
B.\ Korte and J.\ Vygen, 
Combinatorial optimization,
Theory and algorithms, Sixth edition 
{\it Algorithms Combin.\ }, {\bf 21}
Springer, Berlin, 2018. xxi+698 pp.

\bibitem{Lam-Postnikov-1}
T.\ Lam, and A.\ Postnikov, 
Alcoved polytopes. I, 
{\it Discrete Comput.\ Geom.\ } {\bf 38} (2007), 453--478.  

\bibitem{Lam-Postnikov-2}
T.\ Lam and A.\ Postnikov, 
Alcoved polytopes II, in: Lie groups, geometry, and representation
theory, 253--272,  {\it Progr. Math.,} {\bf 326}
Birkh\"auser/Springer, Cham, 2018

\bibitem{Mazin}
M.\ Mazin, 
Multigraph hyperplane arrangements and parking functions, 
{\it Ann.\ Comb.\ } {\bf 21} (2017), 653--661.

\bibitem{Postnikov-Stanley}
A.\ Postnikov and R.\ P.\ Stanley, 
Deformations of Coxeter hyperplane arrangements,
{\it J.\ Combin.\ Theory Ser.\ A} {\bf 91} (2000), 544--97. 

\bibitem{Roos}
  K.\ Roos,
 Linear Optimization: Theorems of the Alternative, In: Floudas, C.,
 Pardalos, P. (eds) Encyclopedia of Optimization. Springer, Boston,
 MA. \url{https://doi.org/10.1007/978-0-387-74759-0_334}   

\bibitem{Sagan}
B.\ E. Sagan, 
Combinatorics: the art of counting,
Grad.\ Stud.\ Math., {\bf 210}
American Mathematical Society, Providence, RI, [2020], \copyright
2020. xix+304 pp.   

\bibitem{Scott-Suppes}
D.\ Scott and P.\ Suppes, 
Foundational aspects of theories of measurement,
{\it J.\ Symbolic Logic} {\bf 23} (1958), 113--128.

\bibitem{Shi}
J.\ Y.\ Shi, 
Sign types corresponding to an affine Weyl group, 
{\it J.\ London Math.\ Soc.\ (2)} {\bf 35} (1987), 56--74.

\bibitem{Stanley-hit}
R.\ P.\ Stanley, 
Hyperplane arrangements, interval orders, and trees,
{\it Proc.\ Nat.\ Acad.\ Sci.\ U.S.A.} {\bf 93} (1996), 2620--2625. 

\bibitem{Stanley-tinv}
R.\ P.\  Stanley, 
Hyperplane arrangements, parking functions and tree inversions,
in: Mathematical essays in honor of Gian-Carlo Rota (Cambridge, MA,
1996), 359--375. {\bf Progr. Math., 161} 
Birkh\"auser Boston, Inc., Boston, MA, 1998

\bibitem{Stanley-PC}
R.\ P.\  Stanley, 
An introduction to hyperplane arrangements, Geometric combinatorics, 389--496.
IAS/Park City Math. Ser., {\bf 13}, 
American Mathematical Society, Providence, RI, 2007

\bibitem{Stanley-EC1}
R.\ P.\  Stanley, 
Enumerative combinatorics, Volume 1, 
Second edition,  Cambridge Stud.\ Adv.\ Math., {\bf 49}
Cambridge University Press, Cambridge, 2012. xiv+626 pp.
  
\bibitem{Tewari}
V.\  Tewari, 
Gessel polynomials, rooks, and extended Linial arrangements, 
{\it J.\ Combin.\ Theory Ser.\ A } {\bf 163} (2019), 98--117.

\bibitem{Whitney}  
H.\ Whitney, 
A logical expansion in mathematics,
{\it Bull.\ Amer.\ Math.\ Soc.\ } {\bf 38}(1932), 572--579.
  
\bibitem{Zaslavsky}
T.\ Zaslavsky,
Facing up to arrangements: face-count formulas for partitions of space
by hyperplanes,  
{\it Mem.\ Amer.\ Math.\ Soc.\ } {\bf 1} (1975), issue 1, no. 154, vii+102 pp. 

\bibitem{Zaslavsky-perpendicular}
T.\ Zaslavsky, 
Perpendicular dissections of space, 
{\it Discrete Comput.\ Geom.\ } {\bf 27} (2002), 303--351.

  
\end{thebibliography}
